\documentclass[11pt]{article}%
\usepackage{amsmath}
\usepackage{graphicx}
\usepackage{amsfonts}
\usepackage{amssymb}
\newtheorem{theorem}{Theorem}

\newtheorem{corollary}[theorem]{Corollary}

\newtheorem{lemma}[theorem]{Lemma}

\newtheorem{proposition}[theorem]{Proposition}
\newtheorem{remark}[theorem]{Remark}

\newenvironment{proof}[1][Proof]{\textbf{#1.} }{\ \rule{0.5em}{0.5em}}
\begin{document}

\title{Gaussian and non-Gaussian processes of zero power variation}
\author{\textsc{Francesco Russo} \thanks{ENSTA-ParisTech. Unit\'e de Math\'ematiques appliqu\'ees, 32, Boulevard Victor, F-75739 Paris Cedex 15 (France)}
\thanks{INRIA Rocquencourt Projet MathFi and Cermics Ecole des Ponts}
\textsc{and}\ \textsc{Frederi VIENS }\thanks{Department of Statistics, Purdue
University, 150 N. University St., West Lafayette, IN 47907-2067, USA } }
\maketitle

\begin{abstract}
This paper considers the class of stochastic processes $X$ defined on $[0,T]$
by \newline$X\left(  t\right)  =\int_{0}^{T}G\left(  t,s\right)  dM\left(
s\right)  $ where $M$ is a square-integrable martingale and $G$ is a
deterministic kernel. When $M$ is Brownian motion, $X$ is Gaussian, and the
class includes fractional Brownian motion and other Gaussian processes with or
without homogeneous increments. Let $m$ be an odd integer. Under the
assumption that the quadratic variation $\left[  M\right]  $ of $M$ is
differentiable with $\mathbf{E}\left[  \left\vert d\left[  M\right](t)
/dt\right\vert ^{m}\right]  $ finite, it is shown that the $m$th power
variation%
\[
\lim_{\varepsilon\rightarrow0}\varepsilon^{-1}\int_{0}^{T}ds\left(  X\left(
s+\varepsilon\right)  -X\left(  s\right)  \right)  ^{m}%
\]
exists and is zero when a quantity $\delta^{2}\left(  r\right)  $ related to
the variance of an increment of $M$ over a small interval of length $r$
satisfies $\delta\left(  r\right)  =o\left(  r^{1/(2m)}\right)  $.

In the case of a Gaussian process with homogeneous increments, $\delta$ is
$X$'s canonical metric, the condition on $\delta$ is proved to be necessary,
and the zero variation result is extended to non-integer symmetric powers,
i.e. using $\left\vert X\left(  s+\varepsilon\right)  -X\left(  s\right)
\right\vert ^{m}\mbox{sgn}\left(  X\left(  s+\varepsilon\right)  -X\left(
s\right)  \right)  $ for any real value $m\geq1$. In the non-homogeneous
Gaussian case, when $m=3$, the symmetric (generalized Stratonovich) integral
is defined, proved to exist, and its It\^{o} formula is proved to hold for all
functions of class $C^{6}$.

\end{abstract}


\bigskip

\textbf{MSC Classification 2000}: 60G07; 60G15; 60G48; 60H05.

\section{Introduction}

The purpose of this article is to study wide classes of processes with zero
cubic variation, and more generally, zero variation of any order. Before
summarizing our results, we give a brief historical description of the topic
of $p$-variations, as a basis for our motivations.

\subsection{Historical background}

The $p$-variation of a function $f:[0,T]\rightarrow\mathbf{R}$ is the supremum
over all the possible partitions $\{0=t_{0}<\ldots<t_{N}=T\}$ of $[0,T]$ of the
quantity
\begin{equation}
\sum_{i=1}^{N-1}|f(t_{i+1})-f({t_{i}})|^{p}.\label{p-var}%
\end{equation}
The analytic monograph \cite{bruneau} contains an interesting study on this
concept, showing that a $p$-variation function is the composition of an
increasing function and a H\"{o}lder-continuous function. The notion of
$p$-variation of a stochastic process or of a function was rediscovered in
stochastic analysis, particularly in the context of pathwise (or
quasi-pathwise) stochastic calculus. The fundamental paper \cite{foellmer},
due to H. F\"{o}llmer, treats the case of 2-variations. More recent dealings
with $p$-variations and their stochastic applications, particularly to rough
path and other integration techniques, are described at length for instance in
the books \cite{dudley-norvaisa} and \cite{lyons-qian} , which also contain
excellent bibliographies on the subject.

For $p=2$, the It\^{o} stochastic calculus for semimartingales has mimicked
the notion of 2-variation, with the notion of quadratic variation. Let $S$ be
a semimartingale; as in (\ref{p-var}), consider the expression%
\begin{equation}
\sum_{i=1}^{N-1}|S(t_{i+1})-S({t_{i}})|^{2}. \label{Q-var}%
\end{equation}
One defines the quadratic variation $[S]$ of $S$ as the limit in probability
of the expression in (\ref{Q-var}) as the partition mesh goes to $0$, instead
of considering the supremum over all partitions. Moreover, the notion becomes
stochastic. In fact for a standard Brownian motion $B$, its 2-variation
$\left[  B\right]  $ is a.s. infinite, but its quadratic variation is equal to
$T$. In order to reconcile $2$-variations with the finiteness of $[B]$, many
authors have proposed restricting the supremum in \eqref{p-var} to the dyadic
partitions. However, in It\^{o} calculus, the idea of quadratic variation is
associated with the notion of covariation (also known as joint quadratic
variation) $[S^{1},S^{2}]$ of two semimartingales $S^{1},S^{2}$, something
which is not present in analytic treatments of $p$-variation. This covariation
$[S^{1},S^{2}]$ is obtained by polarization of (\ref{Q-var}), i.e. is the
limit in probability of $\sum_{i=1}^{N-1}\left(  S^{1}(t_{i+1})-S^{1}({t_{i}%
})\right)  \left(  S^{2}(t_{i+1})-S^{2}({t_{i}})\right)  $ when, again, the
partition mesh goes to zero.

In the study of stochastic processes, the $p$-variation has been analyzed in
some specific cases, such as local time processes (see \cite{walsh}), iterated
Brownian motion, whose 4-th variation is finite, and more recently fractional
Brownian motion (fBm) and related processes.
To work with a general class of processes, the tools of It\^{o} calculus would
nonetheless restrict the study of covariation to semimartingales. In
\cite{RV}, the authors enlarged the notion of covariation to general processes
$X$ and $Y$. They achieved this by modifying the definition, considering
regularizations instead of discretizations. One starting observation is the
following. Let $f:[0,T]\rightarrow\mathbf{R}$ be continuous. This $f$ has
finite variation (i.e. it admits the 1-variation) if and only if
$\lim_{\varepsilon\rightarrow0}\frac{1}{\varepsilon}\int_{0}^{T}%
|f(s+\varepsilon)-f(s)|ds$ exists. In this case, the previous limit equals the
total variation of $f$. An objective was to produce a more efficient
stochastic calculus tool, able to go beyond the case of semimartingales. Given
two processes $X$ and $Y$, their covariation $[X,Y]\left(  t\right)  $ is the
limit in probability, when $\varepsilon$ goes to zero, of
\begin{equation}
\left[  X,Y\right]  _{\varepsilon}\left(  t\right)  =\frac{1}{\varepsilon}%
\int_{0}^{t}\big(X(s+\varepsilon)-X(s)\big)\big(Y(s+\varepsilon
)-Y(s)\big)ds;\quad t\geq0. \label{SIVR1Cn}%
\end{equation}
The limit is again denoted by $[X,Y]\left(  t\right)  $ and this notion
coincides with the classical covariation when $X,Y$ are continuous
semimartingales. The processes $X$ such that $[X,X]$ exists are called finite
quadratic variation processes; their analysis and their applications were
performed in \cite{flandoli-russo00, RV00}.

The notion of covariation was also extended in order to involve more
processes. In \cite{RE} the authors consider the $n$-covariation $[X^{1}%
,X^{2},\cdots,X^{n}]$ of $n$ processes $X^{1},\ldots,X^{n}$, as in formula
(\ref{SIVR1Cn}), but with a product of $n$ increments rather than just two.
For $n=4$, for $X$ being an fBm with so-called \emph{Hurst} parameter $H=1/4$,
the paper \cite{GRV} calculates the $4$-covariation $[g\left(  X\right)
,X,X,X]$ where $g$ is, for instance, a bounded continuous function. If
$X=X^{1}=X^{2}=X^{3}$ is a single stochastic process, we denote
$[X;3]:=[X,X,X]$, which is called the \emph{cubic variation}, and is one of
the main topics of investigation in our article. Note that this variation
involves the signed cubes $(X\left(  s+\varepsilon\right)  -X(s))^{3}$, which
has the same sign as the increment $X\left(  s+\varepsilon\right)  -X(s)$,
unlike the case of quadratic or $2$-variation, or of the so-called
\emph{strong} cubic variation, where absolute values are used inside the cube
function. Consider the case where $X$ is a fractional Brownian motion $B^{H}$
with Hurst parameter $H\in\left(  0,1\right)  $. Then when $H>1/6$, 
in the introduction of \cite{GNRV}
is established that the cubic variation $[X,3]$ equals zero, 
while in \cite[Theorem 4.1
part 2(c)]{GNRV} it is shown that $[X,3]$ does not exist if $H<1/6$. In the
limiting case of $H=1/6$, \cite[Theorem 4.1 part 2(b)]{GNRV} shows that the
regularization approximation $\left[  X,3\right]  _{\varepsilon}\left(
t\right)  $ converges in law to a normal for every $t>0$. This phenomenon is
confirmed in the related study of finite-difference approximating sequence of
$[X,3]$: for instance one may use the so-called Breuer-Major central limit
theorem for stationary Gaussian sequences \cite{BM}, to prove that the
finite-difference approximating sequence of $[X,3]\left(  t\right)  $
converges in law to a Gaussian variable; this was noted in \cite[Theorem
10]{NOT}, where the authors prove that more is true: considered as a process
depending on the upper endpoint of the time interval, the approximation
converges in law to $\kappa W$ where $W$ is an independent Brownian motion,
and $\kappa$ is a universal constant given by%
\[
\kappa^{2}=\frac{3}{4}\sum_{r\in\mathbb{Z}}(|r+1|^{\frac{1}{3}}+|r-1|^{\frac
{1}{3}}-2|r|^{\frac{1}{3}}).
\]

Beyond a basic interest in the variations of non-semimartingale stochastic
processes, the significance of the cubic variation lies in its ability to
guarantee the existence of (generalized symmetric) Stratonovich integrals, and
their associated It\^{o} formula,
for highly irregular processes, notably fBm with $H>1/6$ (such fBm have
H\"{o}lder regularity parameter exceeding $1/6$; compare with the near
$1/2$-H\"{o}lder-regularity for continuous semimartingales). This existence of
a quite general It\^{o}-Stratonovich formula is a relatively
 well-known phenomenon,
established in great generality in  \cite{GNRV}. This result is a
main motivation for our work, in which we attempt to give specific classes of
Gaussian and non-Gaussian processes with the said zero cubic variation
property; indeed the conditions in \cite{GNRV} are used therein specifically
with the Gaussian class of fBms, which are fractionally self-similar and have
stationary increments; the methods in \cite{GNRV} can be extended only to
similar cases, i.e. Gaussian processes with canonical metrics that are bounded
above and below by multiples of the fBm's, for instance the bi-fractional
Brownian motion treated in \cite{TR}. Regarding the existence of
It\^{o}-Stratonovich formulas, it is instructive to recall a variant on
\cite{GNRV} established in \cite{RE}: if $f:${$\mathbf{R}$}$\rightarrow
${$\mathbf{R}$} is a function of class $C^{3}$ and $X$ has a strong cubic
variation, then the following It\^{o} type formula holds:
\begin{equation}
f(X_{t})=f(X_{0})+\int_{0}^{t}f^{\prime}(X_{s})d^{\circ}X-\frac{1}{12}\int
_{0}^{t}f^{\prime\prime\prime}(X_{s})d[X,3]\left(  s\right)  ,\label{ItoER}%
\end{equation}
and the stochastic integral in the right-hand side is the
symmetric-Stratonovich integral introduced for instance in \cite{RV2}, while
the other is a Lebesgue-Stieltjes integral. The problem is that until now no
examples are known of processes $X$ which have a cubic variation $[X,3]$ which
exists but does not vanish. In \cite{NRS}, an analogous formula to
(\ref{ItoER}) is obtained for the case $X=B^{H}$ with $H=1/6$, but in the
sense of distributions only: the symmetric integral has to be interpreted as
existing in law and the integral with respect to the cubic variation makes
sense by replacing $[X,3]$ with the term $\kappa W$, $W$ being the independent
Wiener process identified in \cite{NOT}, so that $\int_{0}^{t}f^{\prime
\prime\prime}(X_{s})d[X,3]\left(  s\right)  $ is merely defined in law as a
conditionally Wiener integral.


\subsection{Specific motivations}

Following the regularization methodology of \cite{RV} or \cite{RV2}, the cubic
variation of a process $X$, denoted by $[X,3]\left(  t\right)  $, was defined
similarly to \cite{RE} as the limit in probability or in the mean square, as
$\varepsilon\rightarrow0$, of
\[
\lbrack X,3]_{\varepsilon}\left(  t\right)  :=\varepsilon^{-1}\int_{0}%
^{t}\left(  X\left(  s+\varepsilon\right)  -X\left(  s\right)  \right)
^{3}ds.
\]
This was already mentioned above. This $\left[  X,3\right]  $ will be null for
a deterministic function $X$ as long as it is $\alpha$-H\"{o}lder-continuous
with $\alpha>1/3$. But the main physical reason for being interested in this
cubic variation for random processes is that, because the cube function is
symmetric, if the process $X$ itself has some probabilistic symmetry as well
(such as the Gaussian property and the stationarity of increments), then we
can expect $[X,3]$ to be $0$ for much more irregular processes than those
which are almost-surely $\alpha$-H\"{o}lder-continuous with $\alpha>1/3$. As
mentioned above, \cite{RE} proves that fBm has zero cubic variation as soon as
$H>1/6$, in spite of the fact that fBm is only $\alpha$-H\"{o}lder-continuous
almost surely for all $\alpha<H$. This doubling improvement over the
deterministic situation is due exclusively to the random symmetries of fBm, as
they combine with the fact that the cube function is odd. Typically for other,
non-symmetric types of variations, $H$ needs to be larger to guarantee
existence of the variation, let alone nullity; for instance, when $X$ is fBm,
its strong cubic variation, defined as the limit in probability of
$\varepsilon^{-1}\int_{0}^{t}\left\vert X\left(  s+\varepsilon\right)
-X\left(  s\right)  \right\vert ^{3}ds$, exists for $H\geq1/3$ only.

Finally, some brief notes in the case where $X$ is fBm with $H=1/6$. We have
already observed that this threshold represents a critical value in terms of
existence of cubic variation, for fBm: we mentioned that whether in the sense
of regularization or of finite-difference, the approximating sequences of
$[X,3]\left(  t\right)  $ converge in law to Gaussian laws. On the other hand,
these normal convergences contrast with one further point in the study of
variations for fBm: in our article, we show as a preliminary result
(Proposition \ref{ChaosConv} herein), that $[X,3]_{\varepsilon}$ does not
converge in probability for $H=1/6$. The non-convergence of
$[X,3]_{\varepsilon}$ in probability for $H<1/6$ was known previously, as we
said above.

These properties of fBm beg the question of what occurs for other Gaussian
processes which may not be self-similar or even have stationary increments, or
even for non-Gaussian processes with similar $\alpha$-H\"{o}lder-continuous
paths, and to what extent the threshold $\alpha>1/6$ is sharp. Similarly, can
the odd symmetry of the cube function be generalized to any \textquotedblleft
symmetric\textquotedblright\ power function, i.e. $x\mapsto|x|^{m}%
\mathrm{sgn}(x)$ with arbitrary integer or non-integer $m>1$~? This refers to
what we will call the \textquotedblleft odd $m$th variation\textquotedblright,
defined (when it exists in the mean-square sense) by
\begin{equation}
\lbrack X,m]\left(  t\right)  :=\lim_{\varepsilon\rightarrow0}\varepsilon
^{-1}\int_{0}^{t}\left\vert X\left(  s+\varepsilon\right)  -X\left(  s\right)
\right\vert ^{m}\mathrm{sgn}\left(  X\left(  s+\varepsilon\right)  -X\left(
s\right)  \right)  ds. \label{mthodd}%
\end{equation}
The qualifier \textquotedblleft odd\textquotedblright\ above, when applied to
$m=3$, can easily yield the term \textquotedblleft odd cubic
variation\textquotedblright, which has historically been called simply
\textquotedblleft cubic variation\textquotedblright\ as we noted before, as
opposed to the \textquotedblleft strong cubic variation\textquotedblright%
\ which involves absolute values; therefore in this article, we will
systematically use the qualifier \textquotedblleft odd\textquotedblright\ for
all higher order $m$th variations based on odd functions, but will typically
omit it for the cubic variation.

\subsection{Summary of results}

This article provides answers to some of the above questions, both in Gaussian
and non-Gaussian settings, and we hope that it will stimulate work on
resolving some of the remaining open problems. Specifically, we consider the
process $X$ defined on $[0,T]$ by%
\begin{equation}
X\left(  t\right)  =\int_{0}^{T}G\left(  t,s\right)  dM\left(  s\right)
\label{generalmotors}%
\end{equation}
where $M$ is a square-integrable martingale on $[0,T]$, and $G$ is a
non-random measurable function on $[0,T]^{2}$, which is square-integrable in
$s$ with respect to $d\left[  M\right]  _{s}$ for every fixed $t$. In other
words, $X$ is defined using a Volterra representation with respect to a
square-integrable martingale. The quadratic variations of these
martingale-based convolutions was studied in \cite{errami-russoCRAS}.

What we call the Gaussian case is that in which $M$ is the standard Wiener
process (Brownian motion) $W$. The itemized list below is a summary of our
results. Here, for the reader's convenience, we have not spelled out the
technical conditions which are needed for some of our results, indicating
instead references to the precise theorem statements in the body of this
article. Some conditions become more restrictive as one move from simple
Gaussian cases to non-Gaussian cases. Yet we cover much wider classes of
processes than has been done in the past. The summary below also provides
indications of how wide a scope we reach, and has references to examples in
the main body of the paper.

One condition which appears in all cases, is essentially equivalent to
requiring that all processes $X$ that we consider are not more regular than
standard Brownian motion, i.e. are not $1/2$-H\"{o}lder-continuous. This
typically takes the form of a concavity condition on the process's squared
canonical metric $\delta^{2}\left(  s,t\right)  :=\mathbf{E}\left[  \left(
X\left(  t\right)  -X\left(  s\right)  \right)  ^{2}\right]  $. This condition
is not a restriction on the range of path regularity, since the main interest
of our results occurs around the H\"{o}lder exponent $1/6$, or more generally
the exponent $1/(2m)$ for any $m>1$: the processes with zero odd $m$th
variation appear as those which are better than $1/(2m)$-H\"{o}lder-continuous
in the $L^{2}\left(  \Omega\right)  $-sense. Processes which are better than
$1/2$-H\"{o}lder-continuous are not covered by this paper, but can be treated
using classical non-probabilistic tools such as the Young integral.

We now give a summary of all our results. For any number $m\geq2$, let
\[
\lbrack X,m]_{\varepsilon}\left(  t\right)  :=\frac{1}{\varepsilon}\int
_{0}^{t}ds\left\vert X\left(  s+\varepsilon\right)  -X\left(  s\right)
\right\vert ^{m}\mbox{sgn}\left(  X\left(  s+\varepsilon\right)  -X\left(
s\right)  \right)
\]
where $\mbox{sgn}\left(  x\right)  $ is the sign function $x/\left\vert
x\right\vert $. The limit in probability of $[X,m]_{\varepsilon}\left(
t\right)  $ as $\varepsilon\rightarrow0$ is the \textquotedblleft odd $m$th
variation\textquotedblright\ of $X$ at time $t$, denotd by $[X,m]\left(
t\right)  $. Except for some results in Section \ref{STOCH}, the results in
this paper are stated without loss of generality for a fixed value of $t\leq
T$, and we typically take $t=T$; we occasionally drop the dependence on $T$,
writing only $[X,m]_{\varepsilon}$ and $[X,m]$.

\begin{itemize}
\item {}[\textbf{Homogeneous Gaussian case, odd powers: }Theorem
\ref{HomogGauss} on page \pageref{HomogGauss}]. When $X$ is Gaussian with
homogeneous increments (meaning $\delta(s,t)$ depends only on $|t-s|$), for
any odd integer $m\geq3$, $X$ has zero odd $m$th variation if \emph{and only}
if $\delta\left(  r\right)  =o\left(  r^{1/\left(  2m\right)  }\right)  $ for
$r$ near $0$.

\begin{itemize}
\item This theorem does not require any assumptions beyond $\delta$ being
increasing and concave.
\end{itemize}

\item {}[\textbf{Homogeneous Gaussian case, arbitrary real powers: }Theorem
\ref{nope} on page \pageref{nope}]. The sufficient condition of the result
above holds for any integer $m>1,$ and for any real non-integer $>1$ modulo a
mild technical condition.

\begin{itemize}
\item This theorem extends the previous result to even powers without
requiring any additional assumptions, and to all real powers under a technical
regularity assumption on the covariance which places no regularity
restrictions on the paths of $X$ (see Remark \ref{noperem}).
\end{itemize}

\item {}[\textbf{Non-homogeneous Gaussian case: }Theorem \ref{nonhomoggauss}
on page \pageref{nonhomoggauss}]. When $X$ is Gaussian with non-homogeneous
increments, for any odd integer $m\geq3$, if $\delta^{2}\left(  s,s+r\right)
=o\left(  r^{1/(2m)}\right)  $ for $r$ near $0$ uniformly in $s$, and under a
technical condition, $X$ has zero odd $m$th variation.

\begin{itemize}
\item The technical condition is a non-explosion assumption on the mixed
partial derivative of $\delta^{2}$ near the diagonal. It places no regularity
restriction on the paths of $X$. The description on page \pageref{RL} shows
that the condition is satisfied for the so-called Riemann-Liouville version of
fBm, and for a wide class of Volterra-convolution-type Gaussian processes with
inhomogeneous increments.
\end{itemize}

\item {}[\textbf{Non-Gaussian martingale case: }Theorem \ref{MartThm} on page
\pageref{MartThm}]. Let $m\geq3$ be an odd integer. When $X$ is non-Gaussian
as in (\ref{generalmotors}), based on a martingale $M$ whose quadratic
variation process has a derivative with $2m$ moments (the actual condition on
$M$ in the theorem is weaker), let $\Gamma\left(  t\right)  =\left(
\mathbf{E}\left[  \left(  d\left[  M\right]  /dt\right)  ^{m}\right]  \right)
^{1/(2m)}$ and consider the Gaussian process%
\[
Z\left(  t\right)  =\int_{0}^{T}\Gamma\left(  s\right)  G\left(  t,s\right)
dW\left(  s\right)  .
\]
Under a technical integrability condition on planar increments of $\Gamma G$
near the diagonal, if $Z$ satisfies the conditions of Theorem \ref{HomogGauss}
or Theorem \ref{nonhomoggauss}, then $X$ has zero odd $m$th variation.

\begin{itemize}
\item Proposition \ref{Ex} on page \pageref{Ex} provides examples of wide
classes of martingales and kernels for which the assumptions of Theorem
\ref{MartThm} are satisfied. Details on how to construct these examples, and
how to evaluate their regularity properties, are given on page \ref{pex}.

\item A key consequence of Proposition \ref{Ex} and Theorem \ref{MartThm} is
that this paper's results extend from the Gaussian case to highly non-Gaussian
situations, insofar as, for $m$ an odd integer, it is easy to construct a
variety of martingales $M$ with no more than $m$ moments, which are comparable
to their Gaussian analogues in terms of path regularity, and for which the
corresponding $X$ in (\ref{defX}) has null odd $m$th variation. This is
explained on page \ref{pex}.

\item It is important to note that while the base process $M$ used here is a
martingale, the process $X$ in (\ref{defX}) whose variation we study is as far
from being a martingale as fBm is.
\end{itemize}

\item {}[\textbf{It\^{o} formula: }Theorem \ref{ForItoThm} on page
\pageref{ForItoThm}, and its corollary]. When $m\geq3$ is an odd integer and
$X$ is a Gaussian process with non-homogeneous increments such that
$\delta^{2}\left(  s,s+r\right)  =o\left(  r^{1/(2m)}\right)  $ uniformly in
$s$, under some additional technical conditions, for every bounded measurable
function $g$ on $\mathbf{R}$,
\[
\lim_{\varepsilon\rightarrow0}\frac{1}{\varepsilon^{2}}\mathbf{E}\left[
\left(  \int_{0}^{T}du\left(  X_{u+\varepsilon}-X_{u}\right)  ^{m}g\left(
\frac{X_{u+\varepsilon}+X_{u}}{2}\right)  \right)  ^{2}\right]  =0.
\]
If $m=3$, by results in \cite{GNRV}, Theorem \ref{ForItoThm} implies that for
any $f\in C^{6}\left(  \mathbf{R}\right)  $ and $t\in\lbrack0,T]$, the It\^{o}
formula $f\left(  X_{t}\right)  =f\left(  X_{0}\right)  +\int_{0}^{t}%
f^{\prime}\left(  X_{u}\right)  d^{\circ}X_u  $ holds, where the
integral is in the symmetric (generalized Stratonovich) sense. This formula is
in Corollary \ref{coroll} on page \pageref{coroll}.

\begin{itemize}
\item The scope of the technical conditions needed for the theorem and its
corollary is discussed immediately after the corollary. These conditions
include similar monotonicity and concavity conditions as are used in the
remainder of the article, plus some coercivity conditions ensuring that the
process $X$ is not too far from having homogeneous increments. The discussion
after Corollary \ref{coroll} establishes that the coercivity conditions are
satisfied in the homogeneous case.
\end{itemize}
\end{itemize}

\subsection{Relation with other recent work}

We finish this introduction with a description of recent work done by several
other authors on problems related to our preoccupations to some extent, in
various directions. The authors of the paper \cite{HNS} consider, as we do,
stochastic processes which can be written as Volterra integrals with respect
to martingales. In fact, they study the concept of \textquotedblleft
fractional martingale\textquotedblright, which is the generalization of the
so-called Riemann-Liouville fractional Brownian motion when the driving noise
is a martingale. This is a special case of the processes we consider in
Section \ref{NGC}, with $K\left(  t,s\right)  =\left(  t-s\right)  ^{H-1/2}$.
The authors' motivation is to prove an analogue of the famous characterization
of Brownian motion as the only continuous square-integrable martingale with a
quadratic variation equal to $t$. They provide similar necessary and
sufficient conditions based on the $1/H$-variation for a process to be
fractional Brownian motion. The paper \cite{HNS} does not follow, however, the
same motivation as our work: for us, say in the case of $m=3$, we study the
threshold $H>1/6$ for vanishing (odd) cubic variations in various Gaussian and
non-Gaussian contexts, and its relation to stochastic calculus.

To find a similar motivation to ours, one may look at the recent result of
\cite{NNT}, where the authors study the central and non-central behavior of
weighted Hermite variations for fBm. Using the Hermite polynomial of order $q$
rather than the power-$q$ function, they show that the threshold value
$H=1/\left(  2q\right)  $ poses an interesting open problem, since above this
threshold (but below $H=1-1/\left(  2q\right)  $) one obtains Gaussian limits
(these limits are conditionally Gaussian when weights are present, and can be
represented as stochastic integrals with respect to an independent Brownian
motion), while below the threshold, degeneracy occurs. The behavior at the
threshold was worked out for $H=1/4,q=2$ in \cite{NNT}, boasting an exotic
correction term with an independent Brownian motion, while the general open
problem of Hermite variations with $H=1/\left(  2q\right)  $ was settled in
\cite{NN}. More questions arise, for instance, with a similar result in
\cite{N2} for $H=1/4$, but this time with bidimensional fBm, in which two
independent Brownian motions are needed to characterize the exotic correction term.

The value $H=1/6$ is mentioned again in the context of the stochastic heat
equation driven by space-time white-noise, in which discrete trapezoidal sums
converge in distribution (not in probability) to a conditionally independent
Brownian motion: see \cite{BS} and \cite{NOT}.\bigskip

Summarizing, when compared to the works described in the above paragraphs, our
work situates itself by

\begin{itemize}
\item choosing to prove necessary and sufficient conditions for nullity of the
cubic variation, around the threshold regularity value $H=1/6$, for Gaussian
processes with homogeneous increments (this is a wider class than previously
considered, showing in particular that self-similarity is not related to the
question of nullity of the cubic variation);

\item studying the nullity threshold for higher order \textquotedblleft
odd\textquotedblright\ power functions, with possibly non-integer order,
showing that this property relies only on the symmetry of Gaussian processes
with homogeneous increments and on the symmetrization of the power functions;

\item showing that our method is able to consider processes that are far from
Gaussian and still yield sharp sufficient conditions for nullity of odd power
variations, since our base noise may be a generic martingale with only a few
moments.\bigskip
\end{itemize}

The article has the following structure. Section \ref{DEF} contains some
formal definitions and notations. The basic theorems in the Gaussian case are
in Section \ref{GAUSS}, where the homogeneous case, non-homogeneous case, and
case of non-integer $m$ are separated in three subsections. The use of
non-Gaussian martingales is treated in Section \ref{NGC}. Section \ref{STOCH}
presents the It\^{o} formula.

\section{Definitions\label{DEF}}

We recall our process $X$ defined for all $t\in\lbrack0,T]$ by%
\begin{equation}
X\left(  t\right)  =\int_{0}^{T}G\left(  t,s\right)  dM\left(  s\right)
\label{defX}%
\end{equation}
where $M$ is a square-integrable martingale on $[0,T]$, and $G$ is a
non-random measurable function on $[0,T]^{2}$, which is square-integrable in
$s$ with respect to $d\left[  M\right]  _{s}$ for every fixed $t$. For any
real number $m\geq2$, let the \emph{odd }$\varepsilon$\emph{-}$m$\emph{-th
variation} of $X$ be defined by%
\begin{equation}
\lbrack X,m]_{\varepsilon}\left(  T\right)  :=\frac{1}{\varepsilon}\int
_{0}^{T}ds\left\vert X\left(  s+\varepsilon\right)  -X\left(  s\right)
\right\vert ^{m}\mbox{sgn}\left(  X\left(  s+\varepsilon\right)  -X\left(
s\right)  \right)  . \label{defXm}%
\end{equation}
The odd variation is different from the absolute (or strong) variation because
of the presence of the sign function, making the function $\left\vert
x\right\vert ^{m}\mbox{sgn}\left(  x\right)  $ an odd function. In the sequel,
in order to lighten the notation, we will write $\left(  x\right)  ^{m}$ for
$\left\vert x\right\vert ^{m}\mbox{sgn}\left(  x\right)  $. We say that $X$
has zero odd $m$-th variation (in the mean-squared sense) if the limit%
\begin{equation}
\lim_{\varepsilon\rightarrow0}[X,m]_{\varepsilon}\left(  T\right)  =0
\label{limnulcub}%
\end{equation}
holds in $L^{2}\left(  \Omega\right)  $.

The \emph{canonical metric} $\delta$ of a stochastic process $X$ is defined as
the pseudo-metric on $[0,T]^{2}$ given by
\[
\delta^{2}\left(  s,t\right)  =\mathbf{E}\left[  \left(  X\left(  t\right)
-X\left(  s\right)  \right)  ^{2}\right]  .
\]
The \emph{covariance function} of $X$ is defined by%
\[
Q\left(  s,t\right)  =\mathbf{E}\left[  X\left(  t\right)  X\left(  s\right)
\right]  .
\]
The special case of a centered Gaussian process is of primary importance; then
the process's entire distribution is characterized by $Q$, or alternately by
$\delta$ and the variances $var\left(  X\left(  t\right)  \right)  =Q\left(
t,t\right)  $, since we have $Q\left(  s,t\right)  =\frac{1}{2}\left(
Q\left(  s,s\right)  +Q\left(  t,t\right)  -\delta^{2}\left(  s,t\right)
\right)  $. We say that $\delta$ has \emph{homogeneous increments} if there
exists a function on $[0,T]$ which we also denote by $\delta$ such that%
\[
\delta\left(  s,t\right)  =\delta\left(  \left\vert t-s\right\vert \right)  .
\]
Below, we will refer to this situation as the \emph{homogeneous case}. This is
in contrast to usual usage of this appellation, which is stronger, since for
example in the Gaussian case, it refers to the fact that $Q\left(  s,t\right)
$ depends only on the difference $s-t$; this would not apply to, say, standard
or fractional Brownian motion, while our definition does. In non-Gaussian
settings, the usual way to interpret the \textquotedblleft
homogeneous\textquotedblright\ property is to require that the processes
$X\left(  t+\cdot\right)  $ and $X\left(  \cdot\right)  $ have the same law,
which is typically much more restrictive than our definition.

The goal of the next two sections is to define various general conditions
under which a characterization of the limit in (\ref{limnulcub}) being zero
can be established. In particular, we aim to show that $X$ has zero odd $m$-th
variation for well-behaved $M$'s and $G$'s as soon as%
\begin{equation}
\delta\left(  s,t\right)  =o\left(  \left\vert t-s\right\vert ^{1/\left(
2m\right)  }\right)  , \label{condelta}%
\end{equation}
and that this is a necessary condition in some cases. Although this is a
mean-square condition, it can be interpreted as a regularity (local) condition
on $X$; for example, when $X$ is a Gaussian process with homogeneous
increments, this condition means precisely that almost surely, the uniform
modulus of continuity $\omega$ of $X$ on any fixed closed interval, defined by
$\omega\left(  r\right)  =\sup\left\{  \left\vert X\left(  t\right)  -X\left(
s\right)  \right\vert :\left\vert t-s\right\vert <r\right\}  $, satisfies
$\omega\left(  r\right)  =o\left(  r^{1/6}\log^{1/2}\left(  1/r\right)
\right)  $. The lecture notes \cite{A}, as well as the article \cite{TTV}, can
be consulted for this type of statement.

\section{Gaussian case\label{GAUSS}}

We assume that $X$ is centered Gaussian. Then we can write $X$ as in formula
(\ref{defX}) with $M=W$ a standard Brownian motion. More importantly,
beginning with the easiest case where $m$ is an odd integer, we can easily
show the following.

\begin{lemma}
\label{lemma1}If $m$ is an odd integer $\geq3$, we have%
\begin{align*}
&  \mathbf{E}\left[  \left(  [X,m]_{\varepsilon}\left(  T\right)  \right)
^{2}\right] \\
&  =\frac{1}{\varepsilon^{2}}\sum_{j=0}^{\left(  m-1\right)  /2}c_{j}\int
_{0}^{T}\int_{0}^{t}dtds\Theta^{\varepsilon}\left(  s,t\right)  ^{m-2j}%
\ Var\left[  X\left(  t+\varepsilon\right)  -X\left(  t\right)  \right]
^{j}\ Var\left[  X\left(  s+\varepsilon\right)  -X\left(  s\right)  \right]
^{j}\\
&  :=\sum_{j=0}^{\left(  m-1\right)  /2}J_{j}%
\end{align*}
where the $c_{j}$'s are constants depending only on $j$, and
\[
\Theta^{\varepsilon}\left(  s,t\right)  :=\mathbf{E}\left[  \left(  X\left(
t+\varepsilon\right)  -X\left(  t\right)  \right)  \left(  X\left(
s+\varepsilon\right)  -X\left(  s\right)  \right)  \right]  .
\]

\end{lemma}

\begin{proof}
The lemma is an easy consequence of the following formula, which can be found
as Lemma 5.2 in \cite{GNRV}: for any centered jointly Gaussian pair of r.v.'s
$\left(  Y,Z\right)  $, we have%
\[
\mathbf{E}\left[  Y^{m}Z^{m}\right]  =\sum_{j=0}^{\left(  m-1\right)  /2}%
c_{j}\mathbf{E}\left[  YZ\right]  ^{m-2j}\ Var\left[  X\right]  ^{j}%
\text{\ }Var\left[  Y\right]  ^{j}.
\]

\end{proof}

We may translate $\Theta^{\varepsilon}\left(  s,t\right)  $ immediately in
terms of $Q$, and then $\delta$. We have:%
\begin{align}
\Theta^{\varepsilon}\left(  s,t\right)   &  =Q\left(  t+\varepsilon
,s+\varepsilon\right)  -Q\left(  t,s+\varepsilon\right)  -Q\left(
s,t+\varepsilon\right)  +Q\left(  s,t\right) \nonumber\\
&  =\frac{1}{2}\left[  -\delta^{2}\left(  t+\varepsilon,s+\varepsilon\right)
+\delta^{2}\left(  t,s+\varepsilon\right)  +\delta^{2}\left(  s,t+\varepsilon
\right)  -\delta^{2}\left(  s,t\right)  \right] \label{seeplanar}\\
&  =:-\frac{1}{2}\Delta_{\left(  s,t\right)  ;\left(  s+\varepsilon
,t+\varepsilon\right)  }\delta^{2}. \label{defDelta}%
\end{align}
Thus $\Theta^{\varepsilon}\left(  s,t\right)  $ appears as the opposite of the
planar increment of the canonical metric over the rectangle defined by its
corners $\left(  s,t\right)  $ and $\left(  s+\varepsilon,t+\varepsilon
\right)  $.

\subsection{The case of fBm}

Before finding sufficient and possibly necessary conditions for various
Gaussian processes to have zero cubic (or $m$th) variation, we discuss the
threshold case for the cubic variation of fBm. Recall that when $X$ is fBm
with parameter $H=1/6$, as mentioned in the Introduction, it is known from
\cite[Theorem 4.1 part (2)]{GNRV} that $[X,3]_{\varepsilon}\left(  T\right)  $
converges in distribution to a non-degenerate normal law. However, there does
not seem to be any place in the literature specifying whether the convergence
may be any stronger than in distribution. We address this issue here.

\begin{proposition}
\label{ChaosConv}Let $X$ be an fBm with Hurst parameter $H=1/6$. Then $X$ does
not have a cubic variation (in the mean-square sense), by which we mean that
$[X,3]_{\varepsilon}\left(  T\right)  $ has no limit in $L^{2}\left(
\Omega\right)  $ as $\varepsilon\rightarrow0$. In fact more is true:
$[X,3]_{\varepsilon}\left(  T\right)  $ has no limit in probability as
$\varepsilon\rightarrow0$.
\end{proposition}

In order to prove the proposition, we study the Wiener chaos representation
and moments of $[X,3]_{\varepsilon}\left(  T\right)  $ when $X$ is fBm; $X$ is
given by (\ref{defX}) where $W$ is Brownian motion and the kernel $G$ is
well-known. Information on $G$ and on the Wiener chaos generated by $W$ can be
found respectively in Chapters 5 and 1 of the textbook \cite{Nbook}. The
covariance formula for an fBm $X$ is
\begin{equation}
R_{H}\left(  s,t\right)  :=\mathbf{E}\left[  X\left(  t\right)  X\left(
s\right)  \right]  =2^{-1}\left(  s^{2H}+t^{2H}-\left\vert t-s\right\vert
^{2H}\right)  . \label{fBmcov}%
\end{equation}

\begin{lemma}
\label{X3eExact}Fix $\varepsilon>0$. Let $\Delta X_{s}:=X\left(
s+\varepsilon\right)  -X\left(  s\right)  $ and $\Delta G_{s}\left(  u\right)
:=G\left(  s+\varepsilon,u\right)  -G\left(  s,u\right)  $. Then%
\begin{align}
\lbrack X,3]_{\varepsilon}\left(  T\right)   &  =\mathcal{I}_{1}%
+\mathcal{I}_{3}\nonumber\\
&  =:\frac{3}{\varepsilon}\int_{0}^{T}ds\int_{0}^{T}\Delta G_{s}\left(
u\right)  dW\left(  u\right)  \left(  \int_{0}^{T}\left\vert \Delta
G_{s}\left(  v\right)  \right\vert ^{2}dv\right) \label{firstchaos}\\
&  +\frac{6}{\varepsilon}\int_{0}^{T}dW\left(  s_{3}\right)  \int_{0}^{s_{3}%
}dW\left(  s_{2}\right)  \int_{0}^{s_{2}}dW\left(  s_{1}\right)  \int_{0}%
^{T}\left[  \prod_{k=1}^{3}\Delta G_{s}\left(  s_{k}\right)  \right]  ds.
\label{thirdchaos}%
\end{align}

\end{lemma}

\begin{proof}
The proof of this lemma is elementary. It follows from two uses of the
multiplication formula for Wiener integrals \cite[Proposition 1.1.3]{Nbook},
for instance. It can also be obtained directly from Lemma \ref{cubechaos}
below, or using the It\^{o} formula technique employed further below in
finding an expression for $[X,m]_{\varepsilon}\left(  T\right)  $ in Step 0 of
the proof of Theorem \ref{MartThm} on page \pageref{MartThm}. All details are
left to the reader.
\end{proof}

The above lemma indicates the Wiener chaos decomposition of
$[X,3]_{\varepsilon}\left(  T\right)  $ into the term $\mathcal{I}_{1}$ of
line (\ref{firstchaos}) which is in the first Wiener chaos (i.e. a Gaussian
term), and the term $\mathcal{I}_{3}$ of line (\ref{thirdchaos}), in the third
Wiener chaos. The next two lemmas contain information on the behavior of each
of these two terms, as needed to prove Proposition \ref{ChaosConv}.

\begin{lemma}
\label{I1pro}The Gaussian term $\mathcal{I}_{1}$ converges to $0$ in
$L^{2}\left(  \Omega\right)  $ as $\varepsilon\rightarrow0$.
\end{lemma}

\begin{lemma}
\label{I3pro}The 3rd chaos term $\mathcal{I}_{3}$ is bounded in $L^{2}\left(
\Omega\right)  $ for all $\varepsilon>0$, and does not converge in
$L^{2}\left(  \Omega\right)  $ as $\varepsilon\rightarrow0$.
\end{lemma}

\begin{proof}
[Proof of Proposition \ref{ChaosConv}]We prove the proposition by
contradiction. Assume $[X,3]_{\varepsilon}\left(  T\right)  $ converges in
probability. For any $p>2$, there exists $c_{p}$ depending only on $p$ such
that $\mathbf{E}\left[  \left\vert \mathcal{I}_{1}\right\vert ^{p}\right]
\leq c_{p}\left(  \mathbf{E}\left[  \left\vert \mathcal{I}_{1}\right\vert
^{2}\right]  \right)  ^{p/2}$ and $\mathbf{E}\left[  \left\vert \mathcal{I}%
_{3}\right\vert ^{p}\right]  \leq c_{p}\left(  \mathbf{E}\left[  \left\vert
\mathcal{I}_{3}\right\vert ^{2}\right]  \right)  ^{p/2}$; this is a general
fact about random variables in fixed Wiener chaos, and can be proved directly
using Lemma \ref{X3eExact} and the Burkh\"{o}lder-Davis-Gundy inequalities.
Therefore, since we have $\sup_{\varepsilon>0}(\mathbf{E}\left[  \left\vert
\mathcal{I}_{1}\right\vert ^{2}\right]  +\mathbf{E}\left[  \left\vert
\mathcal{I}_{3}\right\vert ^{2}\right]  )<\infty$ by Lemmas \ref{I1pro} and
\ref{I3pro}, we also get $\sup_{\varepsilon>0}(\mathbf{E}\left[  \left\vert
\mathcal{I}_{1}+\mathcal{I}_{3}\right\vert ^{p}\right]  )<\infty$ for any $p$.
Therefore, by uniform integrability, $[X,3]_{\varepsilon}\left(  T\right)
=\mathcal{I}_{1}+\mathcal{I}_{3}$ converges in $L^{2}\left(  \Omega\right)  $.
In $L^{2}\left(  \Omega\right)  $, the terms $\mathcal{I}_{1}$ and
$\mathcal{I}_{3}$ are orthogonal. Therefore, $\mathcal{I}_{1}$ and
$\mathcal{I}_{3}$ must converge in $L^{2}\left(  \Omega\right)  $ separately.
This contradicts the non-convergence of $\mathcal{I}_{3}$ in $L^{2}\left(
\Omega\right)  $ obtained in Lemma \ref{I3pro}. Thus $[X,3]_{\varepsilon
}\left(  T\right)  $ does not converge in probability.
\end{proof}

To conclude this section, we only need to prove the above two lemmas. To
improve readability, we write $H$ instead of $1/6$.

\begin{proof}
[Proof of Lemma \ref{I1pro}]Reintroducing the notation $X$ and $\Theta$ into
the formula in Lemma \ref{X3eExact}, we get%
\[
\mathcal{I}_{1}=\frac{3}{\varepsilon} \int_{0}^{T}ds
\left(  X\left(  s+\varepsilon\right)  -X\left(  s\right)  \right)  Var\left(
X\left(  s+\varepsilon\right)  -X\left(  s\right)  \right)
\]
and therefore,
\[
\mathbf{E}\left[  \left\vert \mathcal{I}_{1}\right\vert ^{2}\right]  =\frac
{9}{\varepsilon^{2}}\int_{0}^{T}\int_{0}^{t}dtds\Theta^{\varepsilon}\left(
s,t\right)  Var\left(  X\left(  t+\varepsilon\right)  -X\left(  t\right)
\right)  Var\left(  X\left(  s+\varepsilon\right)  -X\left(  s\right)
\right)
\]
We note here that $\mathbf{E}\left[  \left\vert \mathcal{I}_{1}\right\vert
^{2}\right]  $ coincides with what we called $J_{1}$ in Lemma \ref{lemma1},
but we will not use this fact here. Instead, using the variances of fBm,%
\begin{align*}
\mathbf{E}\left[  \left\vert \mathcal{I}_{1}\right\vert ^{2}\right]   &
=\frac{9}{2}\varepsilon^{-2+4H}\int_{0}^{T}\int_{0}^{T}dtds~Cov\left[
X\left(  t+\varepsilon\right)  -X\left(  t\right)  ;X\left(  s+\varepsilon
\right)  -X\left(  s\right)  \right] \\
&  =\frac{9}{2}\varepsilon^{-2+4H}~Var\left[  \int_{0}^{T}\left(  X\left(
t+\varepsilon\right)  -X\left(  t\right)  \right)  dt\right] \\
&  =\frac{9}{2}\varepsilon^{-2+4H}~Var\left[  \int_{T}^{T+\varepsilon}X\left(
t\right)  dt-\int_{0}^{\varepsilon}X\left(  t\right)  dt\right]  .
\end{align*}
Bounding the variance of the difference by twice the sum of the variances, and
using the fBm covariance formula (\ref{fBmcov}),%
\begin{align*}
\mathbf{E}\left[  \left\vert \mathcal{I}_{1}\right\vert ^{2}\right]   &
\leq9\varepsilon^{-2+4H}\left(  \int_{T}^{T+\varepsilon}\int_{T}%
^{T+\varepsilon}R_{H}\left(  s,t\right)  dsdt+\int_{0}^{\varepsilon}\int
_{0}^{\varepsilon}R_{H}\left(  s,t\right)  dsdt\right) \\
&  \leq9\varepsilon^{-2+4H}\left(  \varepsilon^{2}\left(  T+\varepsilon
\right)  ^{2H}+\varepsilon^{2+2H}\right)  =O\left(  \varepsilon^{4H}\right)  ,
\end{align*}
proving Lemma \ref{I1pro}.
\end{proof}

\begin{proof}
[Proof of Lemma \ref{I3pro}]By the proof of Lemma \ref{X3eExact}, and using
the covariance formula (\ref{fBmcov}) for fBm, we first get
\begin{align*}
\mathbf{E}\left[  \left\vert \mathcal{I}_{3}\right\vert ^{2}\right]   &
=\frac{12}{\varepsilon^{2}}\int_{0}^{T}\int_{0}^{t}dtds~\left(  \Theta
^{\varepsilon}\left(  s,t\right)  \right)  ^{3}\\
&  =\frac{6}{\varepsilon^{2}}\int_{0}^{T}\int_{0}^{t}dtds~\left(  \left\vert
t-s+\varepsilon\right\vert ^{2H}+\left\vert t-s-\varepsilon\right\vert
^{2H}-2\left\vert t-s\right\vert ^{2H}\right)  ^{3}.
\end{align*}
Again, this expression coincides with the term $J_{0}$ from Lemma
\ref{lemma1}, but this will not be used in this proof. We must take care of
the absolute values, i.e. of whether $\varepsilon$ is greater or less than
$t-s$. We define the \textquotedblleft off-diagonal\textquotedblright\ portion
of $\mathbf{E}\left[  \left\vert \mathcal{I}_{3}\right\vert ^{2}\right]  $ as
\[
\mathcal{ODI}_{3}:=6\varepsilon^{-2}\int_{2\varepsilon}^{T}\int_{0}%
^{t-2\varepsilon}dtds\left(  \left\vert t-s+\varepsilon\right\vert
^{2H}+\left\vert t-s-\varepsilon\right\vert ^{2H}-2\left\vert t-s\right\vert
^{2H}\right)  ^{3}.
\]
For $s,t$ in the integration domain for the above integral, since $\bar
{t}:=t-s>2\varepsilon$, by two iterated applications of the Mean Value Theorem
for the function $x^{2H}$ on the intervals $[\bar{t}-\varepsilon,\bar{t}]$ and
$[\bar{t},\bar{t}+\varepsilon]$,%
\[
\left\vert \bar{t}+\varepsilon\right\vert ^{2H}+\left\vert \bar{t}%
-\varepsilon\right\vert ^{2H}-2\bar{t}^{2H}=2H\left(  2H-1\right)
\varepsilon\left(  \xi_{1}-\xi_{2}\right)  \xi^{2H-2}%
\]
for some $\xi_{2}\in\lbrack\bar{t}-\varepsilon,\bar{t}],$ $\xi_{1}\in
\lbrack\bar{t},\bar{t}+\varepsilon]$, and $\xi\in\lbrack\xi_{1},\xi_{2}]$, and
therefore
\begin{align*}
\left\vert \mathcal{ODI}_{3}\right\vert  &  \leq384H^{3}\left\vert
2H-1\right\vert ^{3}\varepsilon^{-2}\int_{2\varepsilon}^{T}\int_{0}%
^{t-2\varepsilon}\left(  \varepsilon\cdot2\varepsilon\cdot\left(
t-s-\varepsilon\right)  ^{2H-2}\right)  ^{3}dtds\\
&  =384H^{3}\left\vert 2H-1\right\vert ^{3}\varepsilon^{4}\int_{2\varepsilon
}^{T}\int_{0}^{t-2\varepsilon}\left(  t-s-\varepsilon\right)  ^{6H-6}dtds\\
&  =\frac{384H^{3}\left\vert 2H-1\right\vert ^{3}\varepsilon^{4}}{5-6H}%
\int_{2\varepsilon}^{T}\left[  \varepsilon^{6H-5}-\left(  t-\varepsilon
\right)  ^{6H-5}\right]  dt\\
&  \leq\frac{384H^{3}\left\vert 2H-1\right\vert ^{3}}{5-6H}T\varepsilon
^{6H-1}=\frac{384H^{3}\left\vert 2H-1\right\vert ^{3}}{5-6H}T=\frac{32}{243}T.
\end{align*}
where in the last line we substituted $H=1/6$. Thus the \textquotedblleft
off-diagonal\textquotedblright\ term is bounded. The diagonal part of
$\mathcal{I}_{3}$ is%
\begin{align*}
\mathcal{DI}_{3}  &  :=6\varepsilon^{-2}\int_{0}^{T}\int_{t-2\varepsilon}%
^{t}dtds\left(  \left\vert t-s+\varepsilon\right\vert ^{2H}+\left\vert
t-s-\varepsilon\right\vert ^{2H}-2\left\vert t-s\right\vert ^{2H}\right)
^{3}\\
&  =6\varepsilon^{-2}T\int_{0}^{2\varepsilon}d\bar{t}\left(  \left\vert
\bar{t}+\varepsilon\right\vert ^{2H}+\left\vert \bar{t}-\varepsilon\right\vert
^{2H}-2\left\vert \bar{t}\right\vert ^{2H}\right)  ^{3}\\
&  =6\varepsilon^{-1+6H}T\int_{0}^{2}dr\left(  \left\vert r+1\right\vert
^{2H}+\left\vert r-1\right\vert ^{2H}-2\left\vert r\right\vert ^{2H}\right)
^{3}dr=CT
\end{align*}
where, having substituted $H=1/6$, yields that $C$ is a universal constant.
Thus the diagonal part $\mathcal{DI}_{3}$ of $\mathbf{E}[\mathbf{|}%
\mathcal{I}_{3}|^{2}]$ is constant. This proves that $\mathcal{I}_{3}$ is
bounded in $L^{2}\left(  \Omega\right)  $, as announced. To conclude that it
cannot converge in $L^{2}\left(  \Omega\right)  $, recall that from
\cite[Theorem 4.1 part (2)]{GNRV}, $[X,3]_{\varepsilon}\left(  T\right)
=\mathcal{I}_{1}+\mathcal{I}_{3}$ converges in distribution to a
non-degenerate normal law. By Lemma \ref{I1pro}, $\mathcal{I}_{1}$ converges
to $0$ in $L^{2}\left(  \Omega\right)  $. Therefore, $\mathcal{I}_{3}$
converges in distribution to a non-degenerate normal law; if it also converged
in $L^{2}\left(  \Omega\right)  $, since the 3rd Wiener chaos is closed in
$L^{2}\left(  \Omega\right)  $, the limit would have to be in that same chaos,
and thus would not have a non-degenerate normal law. This concludes the proof
of Lemma \ref{I3pro}.
\end{proof}

\subsection{The homogeneous case\label{HomogGaussSect}}

We now study the homogeneous case in detail. We are ready to prove a necessary
and sufficient condition for having a zero $m$-th variation when $m$ is an odd integer.

\begin{theorem}
\label{HomogGauss}Let $m>1$ be an odd integer. Let $X$ be a centered Gaussian
process on $[0,T]$ with homogeneous increments; its canonical metric is%
\[
\delta^{2}\left(  s,t\right)  :=\mathbf{E}\left[  \left(  X\left(  t\right)
-X\left(  s\right)  \right)  ^{2}\right]  =\delta^{2}\left(  \left\vert
t-s\right\vert \right)
\]
where the univariate function $\delta^{2}$ is assumed to be increasing and
concave on $[0,T]$. Then $X$ has zero $m$th variation if and only if
$\delta\left(  r\right)  =o\left(  r^{1/\left(  2m\right)  }\right)  $.
\end{theorem}

\begin{proof}
\noindent\emph{Step 0: setup.} We denote by $d\delta^{2}$ the derivative, in
the sense of measures, of $\delta^{2}$; we know that $d\delta^{2}$ is a
positive bounded measure on $[0,T]$. Using homogeneity, we also get
\[
Var\left[  X\left(  t+\varepsilon\right)  -X\left(  t\right)  \right]
=\delta^{2}\left(  \varepsilon\right)  .
\]
Using the notation in Lemma \ref{lemma1}, we get%
\[
J_{j}=\varepsilon^{-2}\delta^{4j}\left(  \varepsilon\right)  c_{j}\int_{0}%
^{T}dt\int_{0}^{t}ds\Theta^{\varepsilon}\left(  s,t\right)  ^{m-2j}.
\]
\vspace{0.1in}

\noindent\emph{Step 1: diagonal. }Let us deal first with the diagonal term. We
define the $\varepsilon$-diagonal $D_{\varepsilon}:=\left\{  0\leq
t-\varepsilon<s<t\leq T\right\}  $. Trivially using Cauchy -Schwarz's
inequality, we have
\[
\left\vert \Theta^{\varepsilon}\left(  s,t\right)  \right\vert \leq
\sqrt{Var\left[  X\left(  t+\varepsilon\right)  -X\left(  t\right)  \right]
Var\left[  X\left(  s+\varepsilon\right)  -X\left(  s\right)  \right]
}=\delta^{2}\left(  \varepsilon\right)  .
\]
Hence, according to Lemma \ref{lemma1}, the diagonal portion $\sum
_{j=0}^{\left(  m-1\right)  /2}J_{j,D_{\varepsilon}}$ of $\mathbf{E}\left[
\left(  [X,m]_{\varepsilon}\left(  T\right)  \right)  ^{2}\right]  $ can be
bounded above, in absolute value, as:%
\begin{align*}
\sum_{j=0}^{\left(  m-1\right)  /2}J_{j,D_{\varepsilon}}  &  :=\sum
_{j=0}^{\left(  m-1\right)  /2}\varepsilon^{-2}\delta^{4j}\left(
\varepsilon\right)  c_{j}\int_{\varepsilon}^{T}dt\int_{t-\varepsilon}%
^{t}ds\Theta^{\varepsilon}\left(  s,t\right)  ^{m-2j}.\\
&  \leq\frac{1}{\varepsilon^{2}}\sum_{j=0}^{\left(  m-1\right)  /2}c_{j}%
\int_{\varepsilon}^{T}dt\int_{t-\varepsilon}^{t}ds\delta^{2m}\left(
\varepsilon\right)  \leq c\cdot\varepsilon^{-1}\delta^{2m}\left(
\varepsilon\right)
\end{align*}
where $cst$ denotes a constant whose value may change in the remainder of the
article's proofs (here it depends only on $\delta$ and $m$). The hypothesis on
$\delta^{2}$ implies that the above converges to $0$ as $\varepsilon$ tends to
$0$.\vspace{0.1in}

\noindent\emph{Step 2: small }$t$\emph{ term . }The term for $t\in
\lbrack0,\varepsilon]$ and any $s\in\lbrack0,t]$ can be dealt with similarly,
and is of a smaller order than the one in Step 1. Specifically we have%
\[
\left\vert J_{j,S}\right\vert :=\varepsilon^{-2}\delta^{4j}\left(
\varepsilon\right)  c_{j}\left\vert \int_{0}^{\varepsilon}dt\int_{0}%
^{t}ds\Theta^{\varepsilon}\left(  s,t\right)  ^{m-2j}\right\vert
\leq\varepsilon^{-2}\delta^{4j}\left(  \varepsilon\right)  c_{j}%
\delta^{2\left(  m-2j\right)  }\left(  \varepsilon\right)  \varepsilon
^{2}=c_{j}\delta^{2m}\left(  \varepsilon\right)  ,
\]
which converges to $0$ like $o\left(  \varepsilon\right)  $.\vspace{0.1in}

\noindent\emph{Step 3: off-diagonal. }Because of the homogeneity hypothesis,
we can calculate from (\ref{seeplanar}) that for any $s,t$ in the
$\varepsilon$-off diagonal set $OD_{\varepsilon}:=\left\{  0\leq
s<t-\varepsilon<t\leq T\right\}  $%
\begin{align}
\Theta^{\varepsilon}\left(  s,t\right)   &  =\left(  \delta^{2}\left(
t-s+\varepsilon\right)  -\delta^{2}\left(  t-s\right)  \right)  -\left(
\delta^{2}\left(  t-s\right)  -\delta^{2}\left(  t-s-\varepsilon\right)
\right) \nonumber\\
&  =\int_{t-s}^{t-s+\varepsilon}d\delta^{2}\left(  r\right)  -\int
_{t-s-\varepsilon}^{t-s}d\delta^{2}\left(  r\right)  . \label{Thetadiffdelta}%
\end{align}
By the concavity hypothesis, we see that $\Theta^{\varepsilon}\left(
s,t\right)  $ is negative in this off-diagonal set $OD_{\varepsilon}$.
Unfortunately, using the notation in Lemma \ref{lemma1}, this negativity does
not help us because the off-diagonal portion $J_{j,OD}$ of $J_{j}$ also
involves the constant $c_{j}$, which could itself be negative. Hence we need
to estimate $J_{j,OD}$ more precisely.

The constancy of the sign of $\Theta^{\varepsilon}$ is still useful, because
it enables our first operation in this step, which is to reduce the estimation
of $\left\vert J_{j,OD}\right\vert $ to the case of $j=\left(  m-1\right)
/2$. Indeed, using Cauchy-Schwarz's inequality and the fact that $\left\vert
\Theta^{\varepsilon}\right\vert =-\Theta^{\varepsilon}$, we write%
\begin{align*}
\left\vert J_{j,OD}\right\vert  &  =\varepsilon^{-2}\delta^{4j}\left(
\varepsilon\right)  \left\vert c_{j}\right\vert \int_{\varepsilon}^{T}%
dt\int_{0}^{t-\varepsilon}ds\left\vert \Theta^{\varepsilon}\left(  s,t\right)
\right\vert ^{m-2j}\\
&  =-\varepsilon^{-2}\delta^{4j}\left(  \varepsilon\right)  \left\vert
c_{j}\right\vert \int_{\varepsilon}^{T}dt\int_{0}^{t-\varepsilon}%
ds\Theta^{\varepsilon}\left(  s,t\right)  \left\vert \Theta^{\varepsilon
}\left(  s,t\right)  \right\vert ^{m-2j-1}\\
&  \leq\varepsilon^{-2}\delta^{4j}\left(  \varepsilon\right)  \left\vert
c_{j}\right\vert \int_{\varepsilon}^{T}dt\int_{0}^{t-\varepsilon}ds\left(
-\Theta^{\varepsilon}\left(  s,t\right)  \right)  \left\vert \delta^{2}\left(
\varepsilon\right)  \right\vert ^{m-2j-1}\\
&  =\varepsilon^{-2}\delta^{2m-2}\left(  \varepsilon\right)  \left\vert
c_{j}\right\vert \int_{\varepsilon}^{T}dt\int_{0}^{t-\varepsilon}ds\left(
-\Theta^{\varepsilon}\left(  s,t\right)  \right)  .
\end{align*}
It is now sufficient to show that the estimate for the case $j=\left(
m-1\right)  /2$ holds, i.e. that
\begin{equation}
\int_{\varepsilon}^{T}dt\int_{0}^{t-\varepsilon}ds\left(  -\Theta
^{\varepsilon}\left(  s,t\right)  \right)  \leq cst\cdot\varepsilon\delta
^{2}\left(  2\varepsilon\right)  \label{claimstep3}%
\end{equation}

We rewrite the planar increments of $\delta^{2}$ as in (\ref{Thetadiffdelta})
to show what cancellations occur: with the notation $s^{\prime}=t-s$,%
\[
-\Theta^{\varepsilon}\left(  s,t\right)  =-\left(  \delta^{2}\left(
s^{\prime}+\varepsilon\right)  -\delta^{2}\left(  s^{\prime}\right)  \right)
+\left(  \delta^{2}\left(  s^{\prime}\right)  -\delta^{2}\left(  s^{\prime
}-\varepsilon\right)  \right)  =-\int_{s^{\prime}}^{s^{\prime}+\varepsilon
}d\delta^{2}\left(  r\right)  +\int_{s^{\prime}-\varepsilon}^{s^{\prime}%
}d\delta^{2}\left(  r\right)  .
\]
Therefore, using the change of variables from $s$ to $s^{\prime}$, and another
to change $\left[  s^{\prime}-\varepsilon,s^{\prime}\right]  $ to $\left[
s^{\prime},s^{\prime}+\varepsilon\right]  $,%
\begin{align}
\int_{\varepsilon}^{T}dt\int_{0}^{t-\varepsilon}ds\left(  -\Theta
^{\varepsilon}\left(  s,t\right)  \right)   &  =\int_{\varepsilon}%
^{T}dt\left[  \int_{\varepsilon}^{t}ds^{\prime}\int_{s^{\prime}-\varepsilon
}^{s^{\prime}}d\delta^{2}\left(  r\right)  -\int_{\varepsilon}^{t}ds^{\prime
}\int_{s^{\prime}}^{s^{\prime}+\varepsilon}d\delta^{2}\left(  r\right)
\right] \nonumber\\
&  =\int_{\varepsilon}^{T}dt\left[  \int_{\varepsilon}^{t}ds^{\prime}%
\int_{s^{\prime}-\varepsilon}^{s^{\prime}}d\delta^{2}\left(  r\right)
-\int_{\varepsilon}^{t}ds^{\prime}\int_{s^{\prime}}^{s^{\prime}+\varepsilon
}d\delta^{2}\left(  r\right)  \right] \nonumber\\
&  =\int_{\varepsilon}^{T}dt\left[  \int_{0}^{t-\varepsilon}ds^{\prime\prime
}\int_{s^{\prime\prime}}^{s^{\prime\prime}+\varepsilon}d\delta^{2}\left(
r\right)  -\int_{\varepsilon}^{t}ds^{\prime}\int_{s^{\prime}}^{s^{\prime
}+\varepsilon}d\delta^{2}\left(  r\right)  \right] \nonumber\\
&  =\int_{\varepsilon}^{T}dt\left[  \int_{0}^{\varepsilon}ds^{\prime\prime
}\int_{s^{\prime\prime}}^{s^{\prime\prime}+\varepsilon}d\delta^{2}\left(
r\right)  -\int_{t-\varepsilon}^{t}ds^{\prime}\int_{s^{\prime}}^{s^{\prime
}+\varepsilon}d\delta^{2}\left(  r\right)  \right]  \label{usein43}%
\end{align}
We may now invoke the positivity of $d\delta^{2}$, to obtain%
\begin{align*}
&  \int_{\varepsilon}^{T}dt\int_{0}^{t-\varepsilon}ds\left(  -\Theta
^{\varepsilon}\left(  s,t\right)  \right)  \leq\int_{\varepsilon}^{T}%
dt\int_{0}^{\varepsilon}ds^{\prime\prime}\int_{s^{\prime\prime}}%
^{s^{\prime\prime}+\varepsilon}d\delta^{2}\left(  r\right) \\
&  =\int_{\varepsilon}^{T}dt\int_{0}^{\varepsilon}ds^{\prime\prime}\left(
\delta^{2}\left(  s^{\prime\prime}+\varepsilon\right)  -\delta^{2}\left(
s^{\prime\prime}\right)  \right)  \leq\int_{\varepsilon}^{T}dt\ \varepsilon
\ \delta^{2}\left(  2\varepsilon\right)  \leq T\varepsilon\delta^{2}\left(
2\varepsilon\right)  .
\end{align*}
This is precisely the claim in (\ref{claimstep3}), which finishes the proof
that for all $j$, $\left\vert J_{j,OD}\right\vert \leq cst\cdot\varepsilon
^{-1}\delta^{2m}\left(  2\varepsilon\right)  $ for some constant $cst$.
Combining this with the results of Steps 1 and 2, we obtain that%
\[
\mathbf{E}\left[  \left(  [X,m]_{\varepsilon}\left(  T\right)  \right)
^{2}\right]  \leq cst\cdot\varepsilon^{-1}\delta^{2m}\left(  2\varepsilon
\right)
\]
which implies the sufficient condition in the theorem.\vspace{0.1in}

\noindent\emph{Step 4: necessary condition.} The proof of this part is more
delicate than the above: it requires an excellent control of the off-diagonal
term, since it is negative and turns out to be of the same order of magnitude
as the diagonal term. We spell out the proof here for $m=3$. The general case
is similar, and is left to the reader.\vspace{0.1in}

\noindent\emph{Step 4.1: positive representation.} The next lemma uses the
following chaos integral notation: for any $n\in\mathbf{N}$, for $g\in
L^{2}\left(  [0,T]^{n}\right)  $, $g$ symmetric in its $n$ variables, then
$I_{n}\left(  g\right)  $ is the multiple Wiener integral of $g$ over
$[0,T]^{n}$ with respect to $W$. This lemma's elementary proof is left to the reader.

\begin{lemma}
\label{cubechaos}Let $f\in L^{2}\left(  [0,T]\right)  $. Then $I_{1}\left(
f\right)  ^{3}=3\left\vert f\right\vert _{L^{2}\left(  [0,T]\right)  }%
^{2}I_{1}\left(  f\right)  +I_{3}\left(  f\otimes f\otimes f\right)  $
\end{lemma}

Using this lemma, as well as definitions (\ref{defX}) and (\ref{defXm}),
recalling the notation $\Delta G_{s}\left(  u\right)  :=G\left(
s+\varepsilon,u\right)  -G\left(  s,u\right)  $ already used in Lemma
\ref{X3eExact}, and exploiting the fact that the covariance of two multiple
Wiener integrals of different orders is $0$, we can write
\begin{align*}
&  \mathbf{E}\left[  \left(  [X,3]_{\varepsilon}\left(  T\right)  \right)
^{2}\right]  =\frac{1}{\varepsilon^{2}}\int_{0}^{T}ds\int_{0}^{T}%
dt\mathbf{E}\left[  \left(  X\left(  s+\varepsilon\right)  -X\left(  s\right)
\right)  ^{3}\left(  X\left(  t+\varepsilon\right)  -X\left(  t\right)
\right)  ^{3}\right] \\
&  =\frac{1}{\varepsilon^{2}}\int_{0}^{T}ds\int_{0}^{T}dt\mathbf{E}\left[
I_{1}\left(  \Delta G_{s}\right)  ^{3}I_{1}\left(  \Delta G_{t}\right)
^{3}\right] \\
&  =\frac{9}{\varepsilon^{2}}\int_{0}^{T}ds\int_{0}^{T}dt\mathbf{E}\left[
I_{1}\left(  \Delta G_{s}\right)  I_{1}\left(  \Delta G_{t}\right)  \right]
\left\vert \Delta G_{s}\right\vert _{L^{2}\left(  [0,T]\right)  }%
^{2}\left\vert \Delta G_{t}\right\vert _{L^{2}\left(  [0,T]\right)  }^{2}\\
&  +\frac{9}{\varepsilon^{2}}\int_{0}^{T}ds\int_{0}^{T}dt\mathbf{E}\left[
I_{3}\left(  \left(  \Delta G_{s}\right)  ^{\otimes3}\right)  I_{3}\left(
\left(  \Delta G_{t}\right)  ^{\otimes3}\right)  \right]  .
\end{align*}
Now we use the fact that $\mathbf{E}\left[  I_{3}\left(  f\right)
I_{3}\left(  g\right)  \right]  =\left\langle f,g\right\rangle _{L^{2}\left(
[0,T]^{3}\right)  }$, plus the fact that in our homogeneous situation
$\left\vert \Delta G_{s}\right\vert _{L^{2}\left(  [0,T]\right)  }^{2}%
=\delta^{2}\left(  \varepsilon\right)  $ for any $s$. Hence the above equals%
\begin{align*}
&  \frac{9\delta^{4}\left(  \varepsilon\right)  }{\varepsilon^{2}}\int_{0}%
^{T}ds\int_{0}^{T}dt\left\langle \Delta G_{s},\Delta G_{t}\right\rangle
_{L^{2}\left(  [0,T]\right)  }+\frac{9}{\varepsilon^{2}}\int_{0}^{T}ds\int
_{0}^{T}dt\left\langle \left(  \Delta G_{s}\right)  ^{\otimes3},\left(  \Delta
G_{t}\right)  ^{\otimes3}\right\rangle _{L^{2}\left(  [0,T]^{3}\right)  }\\
&  =\frac{9\delta^{4}\left(  \varepsilon\right)  }{\varepsilon^{2}}\int
_{0}^{T}ds\int_{0}^{T}dt\int_{0}^{T}du\Delta G_{s}\left(  u\right)  \Delta
G_{t}\left(  u\right)  +\frac{9}{\varepsilon^{2}}\int_{0}^{T}ds\int_{0}^{T}dt%
{\displaystyle\iiint\limits_{[0,T]^{3}}}
\prod_{i=1}^{3}\left(  du_{i}\Delta G_{s}\left(  u_{i}\right)  \Delta
G_{t}\left(  u_{i}\right)  \right) \\
&  =\frac{9\delta^{4}\left(  \varepsilon\right)  }{\varepsilon^{2}}\int
_{0}^{T}du\left\vert \int_{0}^{T}ds\Delta G_{s}\left(  u\right)  \right\vert
^{2}+\frac{9}{\varepsilon^{2}}%
{\displaystyle\iiint\limits_{[0,T]^{3}}}
du_{1}\ du_{2}\ du_{3}\left\vert \int_{0}^{T}ds\prod_{i=1}^{3}\left(  \Delta
G_{s}\left(  u_{i}\right)  \right)  \right\vert ^{2}.
\end{align*}
\vspace{0.1in}

\noindent\emph{Step 4.2: }$J_{1}$\emph{ as a lower bound}. The above
representation is extremely useful because it turns out, as one readily
checks, that of the two summands in the last expression above, the first is
what we called $J_{1}$ and the second is $J_{0}$, and we can now see that both
these terms are positive, which was not at all obvious before, since, as we
recall, the off-diagonal contribution to either term is negative by our
concavity assumption. Nevertheless, we may now have a lower bound on the
$\varepsilon$-variation by finding a lower bound for the term $J_{1}$ alone.

Reverting to our method of separating diagonal and off-diagonal terms, and
recalling by Step 2 that we can restrict $t\geq\varepsilon$, we have%
\begin{align*}
J_{1}  &  =\frac{9\delta^{4}\left(  \varepsilon\right)  }{\varepsilon^{2}%
}2\int_{\varepsilon}^{T}dt\int_{0}^{t}ds\int_{0}^{T}du\Delta G_{s}\left(
u\right)  \Delta G_{t}\left(  u\right) \\
&  =\frac{9\delta^{4}\left(  \varepsilon\right)  }{\varepsilon^{2}}%
2\int_{\varepsilon}^{T}dt\int_{0}^{t}ds\Theta_{\varepsilon}\left(  s,t\right)
\\
&  =\frac{9\delta^{4}\left(  \varepsilon\right)  }{\varepsilon^{2}}%
\int_{\varepsilon}^{T}dt\int_{0}^{t}ds\left(  \delta^{2}\left(
t-s+\varepsilon\right)  -\delta^{2}\left(  t-s\right)  -\left(  \delta
^{2}\left(  t-s\right)  -\delta^{2}\left(  \left\vert t-s-\varepsilon
\right\vert \right)  \right)  \right) \\
&  =J_{1,D}+J_{1,OD}%
\end{align*}
where, performing the change of variables $t-s\mapsto s$%
\begin{align*}
J_{1,D}  &  :=\frac{9\delta^{4}\left(  \varepsilon\right)  }{\varepsilon^{2}%
}\int_{\varepsilon}^{T}dt\int_{0}^{\varepsilon}ds\left(  \delta^{2}\left(
s+\varepsilon\right)  -\delta^{2}\left(  s\right)  -\left(  \delta^{2}\left(
s\right)  -\delta^{2}\left(  \varepsilon-s\right)  \right)  \right) \\
J_{1,OD}  &  :=\frac{9\delta^{4}\left(  \varepsilon\right)  }{\varepsilon^{2}%
}\int_{\varepsilon}^{T}dt\int_{\varepsilon}^{t}ds\left(  \delta^{2}\left(
s+\varepsilon\right)  -\delta^{2}\left(  s\right)  -\left(  \delta^{2}\left(
s\right)  -\delta^{2}\left(  s-\varepsilon\right)  \right)  \right)  .
\end{align*}
\vspace{0.1in}

\noindent\emph{Step 4.3: Upper bound on }$\left\vert J_{1,OD}\right\vert $.
Using the calculations performed in Step 3 (note here that $\left(
m-1\right)  /2=1$, in particular line (\ref{usein43}), we have
\begin{align*}
J_{1,OD}  &  =\frac{9\delta^{4}\left(  \varepsilon\right)  }{\varepsilon^{2}%
}\int_{\varepsilon}^{T}dt\left[  \int_{t-\varepsilon}^{t}ds\int_{s}%
^{s+\varepsilon}d\delta^{2}\left(  r\right)  -\int_{0}^{\varepsilon}ds\int
_{s}^{s+\varepsilon}d\delta^{2}\left(  r\right)  \right] \\
&  =:K_{1}+K_{2}.
\end{align*}
We can already see that $K_{1}\geq0$ and $K_{2}\leq0$, so it's only necessary
to find an upper bound on $\left\vert K_{2}\right\vert $; but in reality, the
reader will easily check that $\left\vert K_{1}\right\vert $ is of the order
$\delta^{6}\left(  \varepsilon\right)  $, and we will see that this is much
smaller than either $J_{1,D}$ or $\left\vert K_{2}\right\vert $. Performing a
Fubini on the variables $s$ and $r$, the integrand in $K_{2}$ is calculated as%
\begin{align*}
\int_{0}^{\varepsilon}ds\int_{s}^{s+\varepsilon}d\delta^{2}\left(  r\right)
&  =\int_{r=0}^{\varepsilon}d\delta^{2}\left(  r\right)  \int_{s=0}^{r}%
ds+\int_{r=\varepsilon}^{2\varepsilon}d\delta^{2}\left(  r\right)
\int_{s=r-\varepsilon}^{\varepsilon}ds\\
&  =\int_{r=0}^{\varepsilon}r\ d\delta^{2}\left(  r\right)  +\int
_{r=\varepsilon}^{2\varepsilon}\left(  2\varepsilon-r\right)  d\delta
^{2}\left(  r\right) \\
&  =\left[  r\delta^{2}\left(  r\right)  \right]  _{0}^{\varepsilon}-\left[
r\delta^{2}\left(  r\right)  \right]  _{\varepsilon}^{2\varepsilon}-\int
_{0}^{\varepsilon}\delta^{2}\left(  r\right)  dr+\int_{\varepsilon
}^{2\varepsilon}\delta^{2}\left(  r\right)  dr+2\varepsilon\left(  \delta
^{2}\left(  2\varepsilon\right)  -\delta^{2}\left(  \varepsilon\right)
\right) \\
&  =-\int_{0}^{\varepsilon}\delta^{2}\left(  r\right)  dr+\int_{\varepsilon
}^{2\varepsilon}\delta^{2}\left(  r\right)  dr.
\end{align*}
In particular, because $\left\vert K_{1}\right\vert \ll\left\vert
K_{2}\right\vert $ and $\delta^{2}$ is increasing, we get%
\begin{align}
\left\vert J_{1,OD}\right\vert  &  \leq\frac{9\delta^{4}\left(  \varepsilon
\right)  }{\varepsilon^{2}}\int_{\varepsilon}^{T}dt\left(  \int_{\varepsilon
}^{2\varepsilon}\delta^{2}\left(  r\right)  dr-\int_{0}^{\varepsilon}%
\delta^{2}\left(  r\right)  dr\right) \nonumber\\
&  =\frac{9\left(  T-\varepsilon\right)  \delta^{4}\left(  \varepsilon\right)
}{\varepsilon^{2}}\left(  \int_{\varepsilon}^{2\varepsilon}\delta^{2}\left(
r\right)  dr-\int_{0}^{\varepsilon}\delta^{2}\left(  r\right)  dr\right)  .
\label{J1ODlater}%
\end{align}
\vspace{0.1in}

\noindent\emph{Step 4.4: Lower bound on }$J_{1,D}$. Note first that
\[
\int_{0}^{\varepsilon}ds\left(  \delta^{2}\left(  s\right)  -\delta^{2}\left(
\varepsilon-s\right)  \right)  =\int_{0}^{\varepsilon}ds\ \delta^{2}\left(
s\right)  -\int_{0}^{\varepsilon}ds\ \delta^{2}\left(  \varepsilon-s\right)
=0.
\]
Therefore%
\begin{align*}
J_{1,D}  &  =\frac{9\delta^{4}\left(  \varepsilon\right)  }{\varepsilon^{2}%
}\int_{\varepsilon}^{T}dt\int_{0}^{\varepsilon}ds\left(  \delta^{2}\left(
s+\varepsilon\right)  -\delta^{2}\left(  s\right)  \right) \\
&  =\frac{9\delta^{4}\left(  \varepsilon\right)  }{\varepsilon^{2}}\left(
T-\varepsilon\right)  \int_{0}^{\varepsilon}ds\int_{s}^{s+\varepsilon}%
d\delta^{2}\left(  r\right)  .
\end{align*}
We can also perform a Fubini on the integral in $J_{1,D}$, obtaining%
\begin{align*}
\int_{0}^{\varepsilon}ds\int_{s}^{s+\varepsilon}d\delta^{2}\left(  r\right)
&  =\int_{0}^{\varepsilon}r\ d\delta^{2}\left(  r\right)  +\varepsilon
\int_{\varepsilon}^{2\varepsilon}d\delta^{2}\left(  r\right) \\
&  =\left[  r\delta^{2}\left(  r\right)  \right]  _{0}^{\varepsilon}-\int
_{0}^{\varepsilon}\delta^{2}\left(  r\right)  dr+\varepsilon\left(  \delta
^{2}\left(  2\varepsilon\right)  -\delta^{2}\left(  \varepsilon\right)
\right) \\
&  =\varepsilon\delta^{2}\left(  2\varepsilon\right)  -\int_{0}^{\varepsilon
}\delta^{2}\left(  r\right)  dr.
\end{align*}
\vspace{0.1in}In other words,%
\[
J_{1,D}=\frac{9\delta^{4}\left(  \varepsilon\right)  }{\varepsilon^{2}}\left(
T-\varepsilon\right)  \left(  \varepsilon\delta^{2}\left(  2\varepsilon
\right)  -\int_{0}^{\varepsilon}\delta^{2}\left(  r\right)  dr\right)  .
\]

\noindent\emph{Step 4.5: conclusion.} We may now compare $J_{1,D}$ and
$\left\vert J_{1,OD}\right\vert $: using the results of Steps 4.1 and 4.2,
\begin{align*}
J_{1}  &  =J_{1,D}-\left\vert J_{1,OD}\right\vert \geq\frac{9\delta^{4}\left(
\varepsilon\right)  }{\varepsilon^{2}}\left(  T-\varepsilon\right)  \left(
\varepsilon\delta^{2}\left(  2\varepsilon\right)  -\int_{0}^{\varepsilon
}\delta^{2}\left(  r\right)  dr\right) \\
&  -\frac{9\delta^{4}\left(  \varepsilon\right)  }{\varepsilon^{2}}\left(
T-\varepsilon\right)  \left(  \int_{\varepsilon}^{2\varepsilon}\delta
^{2}\left(  r\right)  dr-\int_{0}^{\varepsilon}\delta^{2}\left(  r\right)
dr\right) \\
&  =\frac{9\delta^{4}\left(  \varepsilon\right)  }{\varepsilon^{2}}\left(
T-\varepsilon\right)  \int_{\varepsilon}^{2\varepsilon}\left(  \delta
^{2}\left(  2\varepsilon\right)  -\delta^{2}\left(  r\right)  \right)  dr.
\end{align*}
When $\delta$ is in the H\"{o}lder scale $\delta\left(  r\right)  =r^{H}$, the
above quantity is obviously commensurate with $\delta^{6}\left(
\varepsilon\right)  /\varepsilon$, which implies the desired result, but in
order to be sure we are treating all cases, we now present a general proof
which only relies on the fact that $\delta^{2}$ is increasing and concave.

Below we use the notation $\left(  \delta^{2}\right)  ^{\prime}$ for the
density of $d\delta^{2}$, which exists a.e. since $\delta^{2}$ is concave. The
mean value theorem and the concavity of $\delta^{2}$ then imply that for any
$r\in\lbrack\varepsilon,2\varepsilon]$,
\[
\delta^{2}\left(  2\varepsilon\right)  -\delta^{2}\left(  r\right)
\geq\left(  2\varepsilon-r\right)  \inf_{[\varepsilon,2\varepsilon]}\left(
\delta^{2}\right)  ^{\prime}=\left(  2\varepsilon-r\right)  \left(  \delta
^{2}\right)  ^{\prime}\left(  2\varepsilon\right)  .
\]
Thus we can write%
\begin{align*}
J_{1}  &  \geq9(T-\varepsilon)\varepsilon^{-1}\delta^{4}\left(  \varepsilon
\right)  \left(  \delta^{2}\right)  ^{\prime}\left(  2\varepsilon\right)
\int_{\varepsilon}^{2\varepsilon}\left(  2\varepsilon-r\right)  dr\\
&  =9(T-\varepsilon)\varepsilon^{-1}\delta^{4}\left(  \varepsilon\right)
\left(  \delta^{2}\right)  ^{\prime}\left(  2\varepsilon\right)
\varepsilon^{2}/2\\
&  \geq cst\cdot\delta^{4}\left(  \varepsilon\right)  \cdot\left(  \delta
^{2}\right)  ^{\prime}\left(  2\varepsilon\right)  .
\end{align*}
Since $\delta^{2}$ is concave, and $\delta\left(  0\right)  =0$, we have
$\delta^{2}\left(  \varepsilon\right)  \geq\delta^{2}\left(  2\varepsilon
\right)  /2$. Hence, with the notation $f\left(  x\right)  =\delta^{2}\left(
2x\right)  $, we have%
\[
J_{1}\geq cst\cdot f^{2}\left(  \varepsilon\right)  f^{\prime}\left(
\varepsilon\right)  =cst\cdot\left(  f^{3}\right)  ^{\prime}\left(
\varepsilon\right)  .
\]
Therefore we have that $\lim_{\varepsilon\rightarrow0}\left(  f^{3}\right)
^{\prime}\left(  \varepsilon\right)  =0$. We prove this implies $\lim
_{\varepsilon\rightarrow0}\varepsilon^{-1}f^{3}\left(  \varepsilon\right)
=0$. Indeed, fix $\eta>0$; then there exists $\varepsilon_{\eta}>0$ such that
for all $\varepsilon\in(0,\varepsilon_{\eta}]$, $0\leq\left(  f^{3}\right)
^{\prime}\left(  \varepsilon\right)  \leq\eta$ (we used the positivity of
$\left(  \delta^{2}\right)  ^{\prime}$). Hence, also using $f\left(  0\right)
=0$, for any $\varepsilon\in(0,\varepsilon_{\eta}]$,
\[
0\leq\frac{f^{3}\left(  \varepsilon\right)  }{\varepsilon}=\frac
{1}{\varepsilon}\int_{0}^{\varepsilon}\left(  f^{3}\right)  ^{\prime}\left(
x\right)  dx\leq\frac{1}{\varepsilon}\int_{0}^{\varepsilon}\eta dx=\eta.
\]
This proves that $\lim_{\varepsilon\rightarrow0}\varepsilon^{-1}f^{3}\left(
\varepsilon\right)  =0$, which is equivalent to the announced necessary
condition, and finishes the proof of the theorem.
\end{proof}

\subsection{Non-homogeneous case\label{NonHomogGaussSect}}

The concavity and homogeneity assumptions were used heavily above for the
proof of the necessary condition in Theorem \ref{HomogGauss}. However, these
assumptions can be considerably weakened while still resulting in a sufficient
condition. We now show that a weak uniformity condition on the variances,
coupled with a natural bound on the second-derivative measure of $\delta^{2}$,
result in zero $m$-variation processes.

\begin{theorem}
\label{nonhomoggauss}Let $m>1$ be an odd integer. Let $X$ be a centered
Gaussian process on $[0,T]$ with canonical metric%
\[
\delta^{2}\left(  s,t\right)  :=\mathbf{E}\left[  \left(  X\left(  t\right)
-X\left(  s\right)  \right)  ^{2}\right]  .
\]
Define a univariate function on $[0,T]$, also denoted by $\delta^{2}$, via%
\[
\delta^{2}\left(  r\right)  :=\sup_{s\in\lbrack0,T]}\delta^{2}\left(
s,s+r\right)  ,
\]
and assume that for $r$ near $0$,%
\begin{equation}
\delta\left(  r\right)  =o\left(  r^{1/2m}\right)  . \label{nhdeltacond}%
\end{equation}
Assume that, in the sense of distributions, the derivative $\partial\delta
^{2}/\left(  \partial s\partial t\right)  $ is a finite signed $\sigma$ finite
measure $\mu$ on $[0,T]^{2} - \Delta$ where $\Delta$ is the diagonal $\{(s,s)
\vert s \in[0,T]\}$. Denote the off-diagonal simplex by $OD=\{\left(
s,t\right)  :0\leq s\leq t-\varepsilon\leq T\}$; assume $\mu$ satisfies, for
some constant $c$ and for all $\varepsilon$ small enough,%
\begin{equation}
\left\vert \mu\right\vert \left(  OD\right)  \leq c\varepsilon^{-(m-1)/m},
\label{nhmubirdiecond}%
\end{equation}
where $\left\vert \mu\right\vert $ is the total variation measure of $\mu$.
Then $X$ has zero $m$th variation.
\end{theorem}

\begin{proof}
\noindent\emph{Step 0: setup.} Recall that by Lemma \ref{lemma1},%
\begin{align}
\mathbf{E}\left[  \left(  [X,m]_{\varepsilon}\left(  T\right)  \right)
^{2}\right]   &  =\frac{1}{\varepsilon^{2}}\sum_{j=0}^{\left(  m-1\right)
/2}c_{j}\int_{0}^{T}\int_{0}^{t}dtds\Theta^{\varepsilon}\left(  s,t\right)
^{m-2j}\ \delta^{2j}\left(  s,s+\varepsilon\right)  \ \delta^{2j}\left(
t,t+\varepsilon\right) \label{later4Ito}\\
&  :=\sum_{j=0}^{\left(  m-1\right)  /2}J_{j}\nonumber
\end{align}
and now we express%
\begin{equation}
\Theta^{\varepsilon}\left(  s,t\right)  =\mu\left(  \lbrack s,s+\varepsilon
]\times\lbrack t,t+\varepsilon)\right)  =\int_{s}^{s+\varepsilon}\int
_{t}^{t+\varepsilon}\mu\left(  dudv\right)  . \label{Thetamubirdie}%
\end{equation}
We again separate the diagonal term from the off-diagonal term, although this
time the diagonal is twice as wide: it is defined as $\{\left(  s,t\right)
:0\leq t-2\varepsilon\leq s\leq t\}$.\vspace{0.1in}

\noindent\emph{Step 1: diagonal.} Using Cauchy-Schwarz's inequality which
implies $\left\vert \Theta^{\varepsilon}\left(  s,t\right)  \right\vert
\leq\delta\left(  s,s+\varepsilon\right)  \delta\left(  t,t+\varepsilon
\right)  $, and bounding each term $\delta\left(  s,s+\varepsilon\right)  $ by
$\delta\left(  \varepsilon\right)  $, the diagonal portion of $\mathbf{E}%
\left[  \left(  [X,m]_{\varepsilon}\left(  T\right)  \right)  ^{2}\right]  $
can be bounded above, in absolute value, by%
\[
\frac{1}{\varepsilon^{2}}\sum_{j=0}^{\left(  m-1\right)  /2}c_{j}%
\int_{2\varepsilon}^{T}dt\int_{t-2\varepsilon}^{t}ds\delta^{2m}\left(
\varepsilon\right)  =cst\cdot\varepsilon^{-1}\delta^{2m}\left(  \varepsilon
\right)  .
\]
The hypothesis on the univariate $\delta^{2}$ implies that this converges to
$0$ as $\varepsilon$ tends to $0$. The case of $t\leq2\varepsilon$ works
equally easily.\vspace{0.1in}

\noindent\emph{Step 2: off diagonal.} The off-diagonal contribution is the sum
for $j=0,\cdots,\left(  m-1\right)  /2$ of the terms%
\begin{equation}
J_{j,OD}=\varepsilon^{-2}c_{j}\int_{2\varepsilon}^{T}dt\int_{0}%
^{t-2\varepsilon}ds\delta^{2j}\left(  s,s+\varepsilon\right)  \delta
^{2j}\left(  t,t+\varepsilon\right)  \Theta^{\varepsilon}\left(  s,t\right)
^{m-2j} \label{JjOD}%
\end{equation}
As we will prove below, the dominant term turns out to be $J_{\left(
m-1\right)  /2,OD}$; we deal with it now.\vspace{0.1in}

\noindent\emph{Step 2.1: term }$J_{\left(  m-1\right)  /2,OD}$. Denoting
$c=\left\vert c_{\left(  m-1\right)  /2}\right\vert $, we have%
\[
\left\vert J_{\left(  m-1\right)  /2,OD}\right\vert \leq\frac{c\delta
^{2m-2}\left(  \varepsilon\right)  }{\varepsilon^{2}}\int_{2\varepsilon}%
^{T}dt\int_{0}^{t-2\varepsilon}ds\left\vert \Theta^{\varepsilon}\left(
s,t\right)  \right\vert .
\]
We estimate the integral, using the formula (\ref{Thetamubirdie}) and Fubini's
theorem:%
\begin{align*}
\int_{2\varepsilon}^{T}dt\int_{0}^{t-2\varepsilon}ds\left\vert \Theta
^{\varepsilon}\left(  s,t\right)  \right\vert  &  =\int_{2\varepsilon}%
^{T}dt\int_{0}^{t-2\varepsilon}ds\left\vert \int_{s}^{s+\varepsilon}\int
_{t}^{t+\varepsilon}\mu\left(  dudv\right)  \right\vert \\
&  \leq\int_{2\varepsilon}^{T}dt\int_{0}^{t-2\varepsilon}ds\int_{s}%
^{s+\varepsilon}\int_{t}^{t+\varepsilon}\left\vert \mu\right\vert \left(
dudv\right) \\
&  =\int_{v=2\varepsilon}^{T+\varepsilon}\int_{u=0}^{\min(v,T)-\varepsilon
}\left\vert \mu\right\vert \left(  dudv\right)  \int_{t=\max(2\varepsilon
,v-\varepsilon,u+\varepsilon)}^{\min(v,T)}\int_{s=\max\left(  0,u-\varepsilon
\right)  }^{\min\left(  u,t-2\varepsilon\right)  }ds\ dt\\
&  \leq\int_{v=2\varepsilon}^{T+\varepsilon}\int_{u=0}^{v-\varepsilon
}\left\vert \mu\right\vert \left(  dudv\right)  \int_{t=v-\varepsilon}^{v}%
\int_{s=u-\varepsilon}^{u}ds\ dt\\
&  =\varepsilon^{2}\int_{v=2\varepsilon}^{T+\varepsilon}\int_{u=0}%
^{v-\varepsilon}\left\vert \mu\right\vert \left(  dudv\right)
\end{align*}
Hence we have%
\begin{align*}
J_{\left(  m-1\right)  /2,OD}  &  \leq c\delta^{2m-2}\left(  \varepsilon
\right)  \int_{v=2\varepsilon}^{T+\varepsilon}\int_{u=0}^{v-\varepsilon
}\left\vert \mu\right\vert \left(  dudv\right) \\
&  \leq c\delta^{2m-2}\left(  \varepsilon\right)  \left\vert \mu\right\vert
\left(  OD\right)  ,
\end{align*}
which again converges to $0$ by hypothesis as $\varepsilon$ goes to
$0$.\vspace{0.1in}

\noindent\emph{Step 2.2: other }$J_{j,OD}$ \emph{terms}. Let now $j<\left(
m-1\right)  /2$. Using Cauchy-Schwarz's inequality for all but one of the
$m-2j$ factors $\Theta$ in the expression (\ref{JjOD}) for $J_{j,OD}$, which
is allowed because $m-2j\geq1$ here, exploiting the bounds on the variance
terms via the univariate function $\delta$, we have%
\begin{align*}
\left\vert J_{j,OD}\right\vert  &  \leq\frac{\delta^{4j}\left(  \varepsilon
\right)  c_{j}}{\varepsilon^{2}}\int_{2\varepsilon}^{T}dt\int_{0}%
^{t-2\varepsilon}ds\left\vert \Theta^{\varepsilon}\left(  s,t\right)
\right\vert ^{m-2j-1}\left\vert \Theta^{\varepsilon}\left(  s,t\right)
\right\vert \\
&  \leq\delta^{2m-2}\left(  \varepsilon\right)  c_{j}\varepsilon^{-2}%
\int_{2\varepsilon}^{T}dt\int_{0}^{t-2\varepsilon}ds\left\vert \Theta
^{\varepsilon}\left(  s,t\right)  \right\vert ,
\end{align*}
which is the same term we estimated in Step 2.1. This finishes the proof of
the theorem.\bigskip
\end{proof}

\label{RL} A typical situation covered by the above theorem is that of the
Riemann-Liouville fractional Brownian motion. This is the process $B^{H,RL}$
defined by $B^{H,RL}\left(  t\right)  =\int_{0}^{t}\left(  t-s\right)
^{H-1/2}dW\left(  s\right)  $. Its canonical metric is not homogeneous, but we
do have, when $H\in\left(  0,1/2\right)  $,%
\begin{equation}
\label{ERL}\left\vert t-s\right\vert ^{H}\leq\delta\left(  s,t\right)
\leq2\left\vert t-s\right\vert ^{H},
\end{equation}
which implies, incidentally, that $B^{H,RL}$ has the same regularity
properties as fractional Brownian motion, see \cite{MV} for a proof of these
inequalities. To apply the theorem, we must choose $H>1/\left(  2m\right)  $
for the condition on the variances. For the other condition, we calculate that
$\mu\left(  dsdt\right)  =2H\left(  1-2H\right)  \left\vert t-s\right\vert
^{2H-2}dsdt$, and therefore%
\[
\mu\left(  OD\right)  =\left\vert \mu\right\vert \left(  OD\right)  =c_{H}%
\int_{0}^{T}\int_{\varepsilon}^{t}s^{2H-2}dsdt\leq c_{H}T\varepsilon^{2H-1}.
\]
This quantity is bounded above by $\varepsilon^{-1+1/m}$ as soon as
$H\geq1/\left(  2m\right)  $, of course, so the strict inequality is
sufficient to apply the theorem and conclude that $B^{H,RL}$ then has zero
$m$th variation.

One can generalize this example to any Gaussian process with a
Volterra-convolution kernel: let $\gamma^{2}$ be a univariate increasing
concave function, differentiable everywhere except possibly at $0$, and define%
\begin{equation}
X\left(  t\right)  =\int_{0}^{t}\left(  \frac{d\gamma^{2}}{dr}\right)
^{1/2}\left(  t-r\right)  dW\left(  r\right)  . \label{volterra}%
\end{equation}
Then one can show (see \cite{MV}) that the canonical metric $\delta^{2}\left(
s,t\right)  $ of $X$ is bounded above by $2\gamma^{2}\left(  \left\vert
t-s\right\vert \right)  $, so that we can use the univariate $\delta
^{2}=2\gamma^{2}$, and also $\delta^{2}\left(  s,t\right)  $ is bounded below
by $\gamma^{2}\left(  \left\vert t-s\right\vert \right)  $. Similar
calculations to the above then easily show that $X$ has zero $m$th variation
as soon as $\delta^{2}\left(  r\right)  =o\left(  r^{1/\left(  2m\right)
}\right)  $. Hence there are inhomogeneous processes that are more irregular
than fractional Brownian for any $H>1/\left(  2m\right)  $ which still have
zero $m$th variation: use for instance the $X$ above with $\gamma^{2}\left(
r\right)  =r^{1/\left(  2m\right)  }/\log\left(  1/r\right)  $. \newline

\subsection{Non-odd integer powers}

When $m\geq1$ is not an odd integer, recall that to define the $m$th odd
variation, we use the convention $((x))^{m}=\left\vert x\right\vert
^{m}\mbox{sgn}\left(  x\right)  $, which is an odd function. The idea here is
to use the Taylor expansion for this function up to order $\left[  m\right]
$, with a remainder of order $\left[  m\right]  +1$; it can be expressed as
the following elementary lemma, whose proof is omitted.

\begin{lemma}
\label{lemmaabm}Fix $m>1$ and two reals $a$ and $b$ such that $\left\vert
a\right\vert \geq\left\vert b\right\vert $. Let $\binom{m}{k}$ denote the
formal binomial coefficient $m\left(  m-1\right)  \cdots\left(  m-k+1\right)
/\left(  k\left(  k-1\right)  \cdots1\right)  $ and let $\left(  \left(
x\right)  \right)  :=\left\vert x\right\vert \mbox{sgn}\left(  x\right)  $.
Then, for all reals $a,b$,

\begin{enumerate}
\item if $\left\vert b/a\right\vert <1$,%
\[
\left(  \left(  a\right)  \right)  ^{m}\left(  \left(  a+b\right)  \right)
^{m}=\sum_{k=0}^{\left[  m\right]  -1}\binom{m}{k}\mbox{sgn}^{k}\left(
a\right)  \left\vert a\right\vert ^{2m-k}b^{k}+\left\vert a\right\vert
^{2m}f_{m-1}\left(  \frac{b}{a}\right)  ,
\]

\item and if $\left\vert a/b\right\vert <1$%
\[
\left(  \left(  a\right)  \right)  ^{m}\left(  \left(  a+b\right)  \right)
^{m}=\sum_{k=0}^{\left[  m\right]  }\binom{m}{k}\mbox{sgn}^{k+1}\left(
a\right)  \mbox{sgn}^{k+1}\left(  b\right)  \left\vert a\right\vert
^{m+k}\left\vert b\right\vert ^{m-k}+\left(  ab\right)  ^{m}f_{m}\left(
\frac{a}{b}\right)  ,
\]

\end{enumerate}

where for all $\left\vert x\right\vert <1$, $\left\vert f_{m}\left(  x\right)
\right\vert \leq c_{m}\left\vert x\right\vert ^{\left[  m\right]  +1}$ where
$c_{m}$ depends only on $m$. When $m$ is an integer, the above formulas have
null remainder terms $f$.
\end{lemma}

When we apply this lemma to the question of finding Gaussian fields with zero
odd $m$th variation, we are able to prove that the sufficient condition of
Theorem \ref{HomogGauss} still works. The result is the following.

\begin{theorem}
\label{nope}Let $X$ be as in Theorem \ref{HomogGauss} ($X$ with homogeneous
increments, with an increasing and concave $\delta^{2}$). Let $m$ be any real
number $>1$. Consider the condition

\begin{description}
\item[(S)] $\delta^{2}$ is twice differentiable, and for some $c<1$, the
function $r\mapsto\left\vert \left(  \delta^{2}\right)  ^{\prime\prime}\left(
r\right)  \right\vert $ is decreasing and bounded above by $cr^{-2}\delta
^{2}\left(  r\right)  $.
\end{description}

If $m$ \textbf{is not} an integer, then $X$ has zero odd $m$th variation as
soon as $\delta\left(  r\right)  =o\left(  r^{-1/(2m)}\right)  $ and condition
(S) holds.

If $m$ \textbf{is} an integer, the same is true without needing condition (S).
\end{theorem}

\begin{remark}
\label{noperem}The technical Condition (S) is not a restriction on the range
of regularity of $X$. Indeed, for all fBm's that are more irregular than
Brownian motion, we have $\left\vert \left(  \delta^{2}\right)  ^{\prime
\prime}\left(  r\right)  \right\vert =2H(1-2H)r^{-2}\delta^{2}\left(
r\right)  $ which is indeed decreasing, and the constant $c$ can be taken as
$1/4$ (this maximum is attained for $H=1/4$). For perturbations of the fBm
scale, where $\delta\left(  r\right)  $ is of the order $r^{2H}\log^{\beta
}\left(  1/r\right)  $ for some $\beta$, Condition (S) is also typically
satisfied. Beyond the H\"{o}lder scale, in cases where $\delta\left(
r\right)  $ is of the order $\log^{\beta}\left(  1/r\right)  $ with
$\beta<-1/2$, we will actually have a stronger upper-bound condition, of the
type $\left\vert \left(  \delta^{2}\right)  ^{\prime\prime}\left(  r\right)
\right\vert =o\left(  r^{-2}\delta^{2}\left(  r\right)  \right)  $. That
$\left\vert \left(  \delta^{2}\right)  ^{\prime\prime}\left(  r\right)
\right\vert $ be decreasing is typical of all Gaussian processes with
homogeneous increments, even those which are more regular than Brownian motion.
\end{remark}

\begin{proof}
[Proof of Theorem \ref{nope}]\noindent\emph{Step 0: setup.} Recall that with
$Y=X\left(  t+\varepsilon\right)  -X\left(  t\right)  $ and $Z=X\left(
s+\varepsilon\right)  -X\left(  s\right)  $, we have%
\[
\mathbf{E}\left[  \left(  [X,m]_{\varepsilon}\left(  T\right)  \right)
^{2}\right]  =\frac{2}{\varepsilon^{2}}\int_{0}^{T}\int_{0}^{t}dtds\mathbf{E}%
\left[  \left(  \left(  Y\right)  \right)  ^{m}\left(  \left(  Z\right)
\right)  ^{m}\right]  .
\]
Now introduce the shorthand notation $\sigma^{2}=Var\left[  Y\right]  $,
$\tau^{2}=Var\left[  Z\right]  $, and $\theta=\mathbf{E}\left[  YZ\right]
=\Theta^{\varepsilon}(s,t)$. Thus $Y=\sigma M$ where $M$ is a standard normal
r.v.. We can write the \textquotedblleft linear regression\textquotedblright%
\ expansion of $Z$ w.r.t. $Y$, using another standard normal r.v. $N$
independent of $M$:%
\[
Z=\frac{\theta}{\sigma}M+\rho\tau N
\]
where%
\[
\rho:=\left(  1-\left(  \frac{\theta}{\sigma\tau}\right)  ^{2}\right)
^{1/2}.
\]
Note that $\rho$ is always well-defined and positive by Cauchy-Schwarz's
inequality. Therefore%
\[
\left(  \left(  Y\right)  \right)  ^{m}\left(  \left(  Z\right)  \right)
^{m}=\sigma^{m}\left(  \left(  M\right)  \right)  ^{m}\left(  \left(
\frac{\theta}{\sigma}M+\rho\tau N\right)  \right)  ^{m}=\mbox{sgn}\left(
\theta\right)  \sigma^{2m}\left\vert \theta\right\vert ^{-m}\left(  \left(
a\right)  \right)  ^{m}\left(  \left(  a+b\right)  \right)  ^{m}%
\]
where%
\[
a:=\frac{\theta}{\sigma}M\qquad\mbox{and}\qquad b:=\rho\tau N.
\]
Applying Lemma \ref{lemmaabm}, we get that $\left(  \left(  Y\right)  \right)
^{m}\left(  \left(  Z\right)  \right)  ^{m}$ is the sum of the following four
expressions:%
\begin{align}
A  &  :=\mathbf{1}_{\left\vert \frac{\rho\tau\sigma N}{\theta M}\right\vert
<1}\sum_{k=0}^{\left[  m\right]  -1}\binom{m}{k}\mbox{sgn}^{k+1}\left(
\theta\right)  \left\vert \theta\right\vert ^{m-k}\sigma^{k}\tau^{k}\rho
^{k}\ \left\vert M\right\vert ^{2m-k}N^{k}\mbox{sgn}^{k}\left(  M\right)
\label{A}\\
A^{\prime}  &  :=\mbox{sgn}\left(  \theta\right)  \mathbf{1}_{\left\vert
\frac{\rho\tau\sigma N}{\theta M}\right\vert <1}\left\vert \theta\right\vert
^{m}\left\vert M\right\vert ^{2m}f_{m-1}\left(  \frac{\rho\tau\sigma N}{\theta
M}\right) \label{A'}\\
B  &  :=\mathbf{1}_{\left\vert \frac{\rho\tau\sigma N}{\theta M}\right\vert
>1}\sum_{k=0}^{\left[  m\right]  }\binom{m}{k}\mbox{sgn}^{k+2}\left(
\theta\right)  \left\vert \theta\right\vert ^{k}\sigma^{m-k}\tau^{m-k}%
\rho^{m-k}\mbox{sgn}^{k+1}\left(  NM\right)  \left\vert M\right\vert
^{m+k}\left\vert N\right\vert ^{m-k}\label{B}\\
B^{\prime}  &  :=\mbox{sgn}^{2}\left(  \theta\right)  \mathbf{1}_{\left\vert
\frac{\rho\tau\sigma N}{\theta M}\right\vert >1}\left(  \rho\sigma\tau\right)
^{m}\left(  MN\right)  ^{m}f_{m}\left(  \frac{\theta M}{\rho\tau\sigma
N}\right)  \label{B'}%
\end{align}

\noindent\emph{Step 1: cancellations in expectation calculation for }$A$\emph{
and }$B$. In evaluating the $\varepsilon$-$m$th variation $\mathbf{E}\left[
\left(  [X,m]_{\varepsilon}\left(  T\right)  \right)  ^{2}\right]  $, terms in
$A$ and $B$ containing odd powers of $M$ and $N$ will cancel, because of the
symmetry of the normal law, of the fact that the indicator functions in the
expressions for $A$ and $B$ above are even functions of $M$ and of $N$, and of
their independence. Hence we can perform the following calculations, where
$a_{m,k}:=\mathbf{E}\left[  \left\vert M\right\vert ^{2m-k}\left\vert
N\right\vert ^{k}\right]  $ and $b_{m,k}:=\mathbf{E}\left[  \left\vert
M\right\vert ^{m+k}\left\vert N\right\vert ^{m-k}\right]  $ are positive
constants depending only on $m$ and $k$.\vspace*{0.1in}

\noindent\emph{Step 1.1: expectation of }$A$. In this case, because of the
term $N^{k}$, the expectation of all the terms in (\ref{A}) with $k$ odd drop
out. We can expand the term $\rho^{k}$ using the binomial formula, and then
perform a change of variables. We then have, with $n=\left[  m\right]  -2$
when $\left[  m\right]  $ is even, or $n=\left[  m\right]  -1$ when $\left[
m\right]  $ is odd,%
\begin{align*}
\left\vert \mathbf{E}\left[  A\right]  \right\vert  &  \leq\sum
_{\substack{k=0\\k\ even}}^{\left[  m\right]  -1}\binom{m}{k}\left\vert
\theta\right\vert ^{m-k}\sigma^{k}\tau^{k}a_{m,k}\sum_{\ell=0}^{k/2}\binom
{k}{\ell}\left(  -1\right)  ^{\ell}\left(  \frac{\theta}{\sigma\tau}\right)
^{2\ell}\\
&  =\sum_{j=0}^{n/2}\left\vert \theta\right\vert ^{m-2j}\left(  \sigma
\tau\right)  ^{2j}\sum_{\substack{k=2j\\k\ even}}^{n}\binom{m}{k}\binom
{k}{k/2-j}\left(  -1\right)  ^{k/2-j}a_{m,k}\\
&  \leq\sum_{j=0}^{n/2}\left\vert \theta\right\vert ^{m-2j}\left(  \sigma
\tau\right)  ^{2j}c_{m,j}%
\end{align*}
where $c_{m,j}$ are positive constant depending only on $m$ and $j$. In all
cases, the portion of $\mathbf{E}\left[  \left(  [X,m]_{\varepsilon}\left(
T\right)  \right)  ^{2}\right]  $ corresponding to $A$ can be treated using
the same method as in the proof of Theorem \ref{HomogGauss}. More precisely,
after multiplying by $\varepsilon^{-2}$ and integrating over $s$ and $t$, as
we should, each term in the last sum above is of the same form as the term
$J_{j}$ in the proof of Theorem \ref{HomogGauss}, whose upper-estimation is
the subject of Steps 1, 2, and 3 in that proof. The lowest power is attained
when $j=n/2$, i.e., when $\left[  m\right]  $ is even we have $\left\vert
\theta\right\vert ^{2+m-\left[  m\right]  }$, and when $\left[  m\right]  $ is
odd, we have $\left\vert \theta\right\vert ^{1+m-\left[  m\right]  }$. In both
cases, the power is greater than $1$. All other values of $j$ correspond of
course to higher powers of $\left\vert \theta\right\vert $. This means we can
use Cauchy-Schwarz's inequality to get the bound, valid for all $j$,
\[
\left\vert \theta\right\vert ^{m-2j}\left(  \sigma\tau\right)  ^{2j}%
=\left\vert \theta\right\vert \left\vert \theta\right\vert ^{m-2j-1}\left(
\sigma\tau\right)  ^{2j}\leq\left\vert \theta\right\vert \left(  \sigma
\tau\right)  ^{m-1},
\]
and we are back to the situation solved in the proof of Theorem
\ref{HomogGauss}, which corresponded therein to the case $j=(m-1)/2$ when $m$
is was odd integer. Thus the portion of $\mathbf{E}\left[  \left(
[X,m]_{\varepsilon}\left(  T\right)  \right)  ^{2}\right]  $ corresponding to
$A$ tends to $0$ as $\varepsilon\rightarrow0$, as long as $\delta\left(
r\right)  =o\left(  r^{1/\left(  2m\right)  }\right)  $.\vspace*{0.1in}

\noindent\emph{Step 1.2: expectation of }$B$. This portion is dealt with
similarly. Because of the term $\mbox{sgn}^{k+1}\left(  N\right)  $, the
expectation of all the terms in (\ref{B}) with $k$ even drop out. Contrary to
the case of $A$, we do not need to expand $\rho^{m-k}$ in a binomial series.
Since $k$ is now $\geq1$, we simply use Cauchy-Schwarz's inequality to write
$\left\vert \theta\right\vert ^{k}\leq\left\vert \theta\right\vert \left(
\sigma\tau\right)  ^{k-1}$. Of course, we also have $\rho<1$. Hence
\begin{align}
\left\vert \mathbf{E}\left[  B\right]  \right\vert  &  =\left\vert
\mathbf{E}\left[  \mathbf{1}_{\left\vert \frac{\rho\tau\sigma N}{\theta
M}\right\vert >1}\sum_{\substack{k=1\\k\ odd}}^{\left[  m\right]  }\binom
{m}{k}\mbox{sgn}^{k+2}\left(  \theta\right)  \left\vert \theta\right\vert
^{k}\sigma^{m-k}\tau^{m-k}\rho^{m-k}\mbox{sgn}^{k+1}\left(  MN\right)
\left\vert M\right\vert ^{m+k}\left\vert N\right\vert ^{m-k}\right]
\right\vert \nonumber\\
&  \leq\left\vert \theta\right\vert \left(  \sigma\tau\right)  ^{m-1}%
\sum_{\substack{k=1\\k\ odd}}^{\left[  m\right]  }b_{m,k}\binom{m}{k}.
\label{goodineq}%
\end{align}
We see here in all cases that we are exactly in the same situation of the
proof of Theorem \ref{HomogGauss} (again, power of $\left\vert \theta
\right\vert $ is $\left\vert \theta\right\vert ^{1}$). Thus the portion of
$\mathbf{E}\left[  \left(  [X,m]_{\varepsilon}\left(  T\right)  \right)
^{2}\right]  $ corresponding to $B$ converges to $0$ as soon as $\delta
^{2}\left(  r\right)  =o\left(  r^{1/\left(  2m\right)  }\right)  $%
.\vspace*{0.1in}

\noindent\emph{Step 2. The error term }$A^{\prime}$. For $A^{\prime}$ given in
(\ref{A'}), we immediately have%
\begin{align*}
\left\vert \mathbf{E}\left[  A^{\prime}\right]  \right\vert  &  =\left\vert
\mathbf{E}\left[  \mathbf{1}_{\left\vert \frac{\rho\tau\sigma N}{\theta
M}\right\vert <1}\left\vert \theta\right\vert ^{m}\left\vert M\right\vert
^{2m}f_{m-1}\left(  \frac{\rho\tau\sigma N}{\theta M}\right)  \right]
\right\vert \\
&  \leq c_{m}\left\vert \theta\right\vert ^{m}\mathbf{E}\left[  \mathbf{1}%
_{\left\vert \frac{\rho\tau\sigma N}{\theta M}\right\vert <1}\left\vert
M\right\vert ^{2m}\left\vert \frac{\rho\tau\sigma N}{\theta M}\right\vert
^{\left[  m\right]  }\right] \\
&  =c_{m}\left\vert \theta\right\vert ^{m-\left[  m\right]  }\left(  \rho
\tau\sigma\right)  ^{\left[  m\right]  }\mathbf{E}\left[  \mathbf{1}%
_{\left\vert \frac{\rho\tau\sigma N}{\theta M}\right\vert <1}\left\vert
M\right\vert ^{2m-\left[  m\right]  }\left\vert N\right\vert ^{\left[
m\right]  }\right]  `.
\end{align*}
We see here that we cannot ignore the indicator function inside the
expectation, because if we did we would be left with $\left\vert
\theta\right\vert $ to the power $m-\left[  m\right]  $, which is less than
$1$, and therefore does not allow us to use the proof of Theorem
\ref{HomogGauss}.

To estimate the expectation, let $x=\frac{\rho\tau\sigma}{\left\vert
\theta\right\vert }$. We can use H\"{o}lder's inequality for a conjugate pair
$p,q$ with $p$ very large, to write
\[
\mathbf{E}\left[  \mathbf{1}_{\left\vert \frac{\rho\tau\sigma N}{\theta
M}\right\vert <1}\left\vert M\right\vert ^{2m-\left[  m\right]  }\left\vert
N\right\vert ^{\left[  m\right]  }\right]  \leq\mathbf{P}^{1/q}\left[
\left\vert xN\right\vert <\left\vert M\right\vert \right]  \ \mathbf{E}%
^{1/p}\left[  \left\vert M\right\vert ^{2mp-\left[  m\right]  p}\left\vert
N\right\vert ^{\left[  m\right]  p}\right]  .
\]
The second factor in the right-hand side above is a constant $c_{m,p}$
depending only on $m$ and $p$. For the first factor, we can use the following
standard estimate for all $y>0$: $\int_{y}^{\infty}e^{-z^{2}/2}dz\leq
\mathrm{cst} y^{-1}e^{-y^{2}/2}$. Therefore,
\begin{align*}
\mathbf{P}\left[  \left\vert xN\right\vert <\left\vert M\right\vert \right]
&  =2\int_{0}^{\infty}\frac{du}{\sqrt{2\pi}}e^{-u^{2}/2}\mathbf{P}\left[
\left\vert M\right\vert >xu\right]  \leq\sqrt{\frac{2}{\pi}}\int_{0}%
^{1/x}du\ e^{-u^{2}/2}+\sqrt{\frac{2}{\pi}}\int_{1/x}^{\infty}du\frac
{e^{-u^{2}/2}}{ux}e^{-u^{2}x^{2}/2}\\
&  \leq\sqrt{2/\pi}\frac{1}{x}+\sqrt{2/\pi}\int_{1/x}^{\infty}du\frac{1}%
{ux}e^{-u^{2}x^{2}/2}=\sqrt{2/\pi}\left(  \frac{1}{x}+\frac{1}{x}\int
_{1}^{\infty}dv\frac{1}{v}e^{-v^{2}/2}\right)  =\frac{c}{x}%
\end{align*}
where $c$ is a universal constant.

Now choose $q$ so that $m-\left[  m\right]  +1/q=1$, i.e. $q=\left(
1-m+\left[  m\right]  \right)  ^{-1}$, which exceeds $1$ as long as $m$ is not
an integer. Then we get%
\[
\left\vert \mathbf{E}\left[  A^{\prime}\right]  \right\vert \leq
c_{m}\left\vert \theta\right\vert ^{m-\left[  m\right]  }\left(  \rho
\tau\sigma\right)  ^{\left[  m\right]  }c_{m,p}\left(  \frac{c\left\vert
\theta\right\vert }{\rho\sigma\tau}\right)  ^{1/q}=c_{m}c_{m,p}c^{1/q}%
\left\vert \theta\right\vert \left(  \rho\tau\sigma\right)  ^{m-1},
\]
and we are again back to the usual computations. The case of $m$ integer is
dealt with in Step 4 below.\vspace*{0.1in}

\noindent\emph{Step 3. The error term }$B^{\prime}$. For $B^{\prime}$ in
(\ref{B'}), we have
\begin{align*}
\left\vert \mathbf{E}\left[  B^{\prime}\right]  \right\vert  &  \leq
c_{m}\sigma^{m}\tau^{m}\rho^{m}\mathbf{E}\left[  \left\vert MN\right\vert
^{m}\left\vert \frac{\theta M}{\rho\tau\sigma N}\right\vert ^{\left[
m\right]  +1}\right] \\
&  =c_{m}\left(  \rho\sigma\tau\right)  ^{m-\left[  m\right]  -1}\left\vert
\theta\right\vert ^{1+\left[  m\right]  }\mathbf{E}\left[  \left\vert
M\right\vert ^{m+\left[  m\right]  +1}\left\vert N\right\vert ^{-1+m-\left[
m\right]  }\right]  .
\end{align*}
The expectation above is a constant $c_{m}^{\prime}$ depending only on $m$ as
long as $m$ is not an integer. The case of $m$ integer is trivial since then
we have $B^{\prime}=0$. Now we can use Cauchy-Schwarz's inequality to say that
$\left\vert \theta\right\vert ^{\left[  m\right]  }\leq\left(  \sigma
\tau\right)  ^{\left[  m\right]  }$, yielding%
\[
\left\vert \mathbf{E}\left[  B^{\prime}\right]  \right\vert \leq c_{m}%
c_{m}^{\prime}\rho^{m-\left[  m\right]  -1}\left(  \sigma\tau\right)
^{m-1}\left\vert \theta\right\vert ^{1}.
\]
This is again identical to the terms we have already dealt with, but for the
presence of the negative power on $\rho$. We will handle this complication by
showing that $\rho$ can be bounded below by a universal constant.

First note that integration on the $\varepsilon$-diagonal can be handled by
using the same argument as in Steps 1 and 2 of the proof of Theorem
\ref{HomogGauss}. Thus we can assume that $t\geq s+2\varepsilon$. Now that we
are off the diagonal, note that using the mean value theorem on the expression
for $\theta$ in (\ref{seeplanar}), we can write that $\theta=\varepsilon
^{2}\left(  \delta^{2}\right)  ^{\prime\prime}\left(  \xi\right)  $ for some
$\xi$ in $[t-s-\varepsilon,t-s+\varepsilon]$. At this point, Condition (S)
allows us to say first that $\left\vert \left(  \delta^{2}\right)
^{\prime\prime}\left(  r\right)  \right\vert \leq cr^{-2}\delta^{2}\left(
r\right)  $. We now use the expression $\sigma\tau=\delta^{2}\left(
\varepsilon\right)  $, and the fact that off the diagonal, $\xi>\varepsilon$,
combined with the fact that $\left\vert \left(  \delta^{2}\right)
^{\prime\prime}\right\vert $ is decreasing according to Condition (S),\ to
write $\left\vert \theta/\left(  \sigma\tau\right)  \right\vert \leq
\varepsilon^{2}\left\vert \left(  \delta^{2}\right)  ^{\prime\prime}\left(
\varepsilon\right)  \right\vert \delta^{-2}\left(  \varepsilon\right)  \leq
c$. Recalling the definition of $\rho:=\left(  1-\theta^{2}\left(  \sigma
\tau\right)  ^{-2}\right)  ^{1/2}$, we have proved that $\rho$ is bounded
below uniformly (off the $\varepsilon$-diagonal) by the positive constant
$c^{\prime\prime}:=\left(  1-c^{2}\right)  ^{1/2}$. Hence, in the inequality
for $\left\vert \mathbf{E}\left[  B^{\prime}\right]  \right\vert $ above, the
term $\rho^{m-\left[  m\right]  -1}$ can be absorbed into the $m$-dependent
constants. In other words, we have proved the upper bound $\left\vert
\mathbf{E}\left[  B^{\prime}\right]  \right\vert \leq c_{m}c_{m}^{\prime
}\left(  c^{\prime\prime}\right)  ^{m-[m]-1}\left(  \sigma\tau\right)
^{m-1}\left\vert \theta\right\vert ^{1}$, and we are back once again to the
situation solved in the proof of Theorem \ref{HomogGauss}, proving the
corresponding contribution of $B^{\prime}$ to $\mathbf{E}\left[  \left(
[X,m]_{\varepsilon}\left(  T\right)  \right)  ^{2}\right]  $ converges to $0$
as soon as $\delta^{2}\left(  r\right)  =o\left(  r^{1/\left(  2m\right)
}\right)  $.\vspace*{0.1in}

\noindent\emph{Step 4. The case of }$m$\emph{ integer}. Of course, we already
proved the theorem in the case $m$ odd. Now assume $m$ is an even integer. In
this special case, we do not need to use a Taylor expansion, since the
binomial formula has no remainder. Moreover, on the event $\left\vert
\frac{\rho\tau\sigma N}{\theta M}\right\vert <1$, $\mbox{sgn}\left(
a+b\right)  =\mbox{sgn}\left(  a\right)  $. Thus $A^{\prime}=0$ with the
understanding that we must replace $A$ by the full sum for $k=0$ to $m$.
Recalculating this $A$ we get%
\begin{align*}
A  &  =\sigma^{m}\mbox{sgn}\left(  \theta\right)  \mbox{sgn}\left(  M\right)
\left\vert M\right\vert ^{m}\mathbf{1}_{\left\vert \frac{\rho\tau\sigma
N}{\theta M}\right\vert <1}\sum_{k=0}^{m}\binom{m}{k}M^{k}N^{m-k}\left[
\frac{\theta}{\sigma}\right]  ^{k}\left[  \rho\tau\right]  ^{m-k}\\
&  =\mbox{sgn}\left(  \theta\right)  \mathbf{1}_{\left\vert \frac{\rho
\tau\sigma N}{\theta M}\right\vert <1}\sum_{k=0}^{m}\binom{m}{k}%
\mbox{sgn}^{k+1}\left(  M\right)  \left\vert M\right\vert ^{m+k}N^{m-k}%
\theta^{k}\rho^{m-k}\sigma^{m-k}\tau^{m-k}.
\end{align*}
Here, when we take the expectation $\mathbf{E}$, all terms vanish since we
have odd functions of $M$ for $k$ even thanks to the term $\mbox{sgn}^{k+1}%
\left(  M\right)  $, and odd functions of $N$ for $k$ odd thanks to the term
$N^{m-k}$. I.e. the term corresponding to $A$ is entirely null when $m$ is
even. The term $B^{\prime}$ is null since we have no error terms in the Taylor
expansion. The estimation of the term $B$ in Step 1.2 above applied when $m$
is an integer. The proof of the theorem is finished.
\end{proof}

\section{Non-Gaussian case\label{NGC}}

Now assume that $X$ is given by (\ref{defX}) and $M$ is a square-integrable
(non-Gaussian) martingale, $m$ is an odd integer, and define a positive
non-random measure $\mu$ for $\bar{s}=\left(  s_{1},s_{2},\cdots,s_{m}\right)
\in\lbrack0,T]^{m}$ by%
\begin{equation}
\mu\left(  d\bar{s}\right)  =\mu\left(  ds_{1}ds_{2}\cdots ds_{m}\right)
=\mathbf{E}\left[  d\left[  M\right]  \left(  s_{1}\right)  d\left[  M\right]
\left(  s_{2}\right)  \cdots d\left[  M\right]  \left(  s_{m}\right)  \right]
, \label{mu}%
\end{equation}
where $\left[  M\right]  $ is the quadratic variation process of $M$. We make
the following assumption on $\mu$.

\begin{description}
\item[(A)] The non-negative measure $\mu$ is absolutely continuous with
respect to the Lebesgue measure $d\bar{s}$ on $[0,T]^{m}$ and $K\left(
\bar{s}\right)  :=d\mu/d\bar{s}$ is bounded by a tensor-power function: $0\leq
K\left(  s_{1},s_{2},\cdots,s_{m}\right)  \leq\Gamma^{2}\left(  s_{1}\right)
\Gamma^{2}\left(  s_{2}\right)  \cdots\Gamma^{2}\left(  s_{m}\right)  $ for
some non-negative function $\Gamma$ on $[0,T]$.
\end{description}

A large class of processes satisfying (A) is the case where $M\left(
t\right)  =\int_{0}^{t}H\left(  s\right)  dW\left(  s\right)  $ where $H\in
L^{2}\left(  [0,T]\times\Omega\right)  $ and $W$ is a standard Wiener process,
and we assume $\mathbf{E}\left[  H^{2m}\left(  t\right)  \right]  $ is finite
for all $t\in\lbrack0,T]$. Indeed then by H\"{o}lder's inequality, since we
can take $K\left(  \bar{s}\right)  =\mathbf{E}\left[  H^{2}\left(
s_{1}\right)  H^{2}\left(  s_{2}\right)  \cdots H^{2}\left(  s_{m}\right)
\right]  $, we see that $\Gamma\left(  s\right)  =$ $\left(  \mathbf{E}\left[
H^{2m}\left(  t\right)  \right]  \right)  ^{1/\left(  2m\right)  }$ works.

We will show that the sufficient conditions for zero odd variation in the
Gaussian cases generalize to the case of condition (A), by associating $X$
with a Gaussian process. We let%
\[
\tilde{G}\left(  t,s\right)  =\Gamma\left(  s\right)  G\left(  t,s\right)
\]
and
\begin{equation}
Z\left(  t\right)  :=\int_{0}^{T}\tilde{G}\left(  t,s\right)  dW\left(
s\right)  . \label{Zee}%
\end{equation}
We have the following.

\begin{theorem}
\label{MartThm}Let $m$ be an odd integer $\geq3$. Let $X$ and $Z$ be as
defined in (\ref{defX}) and (\ref{Zee}). Assume $M$ satisfies condition
\emph{(A)} and $Z$ is well-defined and satisfies the hypotheses of Theorem
\ref{HomogGauss} or Theorem \ref{nonhomoggauss} relative to a univariate
function $\delta$. Assume that for some constant $c>0$, and every small
$\varepsilon>0$,%
\begin{equation}
\int_{t=2\varepsilon}^{T}dt\int_{s=0}^{t-2\varepsilon}ds\int_{u=0}%
^{T}\left\vert \Delta\tilde{G}_{t}\left(  u\right)  \right\vert \left\vert
\Delta\tilde{G}_{s}\left(  u\right)  \right\vert du\leq c\varepsilon\delta
^{2}\left(  2\varepsilon\right)  , \label{additional}%
\end{equation}
where we use the notation $\Delta\tilde{G}_{t}\left(  u\right)  =\tilde
{G}\left(  t+\varepsilon,u\right)  -\tilde{G}\left(  t,u\right)  $. Then $X$
has zero $m$th variation.
\end{theorem}

\begin{proof}

\noindent\emph{Step 0: setup.} We use an expansion for powers of martingales
written explicitly at Corollary 2.18 of \cite{RE}. For any integer
$k\in\lbrack0,\left[  m/2\right]  ]$, let $\Sigma_{m}^{k}$ be the set of
permutations $\sigma$ of $m-k$ defined as those for which the first $k$ terms
$\sigma^{-1}\left(  1\right)  ,\sigma^{-1}\left(  2\right)  ,\cdots
,\sigma^{-1}\left(  k\right)  $ are chosen arbitrarily and the next $m-2k$
terms are chosen arbitrarily among the remaining integers $\left\{
1,2,\cdots,m-k\right\}  \setminus\left\{  \sigma^{-1}\left(  1\right)
,\sigma^{-1}\left(  2\right)  ,\cdots,\sigma^{-1}\left(  k\right)  \right\}
$. Let $Y$ be a fixed square-integrable martingale. We define the process
$Y_{\sigma,\ell}$ (denoted in the above reference by $\sigma_{Y}^{\ell}$) by
setting, for each $\sigma\in\Sigma_{m}^{k}$ and each $\ell=1,2,\cdots,m-k$,%
\[
Y_{\sigma,\ell}\left(  t\right)  =\left\{
\begin{array}
[c]{c}%
\left[  Y\right]  \left(  t\right)  \ \mbox{ if }\ \sigma\left(  \ell\right)
\in\left\{  1,2,\cdots,k\right\} \\
Y\left(  t\right)  \ \mbox{ if }\ \sigma\left(  \ell\right)  \in\left\{
k+1,\cdots,m-k\right\}  .
\end{array}
\right.
\]
From Corollary 2.18 of \cite{RE}, we then have for all $t\in\lbrack0,T]$%
\[
\left(  Y_{t}\right)  ^{m}=\sum_{k=0}^{[m/2]}\frac{m!}{2^{k}}\sum_{\sigma
\in\Sigma_{m}^{k}}\int_{0}^{t}\int_{0}^{u_{m-k}}\cdots\int_{0}^{u_{2}%
}dY_{\sigma,1}\left(  u_{1}\right)  \ dY_{\sigma,2}\left(  u_{2}\right)
\cdots dY_{\sigma,m-k}\left(  u_{m-k}\right)  .
\]

We use this formula to evaluate%
\[
\lbrack X,m]_{\varepsilon}\left(  T\right)  =\frac{1}{\varepsilon}\int_{0}%
^{T}ds\left(  X\left(  s+\varepsilon\right)  -X\left(  s\right)  \right)  ^{m}%
\]
by noting that the increment $X\left(  s+\varepsilon\right)  -X\left(
s\right)  $ is the value at time $T$ of the martingale $Y_{t}:=\int_{0}%
^{t}\Delta G_{s}\left(  u\right)  dM\left(  u\right)  $ where we set%
\[
\Delta G_{s}\left(  u\right)  :=G\left(  s+\varepsilon,u\right)  -G\left(
s,u\right)  .
\]
Hence%
\begin{align*}
&  \left(  X\left(  s+\varepsilon\right)  -X\left(  s\right)  \right)  ^{m}\\
&  =\sum_{k=0}^{[m/2]}\frac{m!}{2^{k}}\sum_{\sigma\in\Sigma_{m}^{k}}\int
_{0}^{T}\int_{0}^{u_{m-k}}\cdots\int_{0}^{u_{2}}d\left[  M\right]  \left(
u_{\sigma\left(  1\right)  }\right)  \left\vert \Delta G_{s}\left(
u_{\sigma\left(  1\right)  }\right)  \right\vert ^{2}\cdots d\left[  M\right]
\left(  u_{\sigma\left(  k\right)  }\right)  \left\vert \Delta G_{s}\left(
u_{\sigma\left(  k\right)  }\right)  \right\vert ^{2}\\
&  dM\left(  u_{\sigma\left(  k+1\right)  }\right)  \Delta G_{s}\left(
u_{\sigma\left(  k+1\right)  }\right)  \cdots dM\left(  u_{\sigma\left(
m-k\right)  }\right)  \Delta G_{s}\left(  u_{\sigma\left(  m-k\right)
}\right)  .
\end{align*}
Therefore we can write%
\begin{align*}
&  [X,m]_{\varepsilon}\left(  T\right) \\
&  =\frac{1}{\varepsilon}\sum_{k=0}^{[m/2]}\frac{m!}{2^{k}}\sum_{\sigma
\in\Sigma_{m}^{k}}\int_{0}^{T}\int_{0}^{u_{m-k}}\cdots\int_{0}^{u_{2}}d\left[
M\right]  \left(  u_{\sigma\left(  1\right)  }\right)  \cdots d\left[
M\right]  \left(  u_{\sigma\left(  k\right)  }\right)  dM\left(
u_{\sigma\left(  k+1\right)  }\right)  \cdots dM\left(  u_{\sigma\left(
m-k\right)  }\right) \\
&  \left[  \Delta G_{\cdot}\left(  u_{\sigma\left(  k+1\right)  }\right)
;\cdots;\Delta G_{\cdot}\left(  u_{\sigma\left(  m-k\right)  }\right)  ;\Delta
G_{\cdot}\left(  u_{\sigma\left(  1\right)  }\right)  ;\Delta G_{\cdot}\left(
u_{\sigma\left(  1\right)  }\right)  ;\cdots;\Delta G_{\cdot}\left(
u_{\sigma\left(  k\right)  }\right)  ;\Delta G_{\cdot}\left(  u_{\sigma\left(
k\right)  }\right)  \right]  ,
\end{align*}
where we have used the notation
\[
\left[  f_{1},f_{2},\cdots,f_{m}\right]  :=\int_{0}^{T}f_{1}\left(  s\right)
f_{2}\left(  s\right)  \cdots f_{m}\left(  s\right)  ds.
\]
To calculate the expected square of the above, we will bound it above by the
sum over $k$ and $\sigma$ of the expected square of each term. Writing squares
of Lebesgue integrals as double integrals, and using It\^{o}'s formula, each
term's expected square is thus, up to $\left(  m,k\right)  $-dependent
multiplicative constants, equal to the expression%
\begin{align}
K  &  =\frac{1}{\varepsilon^{2}}\int_{u_{m-k}=0}^{T}\int_{u_{m-k}^{\prime}%
=0}^{T}\int_{u_{m-k-1}=0}^{u_{m-k}}\int_{u_{m-k-1}^{\prime}=0}^{u_{m-k}}%
\cdots\int_{u_{1}=0}^{u_{2}}\int_{u_{1}^{\prime}=0}^{u_{2}}\nonumber\\
&  \mathbf{E}\left[  d\left[  M\right]  ^{\otimes k}\left(  u_{\sigma\left(
1\right)  },\cdots,u_{\sigma\left(  k\right)  }\right)  d\left[  M\right]
^{\otimes k}\left(  u_{\sigma\left(  1\right)  }^{\prime},\cdots
,u_{\sigma\left(  k\right)  }^{\prime}\right)  d\left[  M\right]
^{\otimes\left(  m-2k\right)  }\left(  u_{\sigma\left(  k+1\right)  }%
,\cdots,u_{\sigma\left(  m-k\right)  }\right)  \right] \nonumber\\
&  \cdot\left[  \Delta G_{\cdot}\left(  u_{\sigma\left(  k+1\right)  }\right)
;\cdots;\Delta G_{\cdot}\left(  u_{\sigma\left(  m-k\right)  }\right)  ;\Delta
G_{\cdot}\left(  u_{\sigma\left(  1\right)  }\right)  ;\Delta G_{\cdot}\left(
u_{\sigma\left(  1\right)  }\right)  ;\cdots;\Delta G_{\cdot}\left(
u_{\sigma\left(  k\right)  }\right)  ;\Delta G_{\cdot}\left(  u_{\sigma\left(
k\right)  }\right)  \right] \nonumber\\
&  \cdot\left[  \Delta G_{\cdot}\left(  u_{\sigma\left(  k+1\right)  }\right)
;\cdots;\Delta G_{\cdot}\left(  u_{\sigma\left(  m-k\right)  }\right)  ;\Delta
G_{\cdot}\left(  u_{\sigma\left(  1\right)  }^{\prime}\right)  ;\Delta
G_{\cdot}\left(  u_{\sigma\left(  1\right)  }^{\prime}\right)  ;\cdots;\Delta
G_{\cdot}\left(  u_{\sigma\left(  k\right)  }^{\prime}\right)  ;\Delta
G_{\cdot}\left(  u_{\sigma\left(  k\right)  }^{\prime}\right)  \right]  ,
\label{crazycat}%
\end{align}
modulo the fact that one may remove the integrals with respect to those
$u_{j}^{\prime}$'s that are not represented among $\{u_{\sigma\left(
1\right)  }^{\prime},\cdots,u_{\sigma\left(  k\right)  }^{\prime}\}$. The
theorem will now be proved if we can show that for all $k\in\left\{
0,1,2,\cdots,\left[  m/2\right]  \right\}  $ and all $\sigma\in\Sigma_{m}^{k}%
$, the above expression $K=K_{m,k,\sigma}$tends to $0$ as $\varepsilon$ tends
to $0$.

A final note about notation. The bracket notation in the last two lines of the
expression (\ref{crazycat}) above means that we have the product of two
separate Riemann integrals over $s\in\lbrack0,T]$. Below we will denote these
integrals as being with respect to $s\in\lbrack0,T]$ and $t\in\lbrack
0,T]$.\bigskip

\noindent\emph{Step 1: diagonal.} As in Steps 1 of the proofs of Theorems
\ref{HomogGauss} and \ref{nonhomoggauss}, we can use brutal applications of
Cauchy-Schwarz's inequality to deal with the portion of $K_{m,k,\sigma}$ in
(\ref{crazycat}) where $\left\vert s-t\right\vert \leq2\varepsilon$. The
details are omitted.\bigskip

\noindent\emph{Step 2: term for }$k=0$. When $k=0$, there is only one
permutation $\sigma=Id$, and we have, using hypothesis (A)%
\begin{align*}
K_{m,0,Id}  &  =\frac{1}{\varepsilon^{2}}\int_{u_{m}=0}^{T}\int_{u_{m-1}%
=0}^{u_{m}}\cdots\int_{u_{1}=0}^{u_{2}}\mathbf{E}\left[  d\left[  M\right]
^{\otimes m}\left(  u_{1},\cdots,u_{m}\right)  \right]  \cdot\left[  \Delta
G_{\cdot}\left(  u_{1}\right)  ;\cdots;\Delta G_{\cdot}\left(  u_{m}\right)
\right]  ^{2}\\
&  \leq\frac{1}{\varepsilon^{2}}\int_{u_{m-k}=0}^{T}\int_{u_{m-k-1}%
=0}^{u_{m-k}}\cdots\int_{u_{1}=0}^{u_{2}}\Gamma^{2}\left(  u_{1}\right)
\Gamma^{2}\left(  u_{2}\right)  \cdots\Gamma^{2}\left(  u_{m}\right)  \left[
\Delta G_{\cdot}\left(  u_{1}\right)  ;\cdots;\Delta G_{\cdot}\left(
u_{m}\right)  \right]  ^{2}du_{1}du_{2}\cdots du_{m}\\
&  =\frac{1}{\varepsilon^{2}}\int_{u_{m-k}=0}^{T}\int_{u_{m-k-1}=0}^{u_{m-k}%
}\cdots\int_{u_{1}=0}^{u_{2}}\left[  \Delta\tilde{G}_{\cdot}\left(
u_{1}\right)  ;\cdots;\Delta\tilde{G}_{\cdot}\left(  u_{m}\right)  \right]
^{2}du_{1}du_{2}\cdots du_{m}.
\end{align*}
This is precisely the expression one gets for the term corresponding to $k=0$
when $M=W$, i.e. when $X$ is the Gaussian process $Z$ with kernel $\tilde{G}$.
Hence our hypotheses from the previous two theorems guarantee that this
expression tends to $0$.\bigskip

\noindent\emph{Step 3: term for }$k=1$. Again, in this case, there is only one
possible permutation, $\sigma=Id$, and we thus have, using hypothesis (A),
\begin{align*}
&  K_{m,1,Id}=\frac{1}{\varepsilon^{2}}\int_{u_{m-1}=0}^{T}\int_{u_{m-2}%
=0}^{u_{m-1}}\cdots\int_{u_{1}=0}^{u_{2}}\int_{u_{1}^{\prime}=0}^{u_{2}%
}\mathbf{E}\left[  d\left[  M\right]  \left(  u_{1}\right)  d\left[  M\right]
\left(  u_{1}^{\prime}\right)  d\left[  M\right]  ^{\otimes\left(  m-2\right)
}\left(  u_{2},\cdots,u_{m-1}\right)  \right] \\
&  \cdot\left[  \Delta G_{\cdot}\left(  u_{2}\right)  ;\cdots;\Delta G_{\cdot
}\left(  u_{m-1}\right)  ;\Delta G_{\cdot}\left(  u_{1}\right)  ;\Delta
G_{\cdot}\left(  u_{1}\right)  \right]  \cdot\left[  \Delta G_{\cdot}\left(
u_{2}\right)  ;\cdots;\Delta G_{\cdot}\left(  u_{m-1}\right)  ;\Delta
G_{\cdot}\left(  u_{1}^{\prime}\right)  ;\Delta G_{\cdot}\left(  u_{1}%
^{\prime}\right)  \right] \\
&  \leq\frac{1}{\varepsilon^{2}}\int_{u_{m-1}=0}^{T}\int_{u_{m-2}=0}^{u_{m-1}%
}\cdots\int_{u_{1}=0}^{u_{2}}\int_{u_{1}^{\prime}=0}^{u_{2}}du_{1}%
du_{1}^{\prime}du_{2}\cdots du_{m-1}\Gamma^{2}\left(  u_{1}\right)  \Gamma
^{2}\left(  u_{1}^{\prime}\right)  \Gamma^{2}\left(  u_{2}\right)
\cdots\Gamma^{2}\left(  u_{m}\right) \\
&  \cdot\left[  \left\vert \Delta G\right\vert _{\cdot}\left(  u_{2}\right)
;\cdots;\left\vert \Delta G\right\vert _{\cdot}\left(  u_{m-1}\right)
;\left\vert \Delta G\right\vert _{\cdot}\left(  u_{1}\right)  ;\left\vert
\Delta G\right\vert _{\cdot}\left(  u_{1}\right)  \right]  \cdot\left[
\left\vert \Delta G\right\vert _{\cdot}\left(  u_{2}\right)  ;\cdots
;\left\vert \Delta G\right\vert _{\cdot}\left(  u_{m-1}\right)  ;\left\vert
\Delta G\right\vert _{\cdot}\left(  u_{1}^{\prime}\right)  ;\left\vert \Delta
G\right\vert _{\cdot}\left(  u_{1}^{\prime}\right)  \right] \\
&  =\frac{1}{\varepsilon^{2}}\int_{u_{m-1}=0}^{T}\int_{u_{m-2}=0}^{u_{m-1}%
}\cdots\int_{u_{1}=0}^{u_{2}}\int_{u_{1}^{\prime}=0}^{u_{2}}du_{1}%
du_{1}^{\prime}du_{2}\cdots du_{m-1}\left[  \left\vert \Delta\tilde
{G}\right\vert _{\cdot}\left(  u_{2}\right)  ;\cdots;\left\vert \Delta
\tilde{G}\right\vert _{\cdot}\left(  u_{m-1}\right)  ;\left\vert \Delta
\tilde{G}\right\vert _{\cdot}\left(  u_{1}\right)  ;\left\vert \Delta\tilde
{G}\right\vert _{\cdot}\left(  u_{1}\right)  \right] \\
&  \cdot\left[  \left\vert \Delta\tilde{G}\right\vert _{\cdot}\left(
u_{2}\right)  ;\cdots;\left\vert \Delta\tilde{G}\right\vert _{\cdot}\left(
u_{m-1}\right)  ;\left\vert \Delta\tilde{G}\right\vert _{\cdot}\left(
u_{1}^{\prime}\right)  ;\left\vert \Delta\tilde{G}\right\vert _{\cdot}\left(
u_{1}^{\prime}\right)  \right]
\end{align*}
Note now that the product of two bracket operators $\left[  \cdots\right]
\left[  \cdots\right]  $ means we integrate over $0\leq s\leq t-2\varepsilon$
and $2\varepsilon\leq t\leq T$, and get an additional factor of $2$, since the
diagonal term was dealt with in Step 1.

In order to exploit the additional hypothesis (\ref{additional}) in our
theorem, our first move is to use Fubini by bringing the integrals over
$u_{1}$ all the way inside. We get%
\begin{align*}
K_{m,1,Id}  &  \leq\frac{2}{\varepsilon^{2}}\int_{u_{m-1}=0}^{T}\int
_{u_{m-2}=0}^{u_{m-1}}\cdots\int_{u_{2}=0}^{u_{3}}du_{2}\cdots du_{m-1}\\
&  \int_{t=2\varepsilon}^{T}\int_{s=0}^{t-2\varepsilon}ds\ dt\left\vert
\Delta\tilde{G}_{s}\left(  u_{2}\right)  \right\vert \cdots\left\vert
\Delta\tilde{G}_{s}\left(  u_{m-1}\right)  \right\vert \left\vert \Delta
\tilde{G}_{t}\left(  u_{2}\right)  \right\vert \cdots\left\vert \Delta
\tilde{G}_{t}\left(  u_{m-1}\right)  \right\vert \\
&  \int_{u_{1}=0}^{u_{2}}\int_{u_{1}^{\prime}=0}^{u_{2}}du_{1}du_{1}^{\prime
}\left(  \Delta\tilde{G}_{s}\left(  u_{1}\right)  \right)  ^{2}\left(
\Delta\tilde{G}_{t}\left(  u_{1}^{\prime}\right)  \right)  ^{2}.
\end{align*}
The term in the last line above is trivially bounded above by%
\[
\int_{u_{1}=0}^{T}\int_{u_{1}^{\prime}=0}^{T}du_{1}du_{1}^{\prime}\left(
\Delta\tilde{G}_{s}\left(  u_{1}\right)  \right)  ^{2}\left(  \Delta\tilde
{G}_{t}\left(  u_{1}^{\prime}\right)  \right)  ^{2}%
\]
precisely equal to $Var\left[  Z\left(  s+\varepsilon\right)  -Z\left(
s\right)  \right]  \ Var\left[  Z\left(  t+\varepsilon\right)  -Z\left(
t\right)  \right]  $, which by hypothesis is bounded above by $\delta
^{4}\left(  \varepsilon\right)  $. Consequently, we get%
\begin{align*}
K_{m,1,Id}  &  \leq2\frac{\delta^{4}\left(  \varepsilon\right)  }%
{\varepsilon^{2}}\int_{u_{m-1}=0}^{T}\int_{u_{m-2}=0}^{u_{m-1}}\cdots
\int_{u_{2}=0}^{u_{3}}du_{2}\cdots du_{m-1}\\
&  \int_{t=2\varepsilon}^{T}\int_{s=0}^{t-2\varepsilon}ds\ dt\left\vert
\Delta\tilde{G}_{s}\left(  u_{2}\right)  \right\vert \cdots\left\vert
\Delta\tilde{G}_{s}\left(  u_{m-1}\right)  \right\vert \left\vert \Delta
\tilde{G}_{t}\left(  u_{2}\right)  \right\vert \cdots\left\vert \Delta
\tilde{G}_{t}\left(  u_{m-1}\right)  \right\vert .
\end{align*}
We get an upper bound by integrating all the $u_{j}$'s over their entire range
$[0,T]$. I.e. we have,%
\begin{align*}
&  K_{m,1,Id}\leq\frac{\delta^{4}\left(  \varepsilon\right)  }{\varepsilon
^{2}}\int_{t=2\varepsilon}^{T}dt\int_{s=0}^{t-2\varepsilon}ds\\
&  \int_{0}^{T}\int_{0}^{T}\cdots\int_{0}^{T}du_{3}\cdots du_{m-1}\left\vert
\Delta\tilde{G}_{s}\left(  u_{3}\right)  \right\vert \cdots\left\vert
\Delta\tilde{G}_{s}\left(  u_{m-1}\right)  \right\vert \left\vert \Delta
\tilde{G}_{t}\left(  u_{3}\right)  \right\vert \cdots\left\vert \Delta
\tilde{G}_{t}\left(  u_{m-1}\right)  \right\vert \\
&  \cdot\int_{u_{2}=0}^{T}\left\vert \Delta\tilde{G}_{t}\left(  u_{2}\right)
\right\vert \left\vert \Delta\tilde{G}_{s}\left(  u_{2}\right)  \right\vert
du_{2}\\
&  =2\frac{\delta^{4}\left(  \varepsilon\right)  }{\varepsilon^{2}}%
\int_{t=2\varepsilon}^{T}dt\int_{s=0}^{t-2\varepsilon}ds\left(  \int_{0}%
^{T}du\left\vert \Delta\tilde{G}_{s}\left(  u\right)  \right\vert \left\vert
\Delta\tilde{G}_{t}\left(  u\right)  \right\vert \right)  ^{m-3}\cdot
\int_{u_{2}=0}^{u_{3}}\left\vert \Delta\tilde{G}_{t}\left(  u_{2}\right)
\right\vert \left\vert \Delta\tilde{G}_{s}\left(  u_{2}\right)  \right\vert
du_{2}..
\end{align*}
Now we use a simple Cauchy-Schwarz inequality for the integral over $u$, but
not for $u_{2}$. Recognizing that $\int_{0}^{T}\left\vert \Delta\tilde{G}%
_{s}\left(  u\right)  \right\vert ^{2}du$ is the variance $Var\left[  Z\left(
s+\varepsilon\right)  -Z\left(  s\right)  \right]  \leq\delta^{2}\left(
\varepsilon\right)  $, we have%
\begin{align*}
K_{m,1,Id}  &  \leq2\frac{\delta^{4}\left(  \varepsilon\right)  }%
{\varepsilon^{2}}\int_{t=2\varepsilon}^{T}dt\int_{s=0}^{t-2\varepsilon
}ds\left(  \int_{0}^{T}du\left\vert \Delta\tilde{G}_{s}\left(  u\right)
\right\vert ^{2}\right)  ^{m-3}\cdot\int_{u_{2}=0}^{u_{3}}\left\vert
\Delta\tilde{G}_{t}\left(  u_{2}\right)  \right\vert \left\vert \Delta
\tilde{G}_{s}\left(  u_{2}\right)  \right\vert du_{2}.\\
&  \leq2\frac{\delta^{4+2m-6}\left(  \varepsilon\right)  }{\varepsilon^{2}%
}\int_{t=2\varepsilon}^{T}dt\int_{s=0}^{t-2\varepsilon}ds\int_{u_{2}=0}%
^{T}\left\vert \Delta\tilde{G}_{t}\left(  u_{2}\right)  \right\vert \left\vert
\Delta\tilde{G}_{s}\left(  u_{2}\right)  \right\vert du_{2}.
\end{align*}
Condition (\ref{additional}) implies immediately $K_{m,1,Id}\leq\delta
^{2m}\left(  2\varepsilon\right)  \varepsilon^{-1}$ which tends to $0$ with
$\varepsilon$ by hypothesis.\bigskip

\noindent\emph{Step 4: }$k\geq2$. This step proceeds using the same technique
as Step 3. Fix $k\geq2$. Now for each given permutation $\sigma$, there are
$k$ pairs of parameters of the type $\left(  u,u^{\prime}\right)  $. Each of
these contributes precisely a term $\delta^{4}\left(  \varepsilon\right)  $,
as in the previous step, i.e. $\delta^{4k}\left(  \varepsilon\right)  $
altogether. In other words, for every $\sigma\in\Sigma_{m}^{k}$, and deleting
the diagonal term, we have
\begin{align*}
&  K_{m,k,\sigma}\\
&  \leq2\frac{\delta^{4k}\left(  \varepsilon\right)  }{\varepsilon^{2}}%
\int_{t=2\varepsilon}^{T}dt\int_{s=0}^{t-2\varepsilon}ds\int_{0}^{T}\int
_{0}^{u_{m-k}}\cdots\int_{0}^{u_{k+2}}du_{k+1}\cdots du_{m-k}\left[  \int
_{0}^{T}ds\left\vert \Delta\tilde{G}_{s}\left(  u_{k+1}\right)  \right\vert
\cdots\left\vert \Delta\tilde{G}_{s}\left(  u_{m-k}\right)  \right\vert
\right]  ^{2}.
\end{align*}
Since $k\leq\left(  m-1\right)  /2$, there is at least one integral, the one
in $u_{k+1}$, above. We treat all the remaining integrals, if any, over
$u_{k+2},\cdots,u_{m-k}$ with Cauchy-Schwarz's inequality as in Step 3,
yielding a contribution $\delta^{2\left(  m-2k-1\right)  }\left(
\varepsilon\right)  $. The remaining integral over $u_{k+1}$ yields, by
Condition (\ref{additional}), a contribution of $\delta^{2}\left(
2\varepsilon\right)  \varepsilon$. Hence the contribution of $K_{m,k,\sigma}$
is again $\delta^{2m}\left(  2\varepsilon\right)  \varepsilon^{-1}$, which
tends to $0$ with $\varepsilon$ by hypothesis, concluding the proof of the Theorem.
\end{proof}

\bigskip

We state and prove the next proposition, in order to illustrate the range of
applicability of Theorem \ref{MartThm}. It provides a class of
martingale-based processes $X$ which can be associated to non-homogeneous
Gaussian processes $Z$ satisfying the assumptions of Theorem
\ref{nonhomoggauss} and the additional assumption (\ref{additional}).

\begin{proposition}
\label{Ex}Let $X$ be defined by (\ref{defX}) via the kernel $G$ and the
martingale $M$. Assume $m$ is an odd integer $\geq3$ and condition \emph{(A)}
holds. Assume that $\tilde{G}\left(  t,s\right)  :=$ $\Gamma\left(  s\right)
G\left(  t,s\right)  $ can be bounded above as follows: for all $s,t$,%
\[
\tilde{G}\left(  t,s\right)  =\mathbf{1}_{s\leq t}\ g\left(  t,s\right)
=\mathbf{1}_{s\leq t}\left\vert t-s\right\vert ^{1/\left(  2m\right)
-1/2}f\left(  t,s\right)
\]
in which the bivariate function $f\left(  t,s\right)  $ is positive and
bounded as%
\[
\left\vert f\left(  t,s\right)  \right\vert \leq f\left(  \left\vert
t-s\right\vert \right)
\]
where the univariate function $f\left(  r\right)  $ is increasing, and concave
on $\mathbf{R}_{+}$, with $\lim_{r\rightarrow0}f\left(  r\right)  =0$, and
where $g$ has a second mixed derivative such that%
\begin{align*}
\left\vert \frac{\partial g}{\partial t}\left(  t,s\right)  \right\vert
+\left\vert \frac{\partial g}{\partial s}\left(  t,s\right)  \right\vert  &
\leq c\left\vert t-s\right\vert ^{1/\left(  2m\right)  -3/2};\\
\left\vert \frac{\partial^{2}g}{\partial s\partial t}\left(  t,s\right)
\right\vert  &  \leq c\left\vert t-s\right\vert ^{1/\left(  2m\right)  -5/2}.
\end{align*}
Also assume that $g$ is decreasing in $t$ and the bivariate $f$ is increasing
in $t$. Then $X$ has zero $m$-variation.
\end{proposition}

The presence of the indicator function $\mathbf{1}_{s\leq t}$ in the
expression for $\tilde{G}$ above is typical of most models, since it coincides
with asking that $Z$ be adapted to the filtrations of $W$, which is equivalent
to $X$ being adapted to the filtration of $M$. In the case of irregular
processes, which is the focus of this paper, the presence of the indicator
function makes $\tilde{G}$ non-monotone in both $s$ and $t$, which creates
technical difficulties. Examples of non-adapted irregular processes are easier
to treat, since it is possible to require that $\tilde{G}$ be monotone. We do
not consider such non-adapted processes further. Specific examples of adapted
processes which fall in the class defined in the above proposition are given
below, after the proposition's proof.

\begin{proof}
[Proof of Proposition \ref{Ex}]Below the value $1/\left(  2m\right)  -1/2$ is
denoted by $\alpha$. We now show that we can apply Theorem \ref{nonhomoggauss}
directly to the Gaussian process $Z$ given in (\ref{Zee}), which, by Theorem
\ref{MartThm}, is sufficient, together with Condition (\ref{additional}), to
obtain our desired conclusion. Note the assumption about $\tilde{G}$ implies
that $s\mapsto\tilde{G}\left(  t,s\right)  $ is square-integrable, and
therefore $Z$ is well-defined. We will prove Condition (\ref{nhdeltacond})
holds in Step 1; Step 2 will show Condition (\ref{nhmubirdiecond}) holds;
Condition (\ref{additional}) will be established in Step 3.\vspace{0.1in}

\noindent\emph{Step 1. Variance calculation. }We need only to show
$\tilde{\delta}^{2}\left(  s,s+\varepsilon\right)  =o\left(  \varepsilon
^{1/m}\right)  $ uniformly in $s$. We have, for given $s$ and $t=s+\varepsilon
$%
\begin{align}
\tilde{\delta}^{2}\left(  s,s+\varepsilon\right)   &  =\int_{0}^{t}\left\vert
\tilde{G}\left(  t,r\right)  -\tilde{G}\left(  s,r\right)  \right\vert
^{2}dr\nonumber\\
&  =\int_{0}^{s}\left\vert \left(  s+\varepsilon-r\right)  ^{\alpha}f\left(
s+\varepsilon,r\right)  -\left(  s-r\right)  ^{\alpha}f\left(  s,r\right)
\right\vert ^{2}dr\nonumber\\
&  +\int_{s}^{s+\varepsilon}\left\vert s+\varepsilon-r\right\vert ^{2\alpha
}f^{2}\left(  s+\varepsilon,r\right)  dr\label{delta2withf}\\
&  =:A+B.\nonumber
\end{align}
Since $f^{2}\left(  s+\varepsilon,r\right)  \leq f\left(  s+\varepsilon
-r\right)  $ and the univariate $f$ increases, in $B$ we can bound this last
quantity by $f\left(  \varepsilon\right)  $, and we get
\[
B\leq f^{2}\left(  \varepsilon\right)  \int_{0}^{\varepsilon}r^{2\alpha
}dr=3f^{2}\left(  \varepsilon\right)  \varepsilon^{2\alpha+1}=o\left(
\varepsilon^{1/m}\right)  .
\]

The term $A$ is slightly more delicate to estimate. By the fact that $f$ is
increasing and $g$ is decreasing in $t$,%
\begin{align*}
A  &  \leq\int_{0}^{s}f^{2}\left(  s+\varepsilon,r\right)  \left\vert \left(
s+\varepsilon-r\right)  ^{\alpha}-\left(  s-r\right)  ^{\alpha}\right\vert
^{2}dr=\int_{0}^{s}f^{2}\left(  \varepsilon+r\right)  \left\vert r^{\alpha
}-\left(  r+\varepsilon\right)  ^{\alpha}\right\vert ^{2}dr\\
&  =\int_{0}^{\varepsilon}f^{2}\left(  \varepsilon+r\right)  \left\vert
r^{\alpha}-\left(  r+\varepsilon\right)  ^{\alpha}\right\vert ^{2}%
dr+\int_{\varepsilon}^{s}f^{2}\left(  \varepsilon+r\right)  \left\vert
r^{\alpha}-\left(  r+\varepsilon\right)  ^{\alpha}\right\vert ^{2}dr\\
&  =:A_{1}+A_{2}.
\end{align*}
We have, again from the univariate $f$'s increasingness, and the limit
$\lim_{r\rightarrow0}f\left(  r\right)  =0$,%
\[
A_{1}\leq f^{2}\left(  2\varepsilon\right)  \int_{0}^{\varepsilon}\left\vert
r^{\alpha}-\left(  r+\varepsilon\right)  ^{\alpha}\right\vert ^{2}dr=cst\cdot
f^{2}\left(  2\varepsilon\right)  \varepsilon^{2\alpha+1}=o\left(
\varepsilon^{1/m}\right)  .
\]
For the other part of $A$, we need to use $f$'s concavity at the point
$2\varepsilon$ in the interval $[0,\varepsilon+r]$ (since $\varepsilon
+r>2\varepsilon$ in this case), which implies $f\left(  \varepsilon+r\right)
<f\left(  2\varepsilon\right)  \left(  \varepsilon+r\right)  /\left(
2\varepsilon\right)  $. Also using the mean-value theorem for the difference
of negative cubes, we get%
\begin{align*}
A_{2}  &  \leq cst\cdot\varepsilon^{2}\int_{\varepsilon}^{s}f^{2}\left(
\varepsilon+r\right)  r^{2\alpha-2}dr\leq cst\cdot\varepsilon f\left(
2\varepsilon\right)  \int_{\varepsilon}^{s}\left(  \varepsilon+r\right)
r^{2\alpha-2}dr\\
&  \leq cst\cdot\varepsilon f\left(  2\varepsilon\right)  \int_{\varepsilon
}^{s}r^{2\alpha-1}=cst\cdot\varepsilon^{2\alpha+1}f\left(  2\varepsilon
\right)  =o\left(  \varepsilon^{1/3}\right)  .
\end{align*}
This finishes the proof of Condition (\ref{nhdeltacond}).\vspace{0.1in}

\noindent\emph{Step 2. Covariance calculation}. We first calculate the second
mixed derivative $\partial^{2}\tilde{\delta}^{2}/\partial s\partial t$, where
$\tilde{\delta}$ is the canonical metric of $Z$, because we must show
$\left\vert \mu\right\vert \left(  OD\right)  \leq\varepsilon^{2\alpha}$,
which is condition (\ref{nhmubirdiecond}), and $\mu\left(  dsdt\right)
=ds\ dt\ \partial^{2}\tilde{\delta}^{2}/\partial s\partial t$. We have, for
$0\leq s\leq t-\varepsilon$,%
\begin{align*}
\tilde{\delta}^{2}\left(  s,t\right)   &  =\int_{0}^{s}\left(  g\left(
t,s-r\right)  -g\left(  s,s-r\right)  \right)  ^{2}dr+\int_{s}^{t}g^{2}\left(
t,r\right)  dr\\
&  =:A+B.
\end{align*}
We calculate%
\begin{align*}
\frac{\partial^{2}A}{\partial s\partial t}\left(  t,s\right)   &
=2\frac{\partial g}{\partial t}\left(  t,0\right)  \left(  g\left(
t,0\right)  -g\left(  s,0\right)  \right) \\
&  +\int_{0}^{s}2\frac{\partial g}{\partial t}\left(  t,s-r\right)  \left(
\frac{\partial g}{\partial s}\left(  t,s-r\right)  -\frac{\partial g}{\partial
t}\left(  s,s-r\right)  -\frac{\partial g}{\partial s}\left(  s,s-r\right)
\right) \\
&  +\int_{0}^{s}2\left(  g\left(  t,s-r\right)  -g\left(  s,s-r\right)
\right)  \frac{\partial^{2}g}{\partial s\partial t}\left(  t,s-r\right)  dr.\\
&  =A_{1}+A_{2}+A_{3},
\end{align*}
and
\[
\frac{\partial^{2}B}{\partial s\partial t}\left(  t,s\right)  =-2g\left(
t,s\right)  \frac{\partial g}{\partial t}\left(  t,s\right)  .
\]

Next, we immediately get, for the portion of $\left\vert \mu\right\vert
\left(  OD\right)  $ corresponding to $B$, using the hypotheses of our
proposition,%
\begin{align*}
\int_{\varepsilon}^{T}dt\int_{0}^{t-\varepsilon}ds\left\vert \frac
{\partial^{2}B}{\partial s\partial t}\left(  t,s\right)  \right\vert  &
\leq2c\int_{\varepsilon}^{T}dt\int_{0}^{t-\varepsilon}dsf\left(  \left\vert
t-s\right\vert \right)  \left\vert t-s\right\vert ^{\alpha}\left\vert
t-s\right\vert ^{\alpha-1}\\
&  \leq2c\left\Vert f\right\Vert _{\infty}\int_{\varepsilon}^{T}%
dt\ \varepsilon^{2\alpha}=cst\cdot\varepsilon^{2\alpha},
\end{align*}
which is of the correct order for Condition (\ref{nhmubirdiecond}). For the
term corresponding to $A_{1}$, using our hypotheses, we have%
\[
\int_{\varepsilon}^{T}dt\int_{0}^{t-\varepsilon}ds\left\vert A_{1}\right\vert
\leq2\int_{\varepsilon}^{T}dt\int_{0}^{t-\varepsilon}ds\ t^{\alpha}\left\vert
\frac{\partial g}{\partial t}\left(  \xi_{t,s},0\right)  \right\vert
\left\vert t-s\right\vert
\]
where $\xi_{t,s}$ is in the interval $\left(  s,t\right)  $. Our hypothesis
thus implies $\left\vert \frac{\partial g}{\partial t}\left(  \xi
_{t,s},0\right)  \right\vert \leq s^{\alpha}$, and hence%
\[
\int_{\varepsilon}^{T}dt\int_{0}^{t-\varepsilon}ds\left\vert A_{1}\right\vert
\leq2T\int_{\varepsilon}^{T}dt\int_{0}^{t-\varepsilon}ds\ s^{\alpha
-1}t^{\alpha-1}=2T\alpha^{-1}\int_{\varepsilon}^{T}dt\ t^{\alpha-1}\left(
t-\varepsilon\right)  ^{\alpha}\leq\alpha^{-2}T^{1+2\alpha}.
\]
This is much smaller than the right-hand side $\varepsilon^{2\alpha}$ of
Condition (\ref{nhmubirdiecond}), since $2\alpha=1/m-1<0$. The terms $A_{2}$
and $A_{3}$ are treated similarly, thanks to our hypotheses.\vspace{0.1in}

\noindent\emph{Step 3: proving Condition (\ref{additional}).} In fact, we
modify the proof of Theorem \ref{MartThm}, in particular Steps 3 and 4, so
that we only need to prove%
\begin{equation}
\int_{t=2\varepsilon}^{T}dt\int_{s=0}^{t-2\varepsilon}ds\int_{u=0}%
^{T}\left\vert \Delta\tilde{G}_{t}\left(  u\right)  \right\vert \left\vert
\Delta\tilde{G}_{s}\left(  u\right)  \right\vert du\leq c\varepsilon
^{2+2\alpha}=c\varepsilon^{1/m+1}, \label{additional2}%
\end{equation}
instead of Condition (\ref{additional}). Indeed, for instance in Step 3, this
new condition yields a final contribution of order $\delta^{2m-2}\left(
\varepsilon\right)  \varepsilon^{-2}\varepsilon^{1/m+1}$. With the assumption
on $\delta$ that we have, $\delta\left(  \varepsilon\right)  =o\left(
\varepsilon^{1/\left(  2m\right)  }\right)  $, and hence the final
contribution is of order $o\left(  \varepsilon^{\left(  2m-2\right)  /\left(
2m\right)  -1+1/m}\right)  =o\left(  1\right)  $. This proves that the
conclusion of Theorem \ref{MartThm} holds if we assume (\ref{additional2})
instead of Condition (\ref{additional}).

We now prove (\ref{additional2}). We can write%
\begin{align*}
&  \int_{t=2\varepsilon}^{T}dt\int_{s=0}^{t-2\varepsilon}ds\int_{u=0}%
^{T}\left\vert \Delta\tilde{G}_{t}\left(  u\right)  \right\vert \left\vert
\Delta\tilde{G}_{s}\left(  u\right)  \right\vert du\\
&  =\int_{t=2\varepsilon}^{T}dt\int_{s=0}^{t-2\varepsilon}ds\int_{0}%
^{s}\left\vert g\left(  t+\varepsilon,u\right)  -g\left(  t,u\right)
\right\vert \left\vert g\left(  s+\varepsilon,u\right)  -g\left(  s,u\right)
\right\vert du\\
&  +\int_{t=2\varepsilon}^{T}dt\int_{s=0}^{t-2\varepsilon}ds\int
_{s}^{s+\varepsilon}\left\vert g\left(  t+\varepsilon,u\right)  -g\left(
t,u\right)  \right\vert \left\vert g\left(  s+\varepsilon,u\right)
\right\vert du\\
&  =:A+B.
\end{align*}

For $A$, we use the hypotheses of this proposition: for the last factor in
$A$, we exploit the fact that $g$ is decreasing in $t$ while $f$ is increasing
in $t$; for the other factor in\ $A$, use the bound on $\partial g/\partial
t$; thus we have%
\[
A\leq\int_{t=2\varepsilon}^{T}dt\int_{s=0}^{t-2\varepsilon}\varepsilon
\left\vert t-s\right\vert ^{\alpha-1}ds\int_{0}^{s}f\left(  s+\varepsilon
,u\right)  \left(  \left(  s-u\right)  ^{\alpha}-\left(  s+\varepsilon
-u\right)  ^{\alpha}\right)  du.
\]
We separate the integral in $u$ into two pieces, for $u\in\lbrack
0,s-\varepsilon]$ and $u\in\lbrack s-\varepsilon,s]$. For the first integral
in $u$, since $f$ is bounded, we have%
\[
\int_{0}^{s-\varepsilon}f\left(  s+\varepsilon,u\right)  \left(  \left(
s-u\right)  ^{\alpha}-\left(  s+\varepsilon-u\right)  ^{\alpha}\right)
du\leq\left\Vert f\right\Vert _{\infty}\varepsilon\int_{0}^{s-\varepsilon
}\left(  s-u\right)  ^{\alpha-1}du\leq\left\Vert f\right\Vert _{\infty
}c_{\alpha}\varepsilon^{1+\alpha}.
\]
For the second integral in $u$, we use the fact that $s-u+\varepsilon
>\varepsilon$ and $s-u<\varepsilon$ implies $s-u+\varepsilon>2\left(
s-u\right)  $, so that the negative part of the integral can be ignored, and
thus%
\[
\int_{s-\varepsilon}^{s}f\left(  s+\varepsilon,u\right)  \left(  \left(
s-u\right)  ^{\alpha}-\left(  s+\varepsilon-u\right)  ^{\alpha}\right)
du\leq\left\Vert f\right\Vert _{\infty}\int_{s-\varepsilon}^{s}\left(
s-u\right)  ^{\alpha}du=\left\Vert f\right\Vert _{\infty}c_{\alpha}%
\varepsilon^{1+\alpha},
\]
which is the same upper bound as for the other part of the integral in $u$.
Thus%
\[
A\leq cst\cdot\varepsilon^{2+\alpha}\int_{t=2\varepsilon}^{T}dt\int
_{s=0}^{t-2\varepsilon}\left\vert t-s\right\vert ^{\alpha-1}ds\leq
cst\cdot\varepsilon^{2+\alpha}\int_{t=2\varepsilon}^{T}dt\ \varepsilon
^{\alpha}\leq cst\cdot\varepsilon^{2+2\alpha}=cst\cdot\varepsilon^{1/m+1},
\]
which is the conclusion we needed at least for $A$.

Lastly, we estimate $B$. We use the fact that $f$ is bounded, and thus
$\left\vert g\left(  s+\varepsilon,u\right)  \right\vert \leq\left\Vert
f\right\Vert _{\infty}\left\vert s+\varepsilon-u\right\vert ^{\alpha}$, as
well as the estimate on the derivative of $g$ as we did in the calculation of
$A$, yielding
\begin{align*}
B  &  \leq\left\Vert f\right\Vert _{\infty}\varepsilon\int_{t=2\varepsilon
}^{T}dt\int_{s=0}^{t-2\varepsilon}ds\left\vert t-s-\varepsilon\right\vert
^{\alpha-1}\int_{s}^{s+\varepsilon}\left\vert s+\varepsilon-u\right\vert
^{\alpha}du\\
&  =cst\cdot\varepsilon^{\alpha+2}\int_{t=2\varepsilon}^{T}dt\int
_{s=0}^{t-2\varepsilon}ds\left\vert t-s-\varepsilon\right\vert ^{\alpha-1}\\
&  \leq2^{1+\left\vert \alpha\right\vert }cst\cdot\varepsilon^{\alpha+2}%
\int_{t=2\varepsilon}^{T}dt\int_{s=0}^{t-2\varepsilon}ds\left\vert
t-s\right\vert ^{\alpha-1}\leq cst\cdot\varepsilon^{2\alpha+2}=cst\cdot
\varepsilon^{1/m+1}.
\end{align*}
This is the conclusion we needed for $B,$which finishes the proof of the proposition.
\end{proof}

\vspace{0.1in}

\label{pex}The above proposition covers a wide variety of martingale-based
models, which can be quite far from Gaussian models in the sense that they can
have only a few moments. We describe one easily constructed class. Assume that
$M$ is a martingale such that $\mathbf{E}\left[  \left\vert d\left[  M\right]
/dt\right\vert ^{2m}\right]  $ is bounded above by a constant $c^{2m}$
uniformly in $t\leq T$. This uniform boundedness assumption implies that we
can take $\Gamma\equiv c$ in Condition (A). In particular, $\tilde{G}$ can be
chosen to be proportional to $G$. Let $G\left(  t,s\right)  =G_{RLfBm}\left(
t,s\right)  :=\mathbf{1}_{s\leq t}\left\vert t-s\right\vert ^{1/\left(
2m\right)  -1/2+\alpha}$ for some $\alpha>0$; in other words, $G$ is the
Brownian representation kernel of the Riemann-Liouville fractional Brownian
motion with parameter $H=1/\left(  2m\right)  -\alpha>1/\left(  2m\right)  $.
It is immediate to check that the assumptions of Proposition \ref{Ex} are
satisfied for this class of martingale-based models, which implies that the
corresponding $X$ defined by (\ref{defX}) have zero $m$th variation.

More generally, assume that $G$ is bounded above by a multiple of $G_{RLfBm}$,
and assume the two partial derivatives of $G$, and the mixed second order
derivative of $G$, are bounded by the corresponding (multiples of) derivatives
of $G_{RLfBm}$; one can check that the standard fBm's kernel is in this class,
and that the martingale-based models of this class also satisfy the
assumptions of Proposition \ref{Ex}, resulting again zero $m$th variations for
the corresponding $X$ defined in (\ref{defX}). For the sake of conciseness, we
will omit the details, which are tedious and straightforward.\bigskip

The most quantitatively significant condition in Theorem \ref{MartThm}, that
the univariate function $\delta\left(  \varepsilon\right)  $ corresponding to
$\tilde{G}$ be equal to $o\left(  \varepsilon^{1/\left(  2m\right)  }\right)
$, can be interpreted as a regularity condition. In the Gaussian case, it
means that there is a function $f\left(  \varepsilon\right)  =o\left(
\varepsilon^{1/\left(  2m\right)  }\log^{1/2}\left(  1/\varepsilon\right)
\right)  $ such that $f$ is an almost-sure uniform modulus of continuity for
$X$. In non-Gaussian cases, similar interpretations can be given for the
regularity of $X$ itself, provided enough moments of $X$ exist. If $X$ has
fractional exponential moments, in the sense that for some constants
$c>0,0<\beta\leq2$, $\mathbf{E}\left[  \exp\left(  c\left\vert X\left(
t\right)  -X\left(  s\right)  \right\vert ^{\beta}\right)  \right]  $ is
finite for all $s,t$, then the function $f$ above will also serve as an
almost-sure uniform modulus of continuity for $X$, provided the logarithmic
correction term in $f$ is raised to the power $1/\beta$ rather than $1/2$.
Details of how this can be established are in the non-Gaussian regularity
theory in \cite{VV}. If $X$ has standard moments of all orders, then one can
replace $f\left(  \varepsilon\right)  $ by $\varepsilon^{1/\left(  2m\right)
-\alpha}$ for any $\alpha>0$. This is easily achieved using Kolmogorov's
continuity criterion. If $X$ only has finitely many moments, Kolmogorov's
continuity criterion can only guarantee that one may take $\alpha$ greater
than some $\alpha_{0}>0$. We do not delve into the details of these regularity
issues in the non-Gaussian martingale case.

\section{Stochastic calculus\label{STOCH}}

In this section, we investigate the possibility of defining the so-called
symmetric stochastic integral and its associated It\^{o} formula for processes
which are not fractional Brownian motion; fBm was treated in \cite{GNRV}. We
concentrate on Gaussian processes under hypotheses similar to those used in
Section \ref{NonHomogGaussSect} (Theorem \ref{nonhomoggauss}).

The basic strategy is to use the results of \cite{GNRV}. Let $X$ be a
stochastic process on $[0,1]$. According to Sections 3 and 4 in \cite{GNRV}
(specifically, according to the proof of part 1 of Theorem 4.4 therein), if
for every bounded measurable function $g$ on $\mathbf{R}$, the limit%
\begin{equation}
\lim_{\varepsilon\rightarrow0}\frac{1}{\varepsilon}\int_{0}^{1}du\left(
X_{u+\varepsilon}-X_{u}\right)  ^{m}g\left(  \frac{X_{u+\varepsilon}+X_{u}}%
{2}\right)  =0\label{ForIto}%
\end{equation}
holds in probability, for both $m=3$ and $m=5$, then for every $t\in
\lbrack0,1]$ and every $f\in C^{6}\left(  \mathbf{R}\right)  $, the
\emph{symmetric} (\textquotedblleft generalized Stratonovich\textquotedblright%
) stochastic integral%
\begin{equation}
\int_{0}^{t}f^{\prime}\left(  X_{u}\right)  d^{\circ}X_u 
=:\lim_{\varepsilon\rightarrow0}\frac{1}{\varepsilon}\int_{0}^{t}du\left(
X_{u+\varepsilon}-X_{u}\right)  \frac{1}{2}\left(  f^{\prime}\left(
X_{u+\varepsilon}\right)  +f^{\prime}\left(  X_{u}\right)  \right)
\label{Stratoint}%
\end{equation}
exists and we have the It\^{o} formula%
\begin{equation}
f\left(  X_{t}\right)  =f\left(  X_{0}\right)  +\int_{0}^{t}f^{\prime}\left(
X_{u}\right)  d^{\circ}X_u.\label{Ito}%
\end{equation}
Our goal is thus to prove (\ref{ForIto}) for a wide class of Gaussian
processes $X$, which will in turn imply the existence of (\ref{Stratoint}) and
the It\^{o} formula (\ref{Ito}).

If $X$ has homogeneous increments in the sense of Section \ref{HomogGaussSect}%
, meaning that $\mathbf{E}\left[  \left(  X_{s}-X_{t}\right)  ^{2}\right]  $
$=$ $\delta^{2}\left(  t-s\right)  $ for some univariate canonical metric
function $\delta$, then by using $g\equiv\mathbf{1}$ and our Theorem
\ref{HomogGauss}, we see that for (\ref{ForIto}) to hold, we must have
$\delta\left(  r\right)  =o\left(  r^{1/6}\right)  $. If one wishes to treat
non-homogeneous cases, we notice that (\ref{ForIto}) for $g\equiv1$ is the
result of our non-homogeneous Theorem \ref{nonhomoggauss}, so it is necessary
to use that theorem's hypotheses, which include the non-homogeneous version of
$\delta\left(  r\right)  =o\left(  r^{1/6}\right)  $. But we will also need
some non-degeneracy conditions in order to apply the quartic linear regression
method of \cite{GNRV}. These are Conditions (i) and (ii) in the next Theorem.
Condition (iii) therein is essentially a consequence of the condition that
$\delta^{2}$ be increasing and concave. These conditions are all further
discussed after the statement of the next theorem and its corollary.

\begin{theorem}
\label{ForItoThm}Let $m\geq3$ be an odd integer. Let $X$ be a Gaussian process
on $[0,1]$ satisfying the hypotheses of Theorem \ref{nonhomoggauss}. This
means in particular that we denote as usual its canonical metric by
$\delta^{2}\left(  s,t\right)  $, and that there exists a univariate
increasing and concave function $\delta^{2}$ such that $\delta\left(
r\right)  =o\left(  r^{1/(2m)}\right)  $ and $\delta^{2}\left(  s,t\right)
\leq\delta^{2}\left(  \left\vert t-s\right\vert \right)  $. Assume that for
$u<v$, the functions $u\mapsto Var\left[  X_{u}\right]  =:Q_{u}$,
$v\mapsto\delta^{2}\left(  u,v\right)  $, and $u\mapsto-\delta^{2}\left(
u,v\right)  $ are increasing and concave. Assume there exist positive
constants $a>1$, $b<1/2$, $c>1/4$, and $c^{\prime}>0$ such that for all
$\varepsilon<u<v\leq1$,

\begin{description}
\item[(i)] $c\delta^{2}\left(  u\right)  \leq Q_{u},$

\item[(ii)] $c^{\prime}\delta^{2}\left(  u\right)  \delta^{2}\left(
v-u\right)  \leq Q_{u}Q_{v}-Q^{2}\left(  u,v\right)  ,$

\item[(iii)]
\begin{equation}
\frac{\delta\left(  au\right)  -\delta\left(  u\right)  }{\left(  a-1\right)
u}<b\frac{\delta\left(  u\right)  }{u}. \label{concaviii}%
\end{equation}

\end{description}

Then for every bounded measurable function $g$ on $\mathbf{R}$,
\[
\lim_{\varepsilon\rightarrow0}\frac{1}{\varepsilon^{2}}\mathbf{E}\left[
\left(  \int_{0}^{1}du\left(  X_{u+\varepsilon}-X_{u}\right)  ^{m}g\left(
\frac{X_{u+\varepsilon}+X_{u}}{2}\right)  \right)  ^{2}\right]  =0.
\]

\end{theorem}

When we apply this theorem to the case $m=3$, the assumption depending on $m$,
namely $\delta\left(  r\right)  =o\left(  r^{1/(2m)}\right)  $ is satisfied a
fortiori for $m=5$ as well, which means that under the assumption
$\delta\left(  r\right)  =o\left(  r^{1/6}\right)  $, the theorem's conclusion
holds for $m=3$ and $m=5$. Therefore, as mentioned in the strategy above, we
immediately get the following.

\begin{corollary}
\label{coroll}Assume the hypotheses of Theorem \ref{ForItoThm} with $m=3$. We
have existence of the symmetric integral in (\ref{Stratoint}), and its It\^{o}
formula (\ref{Ito}), for every $f\in C^{6}\left(  \mathbf{R}\right)  $ and
$t\in\lbrack0,1]$.
\end{corollary}

Before proceeding to the proof of this theorem, we discuss its hypotheses. We
refer to the description at the end of Section \ref{NonHomogGaussSect} for
examples satisfying the hypotheses of Theorem \ref{nonhomoggauss}; these
examples also satisfy the monotonicity and convexity conditions in the above theorem.

Condition (i) is a type of coercivity assumption on the non-degeneracy of
$X$'s variances in comparison to its increments' variances. The hypotheses of
Theorem \ref{nonhomoggauss} imply that $Q_{u}\leq\delta^{2}\left(  u\right)
$, and Condition (i) simply adds that these two quantities should be
commensurate, with a lower bound that it not too small. The "Volterra
convolution"-type class of processes (\ref{volterra}) given at the end of
Section \ref{NonHomogGaussSect}, which includes the Riemann-Liouville fBm's,
satisfies Condition (i) with $c=1/2$. In the homogeneous case, (i) is
trivially satisfied since $Q_{u}\equiv\delta^{2}\left(  u\right)  $.

Condition (ii) is also a type of coercivity condition. It too is satisfied in
the homogeneous case. We prove this claim, since it is not immediately
obvious. In the homogeneous case, since $\delta^{2}\left(  u,v\right)
=\delta^{2}\left(  v-u\right)  =Q_{v-u}$, we calculate%
\[
Q_{u}Q_{v}-Q^{2}\left(  u,v\right)  =Q_{u}Q_{v}-4^{-1}\left(  Q_{u}%
+Q_{v}-Q_{v-u}\right)  ^{2}%
\]
and after rearranging some terms we obtain%
\[
Q_{u}Q_{v}-Q^{2}\left(  u,v\right)  =2^{-1}Q_{v-u}\left(  Q_{u}+Q_{v}\right)
-4^{-1}\left(  Q_{v}-Q_{u}\right)  ^{2}-4^{-1}Q_{v-u}^{2}.
\]
We note first that by the concavity of $Q$, we have $Q_{v}-Q_{u}<Q_{v-u}$, and
consequently, $\left(  Q_{v}-Q_{u}\right)  ^{2}\leq\left(  Q_{v}-Q_{u}\right)
Q_{v-u}\leq Q_{v}Q_{v-u}$. This implies%
\[
Q_{u}Q_{v}-Q^{2}\left(  u,v\right)  \geq2^{-1}Q_{v-u}Q_{u}+4^{-1}\left(
Q_{v-u}Q_{v}-Q_{v-u}^{2}\right)  .
\]
Now by monotonicity of $Q$, we can write $Q_{v-u}Q_{v}\geq Q_{v-u}^{2}$. This,
together with Condition (i), yield Condition (ii) since we now have%
\[
Q_{u}Q_{v}-Q^{2}\left(  u,v\right)  \geq2^{-1}Q_{v-u}Q_{u}\geq2^{-1}%
c^{2}\delta^{2}\left(  v-u\right)  \delta^{2}\left(  u\right)  .
\]

Lastly, Condition (iii) represents a strengthened concavity condition on the
univariate function $\delta$. Indeed, the left-hand side in (\ref{concaviii})
is the slope of the secant of the graph of $\delta$ between the points $u$ and
$au$, while the right-hand side is $b$ times the slope of the secant from $0$
to $u$. If $b$ were allowed to be $1$, (iii) would simply be a consequence of
convexity. Here taking $b\leq1/2$ means that we are exploiting the concavity
of $\delta^{2}$; the fact that condition (iii) requires slightly more, namely
$b$ strictly less than $1/2$, allows us to work similarly to the scale
$\delta\left(  r\right)  =r^{H}$ with $H<1/2$, as opposed to simply asking
$H\leq1/2$. Since the point of the Theorem is to allow continuity moduli which
are arbitrarily close to $r^{1/6}$, Condition (iii) is hardly a
restriction.\bigskip

\textbf{Proof of Theorem \ref{ForItoThm}.}

\bigskip

\noindent\emph{Step 0: setup.} The expectation to be evaluated is written, as
usual, as a double integral over $\left(  u,v\right)  \in\lbrack0,1]^{2}$. For
$\varepsilon>0$ fixed, we define the \textquotedblleft
off-diagonal\textquotedblright\ set
\[
D_{\varepsilon}=\left\{  \left(  u,v\right)  \in\lbrack0,1]^{2}:\varepsilon
^{1-\rho}\leq u\leq v-\varepsilon^{1-\rho}<v\leq1\right\}
\]
where $\rho\in(0,1)$ is fixed. Using the boundedness of $g$ and
Cauchy-Schwarz's inequality, thanks to the hypothesis $\delta\left(  r\right)
=o\left(  r^{1/\left(  2m\right)  }\right)  $, the term corresponding to the
diagonal part (integral over $D_{\varepsilon}^{c}$) can be treated identically
to what was done in \cite{GNRV} in dealing with their term $\mathcal{J}%
^{\prime}\left(  \varepsilon\right)  $ following the statement of their Lemma
5.1, by choosing $\rho$ small enough. It is thus sufficient to prove that
\[
\mathcal{J}\left(  \varepsilon\right)  :=\frac{1}{\varepsilon^{2}}%
\mathbf{E}\left[  \iint_{D_{\varepsilon}}dudv\left(  X_{u+\varepsilon}%
-X_{u}\right)  ^{m}\left(  X_{v+\varepsilon}-X_{v}\right)  ^{m}g\left(
\frac{X_{u+\varepsilon}+X_{u}}{2}\right)  g\left(  \frac{X_{v+\varepsilon
}+X_{v}}{2}\right)  \right]
\]
tends to $0$ as $\varepsilon$ tends to $0$. We now use the same method and
notation as in Step 3 of the proof of Theorem 4.1 in \cite{GNRV}. In order to
avoid repeating arguments from that proof, we only state and prove the new
lemmas which are required.\vspace{0.15in}

\noindent\emph{Step 1: translating Lemma 5.3 from \cite{GNRV}.} Using the fact
that $\mathbf{E}\left[  Z_{\ell}^{2}\right]  \leq\mathbf{E}\left[  G_{\ell
}^{2}\right]  \leq\delta^{2}\left(  \varepsilon\right)  $, this lemma
translates as:

\begin{lemma}
\label{Lemma53}Let $k\geq2$ be an integer. Then
\[
\iint_{D_{\varepsilon}}\mathbf{E}\left[  \left\vert \Gamma_{\ell}\right\vert
^{k}\right]  dudv\leq cst\cdot\varepsilon\delta^{k}\left(  \varepsilon\right)
.
\]

\end{lemma}

This step and the next 4 steps are devoted to the \textbf{Proof of lemma
\ref{Lemma53}. }We only need to show that for all $i,j\in\left\{  1,2\right\}
$,%
\begin{equation}
\iint_{D_{\varepsilon}}\left\vert r_{ij}\right\vert ^{k}dudv\leq
cst\cdot\varepsilon\delta^{k}\left(  \varepsilon\right)  . \label{intrijk}%
\end{equation}
Recall the function $K$ defined in \cite{GNRV}
\begin{align*}
K\left(  u,v\right)   &  :=\mathbf{E}\left[  \left(  X_{u+\varepsilon}%
+X_{u}\right)  \left(  X_{v+\varepsilon}+X_{v}\right)  \right] \\
&  =Q\left(  u+\varepsilon,v+\varepsilon\right)  +Q\left(  u,v+\varepsilon
\right)  +Q\left(  u+\varepsilon,v\right)  +Q\left(  u,v\right)  .
\end{align*}
This is not to be confused with the usage of the letter $K$ in previous
sections, to which there will be made no reference in this proof; the same
remark hold for the notation $\Delta$ borrowed again from \cite{GNRV}, and
used below.

To follow the proof in \cite{GNRV}, we need to prove the following items for
some constants $c_{1}$ and $c_{2}$:

\begin{enumerate}
\item $c_{1}\delta^{2}\left(  u\right)  \leq K\left(  u,u\right)  \leq
c_{2}\delta^{2}\left(  u\right)  ;$

\item $K\left(  u,v\right)  \leq c_{2}\delta\left(  u\right)  \delta\left(
v\right)  ;$

\item $\Delta\left(  u,v\right)  :=K\left(  u,u\right)  K\left(  v,v\right)
-K\left(  u,v\right)  ^{2}\geq c_{1}\delta^{2}\left(  u\right)  \delta
^{2}\left(  v-u\right)  .$
\end{enumerate}

By the Theorem's upper bound assumption on the bivariate $\delta^{2}$
(borrowed from Theorem \ref{nonhomoggauss}), its assumptions on the
monotonicity of $Q$ and the univariate $\delta$, and finally using the
coercivity assumption (i), we have%
\begin{align*}
K\left(  u,u\right)   &  =Q_{u}+Q_{u+\varepsilon}+2Q\left(  u,u+\varepsilon
\right)  =2\left(  Q_{u}+Q_{u+\varepsilon}\right)  -\delta^{2}\left(
u,u+\varepsilon\right) \\
&  \geq2\left(  Q_{u}+Q_{u+\varepsilon}\right)  -\delta^{2}\left(
\varepsilon\right) \\
&  \geq4Q_{u}-\delta^{2}\left(  \varepsilon\right) \\
&  \geq\left(  4-c^{-1}\right)  Q_{u}.
\end{align*}
This proves the lower bound in Item 1 above. The upper bound in Item 1 is a
special case of Item 2, which we now prove. Again, the assumption borrowed
from Theorem \ref{nonhomoggauss}, which says that $\delta^{2}\left(
s,t\right)  \leq\delta^{2}\left(  \left\vert t-s\right\vert \right)  $, now
implies, for $s=0$, that%
\begin{equation}
\delta^{2}\left(  0,u\right)  =Q_{u}\leq\delta^{2}\left(  u\right)  .
\label{Quudelta}%
\end{equation}
We write, via Cauchy-Schwarz's inequality and the fact that $\delta^{2}$ is
increasing, and thanks to (\ref{Quudelta}),%
\[
K\left(  u,v\right)  \leq4\delta\left(  u+\varepsilon\right)  \delta\left(
v+\varepsilon\right)  .
\]
However, since $\delta^{2}$ is concave with $\delta\left(  0\right)  =0$, we
have $\delta^{2}\left(  2u\right)  /2u\leq\delta^{2}\left(  u\right)  /u$.
Also, since we are in the set $D_{\varepsilon}$, $u+\varepsilon\leq2u$ and
$v+\varepsilon\leq2v$. Hence%
\begin{align*}
K\left(  u,v\right)   &  \leq4\delta\left(  2u\right)  \delta\left(  2v\right)
\\
&  \leq8\delta\left(  u\right)  \delta\left(  v\right)  ,
\end{align*}
which is Item 2.

We now verify Item 3 for all $u,v\in D_{\varepsilon}$ , assuming in addition
that $v$ is not too small, specifically $v>\varepsilon^{\rho/2}$. One can
estimate the integral in Lemma \ref{Lemma53} restricted to those values where
$v\leq\varepsilon^{\rho/2}$ using coarser tools than we use below; we omit the
corresponding calculations. From the definition of $K$ above, using the fact
that, by our concavity assumptions, $Q$ is, in both variables, a sum of
Lipschitz functions, we have, for small $\varepsilon$,%
\[
K\left(  u,v\right)  =4Q\left(  u,v\right)  +O\left(  \varepsilon\right)  .
\]
Therefore,
\[
\Delta=16\left(  Q_{u}Q_{v}-Q^{2}\left(  u,v\right)  \right)  +O\left(
\varepsilon\right)  .
\]
Assumption (ii) in the Theorem now implies%
\[
\Delta\geq16c^{\prime}\delta^{2}\left(  u\right)  \delta^{2}\left(
v-u\right)  +O\left(  \varepsilon\right)  .
\]
The concavity of $Q$ and Assumption (i) imply $\delta^{2}\left(  r\right)
\geq Q_{r}\geq cst\cdot r$. Moreover, because of the restriction on $v$,
either $v-u>cst\cdot\varepsilon^{\rho/2}$ or $u>cst\cdot\varepsilon^{\rho/2}$.
Therefore $\delta^{2}\left(  u\right)  \delta^{2}\left(  v-u\right)  \geq
cst\cdot\varepsilon^{1-\rho}\varepsilon^{\rho/2}\gg\varepsilon$. Therefore,
for $\varepsilon$ small enough, $\Delta\geq8c^{\prime}\delta^{2}\left(
u\right)  \delta^{2}\left(  v-u\right)  $, proving Item 3.\bigskip

It will now be necessary to reestimate the components of the matrix
$\Lambda_{21}$ where we recall%
\begin{align*}
\Lambda_{21}[11]  &  :=\mathbf{E}\left[  \left(  X_{u+\varepsilon}%
+X_{u}\right)  \left(  X_{u+\varepsilon}-X_{u}\right)  \right]  ,\\
\Lambda_{21}[12]  &  :=\mathbf{E}\left[  \left(  X_{v+\varepsilon}%
+X_{v}\right)  \left(  X_{u+\varepsilon}-X_{u}\right)  \right]  ,\\
\Lambda_{21}[21]  &  :=\mathbf{E}\left[  \left(  X_{u+\varepsilon}%
+X_{u}\right)  \left(  X_{v+\varepsilon}-X_{v}\right)  \right]  ,\\
\Lambda_{21}[22]  &  :=\mathbf{E}\left[  \left(  X_{v+\varepsilon}%
+X_{v}\right)  \left(  X_{v+\varepsilon}-X_{v}\right)  \right]  .
\end{align*}

\emph{Step 2: the term }$r_{11}$. We have by the lower bound of item 1 above
on $K\left(  u,u\right)  $,%
\[
\left\vert r_{11}\right\vert =\left\vert \frac{1}{\sqrt{K\left(  u,u\right)
}}\Lambda_{21}[11]\right\vert \leq\frac{cst}{\delta\left(  u\right)
}\left\vert \Lambda_{21}[11]\right\vert .
\]
To bound $\left\vert \Lambda_{21}[11]\right\vert $ above, we write%
\begin{align*}
\left\vert \Lambda_{21}[11]\right\vert  &  =\left\vert \mathbf{E}\left[
\left(  X_{u+\varepsilon}+X_{u}\right)  \left(  X_{u+\varepsilon}%
-X_{u}\right)  \right]  \right\vert \\
&  =Q_{u+\varepsilon}-Q_{u}\\
&  \leq\varepsilon Q\left(  u\right)  /u\\
&  \leq\varepsilon\delta^{2}\left(  u\right)  /u
\end{align*}
where we used the facts that $Q_{u}$ is increasing and concave, and that
$Q_{u}\leq\delta^{2}\left(  u\right)  $. Thus we have%
\[
\left\vert r_{11}\right\vert \leq\varepsilon~cst\frac{\delta\left(  u\right)
}{u}.
\]
The result (\ref{intrijk}) for $i=j=1$ now follows by the next lemma.

\begin{lemma}
\label{deltauuk}For every $k\geq2$, there exists $c_{k}>0$ such that for every
$\varepsilon\in(0,1)$,%
\[
\int_{\varepsilon}^{1}\left\vert \frac{\delta\left(  u\right)  }{u}\right\vert
^{k}du\leq c_{k}\varepsilon\left\vert \frac{\delta\left(  \varepsilon\right)
}{\varepsilon}\right\vert ^{k}.
\]

\end{lemma}

\textbf{Proof of lemma \ref{deltauuk}.} Our hypothesis (iii) can be rewritten
as%
\begin{align*}
\frac{\delta\left(  au\right)  }{au}  &  <\left(  \frac{1+\left(  a-1\right)
b}{a}\right)  \frac{\delta\left(  u\right)  }{u}\\
&  =:K_{a,b}\frac{\delta\left(  u\right)  }{u}.
\end{align*}
The concavity of $\delta$ also implies that $\delta\left(  u\right)  /u$ is
increasing. Thus we can write%
\begin{align*}
\int_{\varepsilon}^{1}\left\vert \frac{\delta\left(  u\right)  }{u}\right\vert
^{k}du  &  \leq\sum_{j=0}^{\infty}\int_{\varepsilon a^{j}}^{\varepsilon
a^{j+1}}\left\vert \frac{\delta\left(  u\right)  }{u}\right\vert ^{k}du\\
&  \leq\sum_{j=0}^{\infty}\left(  \varepsilon a^{j+1}-\varepsilon
a^{j}\right)  \left\vert K_{a,b}\right\vert ^{jk}\left\vert \frac
{\delta\left(  \varepsilon\right)  }{\varepsilon}\right\vert ^{k}\\
&  =\varepsilon\left(  a-1\right)  \left\vert \frac{\delta\left(
\varepsilon\right)  }{\varepsilon}\right\vert ^{k}\sum_{j=0}^{\infty}\left(
\left\vert K_{a,b}\right\vert ^{k}a\right)  ^{j}.
\end{align*}
The lemma will be proved if we can show that $f\left(  a\right)  :=\left\vert
K_{a,b}\right\vert ^{k}a<1$ for some $a>1$. We have $f\left(  1\right)  =0$
and $f^{\prime}\left(  1\right)  =k\left(  1-b\right)  -1$. This last quantity
is strictly positive for all $k\geq2$ as soon as $b<1/2$. This finishes the
proof of the lemma \ref{deltauuk}.\hfill$\square$\vspace{0.15in}

\emph{Step 3: the term }$r_{12}$. We have
\[
r_{12}=\Lambda_{21}\left[  11\right]  \frac{-K\left(  u,v\right)  }%
{\sqrt{K\left(  u,u\right)  \Delta\left(  u,v\right)  }}+\Lambda_{21}\left[
12\right]  \frac{\sqrt{K\left(  u,u\right)  }}{\sqrt{\Delta\left(  u,v\right)
}}.
\]
We saw in the previous step that $\left\vert \Lambda_{21}\left[  11\right]
\right\vert =\left\vert Q_{u+\varepsilon}-Q_{u}\right\vert \leq cst\cdot
\varepsilon\delta^{2}\left(  u\right)  /u$. For $\Lambda_{21}\left[
12\right]  $, using the hypotheses on our increasing and concave functions, we
calculate%
\begin{align}
\left\vert \Lambda_{21}\left[  12\right]  \right\vert  &  =\left\vert 2\left(
Q_{u+\varepsilon}-Q_{u}\right)  +\delta^{2}\left(  u+\varepsilon
,v+\varepsilon\right)  -\delta^{2}\left(  u,v+\varepsilon\right)  +\delta
^{2}\left(  u+\varepsilon,v\right)  -\delta^{2}\left(  u,v\right)  \right\vert
\nonumber\\
&  \leq2\left\vert \Lambda_{21}\left[  11\right]  \right\vert +\varepsilon
\delta^{2}\left(  u+\varepsilon,v+\varepsilon\right)  /\left(  v-u\right)
+\varepsilon\delta^{2}\left(  u+\varepsilon,v\right)  /\left(  v-u-\varepsilon
\right) \nonumber\\
&  \leq2\left\vert \Lambda_{21}\left[  11\right]  \right\vert +\varepsilon
\delta^{2}\left(  v-u\right)  /\left(  v-u\right)  +\varepsilon\delta
^{2}\left(  v-u-\varepsilon\right)  /\left(  v-u-\varepsilon\right)
\nonumber\\
&  \leq2cst\cdot\varepsilon\delta^{2}\left(  u\right)  /u+2\varepsilon
\delta^{2}\left(  v-u-\varepsilon\right)  /\left(  v-u-\varepsilon\right)  .
\label{Lambda12}%
\end{align}
The presence of the term $-\varepsilon$ in the last expression above is
slightly aggravating, and one would like to dispose of it. However, since
$\left(  u,v\right)  \in D_{\varepsilon}$, we have $v-u>\varepsilon^{\rho}$
for some $\rho\in\left(  0,1\right)  $. Therefore $v-u-\varepsilon
>\varepsilon^{\rho}-\varepsilon>\varepsilon^{\rho/2}$ for $\varepsilon$ small
enough. Hence by using $\rho/2$ instead of $\rho$ in the definition of
$D_{\varepsilon}$ in the current calculation, we can ignore the term
$-\varepsilon$ in the last displayed line above. Together with items 1, 2, and
3 above which enable us to control the terms $K$ and $\Delta$ in $r_{12}$, we
now have%
\begin{align*}
\left\vert r_{12}\right\vert  &  \leq cst\cdot\varepsilon\frac{\delta
^{2}\left(  u\right)  }{u}\left(  \frac{\delta\left(  u\right)  \delta\left(
v\right)  }{\delta\left(  u\right)  \delta\left(  u\right)  \delta\left(
v-u\right)  }+\frac{\delta\left(  u\right)  }{\delta\left(  u\right)
\delta\left(  v-u\right)  }\right) \\
&  +cst\cdot\varepsilon\frac{\delta^{2}\left(  v-u\right)  }{v-u}\frac
{\delta\left(  u\right)  }{\delta\left(  u\right)  \delta\left(  v-u\right)
}\\
&  =cst\cdot\varepsilon~\left(  \frac{\delta\left(  u\right)  \delta\left(
v\right)  }{u\delta\left(  v-u\right)  }+\frac{\delta^{2}\left(  u\right)
}{u\delta\left(  v-u\right)  }+\frac{\delta\left(  v-u\right)  }{v-u}\right)
.
\end{align*}
We may thus write%
\[
\iint_{D_{\varepsilon}}\left\vert r_{12}\right\vert ^{k}dudv\leq
cst\cdot\varepsilon^{k}~\iint_{D_{\varepsilon}}\left(  \left\vert \frac
{\delta\left(  u\right)  \delta\left(  v\right)  }{u\delta\left(  v-u\right)
}\right\vert ^{k}+\left\vert \frac{\delta^{2}\left(  u\right)  }%
{u\delta\left(  v-u\right)  }\right\vert ^{k}+\left\vert \frac{\delta\left(
v-u\right)  }{v-u}\right\vert ^{k}\right)  dudv.
\]
The last term $\iint_{D_{\varepsilon}}\left\vert \frac{\delta\left(
v-u\right)  }{v-u}\right\vert ^{k}dudv$ is identical, after a trivial change
of variables, to the one dealt with in Step 2. Since $\delta$ is increasing,
second the term $\iint_{D_{\varepsilon}}\left\vert \frac{\delta^{2}\left(
u\right)  }{u\delta\left(  v-u\right)  }\right\vert ^{k}dudv$ is smaller than
the first term $\iint_{D_{\varepsilon}}\left\vert \frac{\delta\left(
u\right)  \delta\left(  v\right)  }{u\delta\left(  v-u\right)  }\right\vert
^{k}dudv$. Thus we only need to deal with that first term; it is more delicate
than what we estimated in Step 2.

We separate the integral over $u$ at the intermediate point $v/2$. When
$u\in\lbrack v/2,v-\varepsilon]$, we use the estimate%
\[
\frac{\delta\left(  u\right)  }{u}\leq\frac{\delta\left(  v/2\right)  }%
{v/2}\leq2\frac{\delta\left(  v\right)  }{v}.
\]
On the other hand when $u\in\lbrack\varepsilon,v/2]$ we simply bound
$1/\delta\left(  v-u\right)  $ by $1/\delta\left(  v/2\right)  $. Thus%
\begin{align*}
&  \iint_{D_{\varepsilon}}\left\vert \frac{\delta\left(  u\right)
\delta\left(  v\right)  }{u\delta\left(  v-u\right)  }\right\vert ^{k}dudv\\
&  =\int_{v=2\varepsilon}^{1}dv\int_{u=\varepsilon}^{v/2}\left\vert
\frac{\delta\left(  u\right)  \delta\left(  v\right)  }{u\delta\left(
v-u\right)  }\right\vert ^{k}du+\int_{v=\varepsilon}^{1}dv\int_{u=v/2}%
^{v-\varepsilon}\left\vert \frac{\delta\left(  u\right)  \delta\left(
v\right)  }{u\delta\left(  v-u\right)  }\right\vert ^{k}du\\
&  \leq\int_{v=2\varepsilon}^{1}dv\left\vert \frac{\delta\left(  v\right)
}{\delta\left(  v/2\right)  }\right\vert ^{k}\int_{u=\varepsilon}%
^{v/2}\left\vert \frac{\delta\left(  u\right)  }{u}\right\vert ^{k}%
du+2\int_{v=\varepsilon}^{1}\left\vert \frac{\delta^{2}\left(  v\right)  }%
{v}\right\vert ^{k}dv\int_{u=v/2}^{v-\varepsilon}\left\vert \frac{1}%
{\delta\left(  v-u\right)  }\right\vert ^{k}du\\
&  \leq2^{k}\int_{u=\varepsilon}^{1}\left\vert \frac{\delta\left(  u\right)
}{u}\right\vert ^{k}du+2\frac{1}{\delta^{k}\left(  \varepsilon\right)  }%
\int_{v=\varepsilon}^{1}v^{k}\left\vert \frac{\delta\left(  v\right)  }%
{v}\right\vert ^{2k}dv\\
&  \leq cst\cdot\varepsilon\left(  \frac{\delta\left(  \varepsilon\right)
}{\varepsilon}\right)  ^{k};
\end{align*}
here we used the concavity of $\delta$ to imply that $\delta\left(  v\right)
/\delta\left(  v/2\right)  \leq2$, and to obtain the last line, we used Lemma
\ref{deltauuk} for the first term in the previous line, and we used the fact
that $\delta$ is increasing and that $v\leq1$, together again with Lemma
\ref{deltauuk} for the second term in the previous line. This finishes the
proof of (\ref{intrijk}) for $r_{12}.$\vspace{0.15in}

\emph{Step 4: the term }$r_{21}$. We have%
\[
r_{21}=\Lambda_{21}\left[  21\right]  \frac{1}{\sqrt{K\left(  u,u\right)  }}%
\]
and similarly to the previous step,%
\begin{align*}
\left\vert \Lambda_{21}\left[  21\right]  \right\vert  &  =\left\vert Q\left(
u+\varepsilon,v+\varepsilon\right)  -Q\left(  u+\varepsilon,v\right)
+Q\left(  u,v+\varepsilon\right)  -Q\left(  u,v\right)  \right\vert \\
&  =\left\vert 2\left(  Q_{v+\varepsilon}-Q_{v}\right)  +\delta^{2}\left(
u+\varepsilon,v\right)  -\delta^{2}\left(  u+\varepsilon,v+\varepsilon\right)
+\delta^{2}\left(  u,v\right)  -\delta^{2}\left(  u,v+\varepsilon\right)
\right\vert \\
&  \leq2\left\vert \Lambda_{21}\left[  11\right]  \right\vert +\varepsilon
\frac{\delta^{2}\left(  u+\varepsilon,v\right)  }{v-u-\varepsilon}%
+\varepsilon\frac{\delta^{2}\left(  u,v\right)  }{v-u}\\
&  \leq2cst\cdot\varepsilon\delta^{2}\left(  u\right)  /u+4\varepsilon
\delta^{2}\left(  v-u\right)  /\left(  v-u\right)  ,
\end{align*}
which is the same expression as in (\ref{Lambda12}). Hence with the lower
bound of Item 1 on $K\left(  u,u\right)  $ we have%
\begin{align*}
\iint_{D_{\varepsilon}}\left\vert r_{21}\right\vert ^{k}dudv  &  \leq
cst\cdot\varepsilon^{k}~\iint_{D_{\varepsilon}}\left(  \left\vert \frac
{\delta\left(  u\right)  }{u}\right\vert ^{k}+\left\vert \frac{\delta
^{2}\left(  v-u\right)  }{\left(  v-u\right)  \delta\left(  u\right)
}\right\vert ^{k}\right)  dudv\\
&  =cst\cdot\varepsilon^{k}~\iint_{D_{\varepsilon}}\left(  \left\vert
\frac{\delta\left(  u\right)  }{u}\right\vert ^{k}+\left\vert \frac{\delta
^{2}\left(  u\right)  }{u\delta\left(  v-u\right)  }\right\vert ^{k}\right)
dudv.
\end{align*}
This is bounded above by the expression obtained as an upper bound in Step 3
for $\iint_{D_{\varepsilon}}\left\vert r_{12}\right\vert ^{k}dudv$, which
finishes the proof of (\ref{intrijk}) for $r_{21}.$\vspace{0.15in}

\emph{Step 5: the term }$r_{22}$. Here we have%
\[
r_{22}=\Lambda_{21}\left[  21\right]  \frac{-K\left(  u,v\right)  }%
{\sqrt{K\left(  u,u\right)  \Delta\left(  u,v\right)  }}+\Lambda_{21}\left[
22\right]  \frac{\sqrt{K\left(  u,u\right)  }}{\sqrt{\Delta\left(  u,v\right)
}}.
\]
We have already seen in the previous step that
\[
\left\vert \Lambda_{21}\left[  21\right]  \right\vert \leq cst\cdot
\varepsilon\left(  \frac{\delta^{2}\left(  u\right)  }{u}+\frac{\delta
^{2}\left(  v-u\right)  }{v-u}\right)  .
\]
Moreover, we have, as in Step 2,%
\[
\left\vert \Lambda_{21}\left[  22\right]  \right\vert =\left\vert
Q_{v+\varepsilon}-Q_{v}\right\vert \leq cst\cdot\varepsilon\frac{\delta
^{2}\left(  v\right)  }{v}.
\]
Thus using the bounds in items 1, 2, and 3,%
\begin{align*}
\left\vert r_{22}\right\vert  &  \leq cst\cdot\varepsilon\left[  \left(
\frac{\delta^{2}\left(  u\right)  }{u}+\frac{\delta^{2}\left(  v-u\right)
}{v-u}\right)  \frac{\delta\left(  u\right)  \delta\left(  v\right)  }%
{\delta^{2}\left(  u\right)  \delta\left(  v-u\right)  }+\frac{\delta
^{2}\left(  v\right)  }{v}\frac{\delta\left(  u\right)  }{\delta\left(
u\right)  \delta\left(  v-u\right)  }\right] \\
&  =cst\cdot\varepsilon\left[  \frac{\delta\left(  u\right)  \delta\left(
v\right)  }{u\delta\left(  v-u\right)  }+\frac{\delta\left(  v\right)
\delta\left(  v-u\right)  }{\delta\left(  u\right)  \left(  v-u\right)
}+\frac{\delta^{2}\left(  v\right)  }{v\delta\left(  v-u\right)  }\right]  .
\end{align*}
Of the last three terms, the first term was already treated in Step 3, the
second is, up to a change of variable, identical to the first, and the third
is smaller than $\frac{\delta^{2}\left(  u\right)  }{u\delta\left(
v-u\right)  }$ which was also treated in Step 3. Thus (\ref{intrijk}) is
proved for $r_{22}$, which finishes the entire proof of Lemma \ref{Lemma53}%
.\hfill$\square$\vspace{0.15in}

\noindent\emph{Step 6: translating Lemma 5.4 from \cite{GNRV}.} We will prove
the following result

\begin{lemma}
\label{Lemma54}For all $j\in\{0,1,\cdots,\left(  m-1\right)  /2\}$,%
\[
\iint_{D_{\varepsilon}}\left\vert \mathbf{E}\left[  Z_{3}Z_{4}\right]
\right\vert ^{m-2j}dudv\leq cst\cdot\varepsilon\delta^{2\left(  m-2j\right)
}\left(  \varepsilon\right)  .
\]

\end{lemma}

\textbf{Proof of Lemma \ref{Lemma54}.} As in \cite{GNRV}, we have%
\[
\left\vert \mathbf{E}\left[  Z_{3}Z_{4}\right]  \right\vert ^{m-2j}\leq
cst\cdot\left\vert \mathbf{E}\left[  G_{3}G_{4}\right]  \right\vert
^{m-2j}+cst\cdot\left\vert \mathbf{E}\left[  \Gamma_{3}\Gamma_{4}\right]
\right\vert ^{m-2j}.
\]
The required estimate for the term corresponding to $\left\vert \mathbf{E}%
\left[  \Gamma_{3}\Gamma_{4}\right]  \right\vert ^{m-2j}$ follows by
Cauchy-Schwarz's inequality and Lemma \ref{Lemma53}. For the term
corresponding to $\left\vert \mathbf{E}\left[  G_{3}G_{4}\right]  \right\vert
^{m-2j}$, we recognize that $\mathbf{E}\left[  G_{3}G_{4}\right]  $ is the
negative planar increment $\Theta^{\varepsilon}\left(  u,v\right)  $ defined
in (\ref{defDelta}). Thus the corresponding term was already considered in the
proof of Theorem (\ref{nonhomoggauss}). More specifically, up to the factor
$\varepsilon^{2}\delta^{-4j}\left(  \varepsilon\right)  $, we now have to
estimated the same integral as in Step 2 of that theorem's proof: see
expression (\ref{JjOD}) for the term we called $J_{j,OD}$. This means that%
\begin{align*}
\iint_{D_{\varepsilon}}\left\vert \mathbf{E}\left[  G_{3}G_{4}\right]
\right\vert ^{m-2j}dudv  &  \leq\frac{\varepsilon^{2}}{\delta^{4j}\left(
\varepsilon\right)  }J_{j,OD}\\
&  \leq\varepsilon^{2}\left\vert \mu\right\vert \left(  OD\right)
\delta^{2\left(  m-2j-1\right)  }\left(  \varepsilon\right)  .
\end{align*}
Our hypotheses borrowed from Theorem (\ref{nonhomoggauss}) that $\left\vert
\mu\right\vert \left(  OD\right)  \leq cst\cdot\varepsilon^{1/m-1}$ and that
$\delta^{2}\left(  \varepsilon\right)  =o\left(  r^{1/\left(  2m\right)
}\right)  $ now imply that the above is $\ll\varepsilon\delta^{2\left(
m-2j\right)  }\left(  \varepsilon\right)  $, concluding the lemma's
proof.\hfill$\square$\vspace{0.15in}

\noindent\emph{Step 7. Conclusion}. The remainder of the proof of the theorem
is to check that Lemmas \ref{Lemma53} and \ref{Lemma54} do imply the claim of
the theorem; this is done exactly as in Steps 3 and 4 of the proof of Theorem
4.1 in \cite{GNRV}. Since such a task is only bookkeeping, we omit it,
concluding the proof of Theorem \ref{ForItoThm}.\hfill$\blacksquare$

\begin{center}
\textbf{Acknowledgements}
\end{center}

\noindent The work of F. Russo was partially supported by the ANR Project
MASTERIE 2010 BLAN-0121-01.  Part of it was done during a stay of 
this author at the Bernoulli Center of the EPF Lausanne.
The work of F. Viens is partially supported by NSF
DMS grant 0907321. 

\end{document}